\documentclass{scrartcl}

\usepackage{amsmath,amsfonts,amsthm}%
\usepackage[linesnumbered,vlined,boxed]{algorithm2e}%
\let\oldnl\nl
\newcommand{\nonl}{\renewcommand{\nl}{\let\nl\oldnl}}

\usepackage{enumerate}%
\usepackage{tikz}%
\usepackage{pgfplots}%

\usepackage[numbers,sort&compress]{natbib}

\usepackage{hyperref}%
\hypersetup{
  colorlinks=true,
  breaklinks=true,
  urlcolor= blue,
  linkcolor= blue,
  bookmarksopen=true,
  pdftitle={Irreducibility test},
  pdfauthor={Martin Weimann},
}

\newcommand{\hbb}[1]{\ensuremath{\mathbb{#1}}}
\newcommand{\Ai}{\hbb{A}}
\newcommand{\Ci}{\hbb{C}}
\newcommand{\Fi}{\hbb{F}}
\newcommand{\Ki}{\hbb{K}}
\newcommand{\Ni}{\hbb{N}}
\newcommand{\Qi}{\hbb{Q}}
\newcommand{\Ri}{\hbb{R}}

\newcommand{\Li}{\hbb{L}}

\newcommand{\hcal}[1]{\ensuremath{\mathcal{#1}}}

\newcommand{\Bc}{\hcal{B}}

\newcommand{\Nc}{\hcal{N}}
\newcommand{\Oc}{\hcal{O}}

\newcommand{\Ut}{\tilde U}
\newcommand{\At}{\tilde A}
\newcommand{\Gt}{\tilde G}
\newcommand{\Ht}{\tilde H}
\newcommand{\St}{\tilde S}
\newcommand{\Tt}{\tilde T}

\theoremstyle{plain}
\newtheorem{thm}{Theorem}
\newtheorem{prop}{Proposition}
\newtheorem{cor}{Corollary}
\newtheorem{lem}{Lemma}
\theoremstyle{definition}
\newtheorem{dfn}{Definition}

\newtheorem{xmp}{Example}
\theoremstyle{remark}
\newtheorem{rem}{Remark}

\newcommand{\RNP}{\texttt{RNP}}
\newcommand{\Edata}{\texttt{EdgeData}}
\newcommand{\abhyankar}{\texttt{Abhyankar}}
\newcommand{\AbhyankarMoh}{\texttt{AbhyankarTest}}
\newcommand{\irreducible}{\texttt{Irreducible}}
\newcommand{\Edgepoly}{\texttt{BoundaryPol}}
\newcommand{\NPA}{\texttt{ARNP}}
\newcommand{\PNPA}{\texttt{Pseudo-ARNP}}
\newcommand{\PIrr}{\texttt{Pseudo-Irreducible}}
\newcommand{\PDegenerated}{\texttt{Pseudo-Degenerated}}
\newcommand{\AppRoot}{\texttt{AppRoot}}
\newcommand{\Expand}{\texttt{Expand}}
\newcommand{\Primitive}{\texttt{Primitive}}
\newcommand{\quorem}{\textnormal{\texttt{QuoRem}}}
\newcommand{\onestep}{\textnormal{\texttt{HenselStep}}}

\newcommand{\assign}{{\;\leftarrow{}\;}}%

\renewcommand{\O}{\textrm{\Oc}}
\newcommand{\Ot}{\O\tilde\,\,}
\newcommand{\M}{\textup{\textsf{M}}}
\newcommand{\I}{\textup{\textsf{I}}}
\newcommand{\dy}{{d}}
\newcommand{\dx}{{n}}

\newcommand{\val}[1][x]{\ensuremath{v_{#1}}}
\newcommand{\D}{\textrm{Data}}
\newcommand{\C}{\textrm{C}}
\newcommand{\Cont}{\textrm{Cont}}
\newcommand{\Res}{\textrm{Res}}

\newcommand{\Card}{\textrm{Card}}
\newcommand{\Conv}{\textrm{Conv}}

\newcommand{\vF}[1][]{{\ensuremath{\delta_{#1}}}}

\newcommand{\edgepoly}{boundary polynomial}
\newcommand{\True}{\texttt{True}}
\newcommand{\False}{\texttt{False}}
\newcommand{\Char}{\textrm{Char}}

\newcommand{\NP}{\Nc}
\newcommand{\NPb}{\NP^-}
\newcommand{\algclos}[1]{\overline{#1}}
\newcommand{\tronc}[2]{{\left\lceil #1 \right\rceil}^{#2}}
\newcommand{\tc}[2][y]{\mbox{tc}_{#1}\left(#2\right)}
\newcommand{\Z}{\underline Z}

\begin{document}

\title{Using approximate roots for irreducibility and equi-singularity issues in $\Ki[[x]][y]$}

\author{%
  Adrien POTEAUX,\\%
  {CRIStAL-INRIA}\\%
  { Universit\'e de Lille}\\%
  {UMR CNRS 9189, B\^atiment M3}\\%
  {59655 Villeneuve d'Ascq, France}\\%
  \texttt{adrien.poteaux@univ-lille.fr}
  \vspace{3mm}
  \and Martin WEIMANN,\\%
  {GAATI\footnote{Current delegation. Permanent position at LMNO,
      University of Caen-Normandie, BP 5186, 14032 Caen Cedex,
      France.}}\\%
      {Universit\'e de Polyn\'esie Fran\c{c}aise}\\%
  {BP 6570, 98702 Faa'a}\\%
  \texttt{martin.weimann@upf.pf}%
}


\maketitle

\begin{abstract}
  We provide an irreducibility test in the ring $\Ki[[x]][y]$ whose
  complexity is quasi-linear with respect to the discriminant
  valuation, assuming the input polynomial $F$ square-free and $\Ki$ a perfect field of
  characteristic zero or greater than $\deg(F)$. The algorithm uses
  the theory of approximate roots and may be seen as a generalization
  of Abhyankhar's irreducibility criterion to the case of non algebraically
  closed residue fields. More generally, we show that we can test
  within the same complexity if a polynomial is pseudo-irreducible, a
  larger class of polynomials containing irreducible ones.  If $F$ is
  pseudo-irreducible, the algorithm computes also the
  discriminant valuation of $F$ and the equisingularity classes of the germs of plane curves defined by $F$ along the fiber $x=0$.
\end{abstract}

\section{Introduction}

\paragraph{Context and main result.} This paper provides new
complexity results for testing the irreducibility of polynomials with
coefficients in a ring of formal power series of characteristic zero
or big enough. We consider $\Ki$ a perfect field, $x$ and $y$ two
indeterminates over $\Ki$ and $F\in\Ki[[x]][y]$ a polynomial of degree
$\dy$. In all of the sequel, we will assume that the following
hypothesis holds:
\begin{center}
  \emph{The characteristic of $\Ki$ is either $0$ or greater than $\dy$.}
\end{center}
Assuming $F$ square-free, we let $\vF$ be the $x$-valuation of the
resultant between $F$ and its $y$-derivative
$F_y$. We prove: 
\begin{thm}\label{thm:main}
  There exists an algorithm which tests if $F$ is irreducible in $\Ki[[x]][y]$ with an expected
  $\Ot(\vF+\dy)$ operations over $\Ki$ and two univariate
irreducibility tests over $\Ki$ of degree at most $\dy$. 
\end{thm}
If $F$ is  Weierstrass\footnote{We recall that in our
  context, $F=\sum_{i=0}^\dy a_i(x)\,y^i$ is Weierstrass if $a_\dy=1$
  and $a_i(0)=0$ for $i<\dy$}, the complexity drops to $\Ot(\vF)$ operations over $\Ki$ and one univariate
irreducibility test of degree at most $\dy$. The notation $\Ot()$ hides logarithmic
factors. Our algorithm is Las Vegas, due to the computation of
primitive elements; it should become deterministic via the preprint
\cite{HoLe18}. See Section \ref{sec:comp} for more details, including
our complexity model.

We say that $F$ is absolutely irreducible if it is irreducible in
$\algclos{\Ki}[[x]][y]$, where $\algclos\Ki$ stands for the algebraic
closure of $\Ki$. In such a case, we avoid univariate irreducibility
tests and there is no need to deal with extensions of residue
fields. We get:
\begin{thm}\label{thm:absolute}
  There exists a deterministic algorithm which tests if $F$ is absolutely irreducible
  with $\Ot(\vF+d)$ operations over $\Ki$, which is $\Ot(\vF)$ when $F$
  is Weierstrass.
\end{thm}
\paragraph{Pseudo-Irreducible polynomials.} If $F$ is irreducible, the
algorithms above compute also the discriminant valuation $\vF$ and the
number of absolutely irreducible factors together with their sets of
characteristic exponents and pairwise intersection
multiplicities. These numerical data capture the main relevant
information about the singularities of the germs of plane curves
defined by $F$ along $x=0$.  In particular, they uniquely determine
their equisingularity classes, hence their topological classes if $\Ki=\Ci$. It turns out that we can
compute these invariants within the same complexity (avoiding
furthermore any univariate irreducibility test) for a larger class of
polynomial: we say that $F$ is \emph{pseudo-irreducible} (the terminology \emph{balanced} will also be used in the sequel) if its irreducible
factors in $\algclos{\Ki}[[x]][y]$ have same characteristic exponents
and same sets of pairwise intersection multiplicities (see Section
\ref{sec:equising}). If $F$ is irreducible in $\Li[[x]][y]$ for some
field extension $\Li$ of $\Ki$, then it is pseudo-irreducible by a Galois
argument, but the converse does not hold.

\begin{thm}\label{thm:main2}
  There exists an algorithm which tests if $F$ is pseudo-irreducible with an expected
  $\Ot(\dy+\vF)$ operations over $\Ki$, which is $\Ot(\vF)$ if $F$ is
  Weierstrass. If $F$ is pseudo-irreducible, the algorithm computes 
  $\vF$ and the number of irreducible factors in
  $\algclos{\Ki}[[x]][y]$ together with their characteristic exponents
  and pairwise intersection multiplicities.
\end{thm}

Note that if $F$ is pseudo-irreducible, all its absolutely irreducible
factors have same degree, and the algorithm computes it. We can
compute also the degrees, residual degrees and ramifications indices of the irreducible factors of $F$ in
$\Li[[x]][y]$ over any given field extension $\Li$ of $\Ki$ by
performing an extra univariate factorization of degree $d$ over $\Li$.

\paragraph{Bivariate case.} If $F\in \Ki[x,y]$ is a square-free
bivariate polynomial of bidegree $(\dx,\dy)$, we have
$\vF\le 2 \dx \dy-\dx$, hence our algorithms are quasi-linear with respect
to the arithmetic size $nd$ of $F$. In fact, we can avoid the square-free
hypothesis in this case:
\begin{thm}\label{thm:nonsqrfree}
  If $F\in \Ki[x,y]$, then the previous irreducibility or
  pseudo-irreducibility tests have complexity $\Ot(\dx \dy)$ up to
  univariate irreducibility tests, and so without assuming
  square-freeness of $F$.
\end{thm}
Note that this does not mean that we can check square-freeness of $F$
within $\Ot(\dx \dy)$ operations (this costs $\Ot(\dx \dy^2)$
operations with usual algorithms). Also, note that there is no hope to
test irreducibility of a non square-free polynomial
$F\in \Ki[[x]][y]$, as this would require to deal with an infinite
precision. 

\paragraph{Local case.} Our algorithms provide also (pseudo)-irreducibility tests in the local rings $\Ki[[x,y]]$ or
$\overline{\Ki}[[x,y]]$. To this aim, we first apply the Weierstrass
Preparation Theorem and compute a factorization $F=UH$ up to a
suitable precision using a Hensel like strategy, with
$H\in \Ki[[x]][y]$ a Weierstrass polynomial and $U$ a unit in
$\Ki[[x,y]]$, and we eventually check the (pseudo)-irreducibility of $H$ using
algorithms above. Unfortunately, if $F$ is non Weierstrass, the
computation of $H$ up to a suitable precision is $\Omega(\dy\vF)$ in
the worst case scenario (\cite[Example 5]{PoWe17} provides an
explicit family of polynomials $F\in \Ki[x,y]$ for which a local
irreducibility test in $\Ki[[x,y]]$ is cubic in the total degree).

\paragraph{Main ideas.} All algorithms are based on the same idea. We
recursively compute some well chosen approximate roots
$\psi_0,\ldots,\psi_g$ of $F$. At each step, we compute the
$(\psi_0,\ldots,\psi_k)$-adic expansion of $F$. We deduce the $k$-th
generalised Newton polygon and check if it is straigth. If so,
we compute the related \edgepoly{} and test if it is the power of some
irreducible polynomial. In such a case, we deduce the degree of the
next approximate root $\psi_{k+1}$ that has to be computed.  The
algorithm gives moreover the characteristic exponents of $F$, and so
without performing any blow-ups and liftings inherent to the classical
Newton-Puiseux algorithm. Such a strategy was developped by Abhyankar
for testing irreducibility in $\Ci[[x]][y]$ in \cite{Ab89}. A major
difference here is that testing irreducibility for non algebraically
closed residue field $\Ki$ requires to compute also the \edgepoly{}s,
a key point which is not an issue in Abhyankhar's algorithm.  Also, in
order to perform a unique univariate irreducibility test over $\Ki$,
we rely on dynamic evaluation and rather check if the \edgepoly{}s are
powers of a square-free polynomial. The pseudo-irreducibility test is based on such a modification, allowing moreover several edges of the Newton polygon in some particular cases.

\paragraph{Related results.} Factorization in $\Ki[[x]][y]$ (and
\textit{a fortiori} irreducibility test) is an important issue in the
algorithmic of algebraic curves, both for local aspects (studying
plane curves singularities) and for global aspects (e.g. computing
integral basis of function fields \cite{vH94}, computing the geometric
genus \cite{PoWe17}, factoring polynomials in $\Ki[x,y]$ taking
advantage of critical fibers \cite{We16}, etc). Probably the most
classical approach for factoring polynomials in $\Ki[[x]][y]$ is
derived from the Newton-Puiseux algorithm, as a combination of
blow-ups (monomial transforms and shifts) and Hensel liftings. This
approach allows moreover to compute the roots of $F$ - represented as
fractional Puiseux series - up to an arbitrary precision. The
Newton-Puiseux algorithm has been studied by many authors (see
e.g. \cite{Du89, DeDiDu85, Wa00, Te90, Po08, PoRy11, PoRy12, PoRy15,
  PoWe17} and the references therein). Up to our knowledge, the best
current arithmetic complexity was obtained in \cite{PoWe17}, using a
divide and conquer strategy leading to a fast Newton-Puiseux algorithm
(hence an irreducibility test) which computes the singular parts of
all Puiseux series above $x=0$ in an expected $\Ot(\dy\,\vF)$
operations over $\Ki$%
. There exists also other methods for factorization, as the Montes
algorithm which allow to factor polynomials over general local fields
\cite{Mo99,GuMoNa12} with no assumptions on the characteristic of the
residue field. Similarly to the algorithms we present in this paper,
Montes et al. compute higher order Newton polygons and \edgepoly{}s
from the $\Phi$-adic expansion of $F$, where $\Phi$ is a sequence of
some well-chosen polynomials which is updated at each step of the
algorithm. With our notations, this leads to an irreducibility test in
$\Ot(\dy^2+\vF^2)$ \cite[Corollary 5.10 p.163]{BaNaSt13} when $\Ki$ is
a ``small enough'' finite field\footnote{This restriction on the field
  $\Ki$ is due to the univariate factorization complexity. It could be
  probably avoided by using dynamic evaluation.}. In particular, their
work provide a complete description of \emph{augmented valuations},
apparently rediscovering the one of MacLane
\cite{Ma36a,Ma36b,Ru14}. The closest
related result to this topic is the work of Abhyanhar \cite{Ab89},
which provides a new irreducibility test in $\Ci[[x]][y]$ based on
approximate roots, generalised to algebraically closed residue fields
of arbitrary characteristic in \cite{CoMo03}. No complexity estimates
have been made up to our knowledge, but we will prove that Abhyanhar's
irreducibility criterion is $\Ot(\vF)$ when $F$ is Weierstrass. In
this paper, we extend this result to non algebraically closed residue
field $\Ki[[x]][y]$ of characteristic zero or big enough. In some
sense, our approach establishes a bridge between the Newton-Puiseux
algorithm, the Montes algorithm and Abhyankar's irreducibility
criterion.  Let us mention also \cite{GaGw10,GaGw12} where an other
irreducibility criterion in $\algclos{\Ki}[[x]][y]$ is given in terms
of the Newton polygon of the discriminant curve of $F$, without
complexity estimates, and \cite{Pop02}, which provides a good
reference for the relations between approximate roots, Puiseux series
and resolution of singularities of an irreducible Weierstrass
polynomial $F\in \Ci[[x]][y]$.

\paragraph{Organisation.} In Section \ref{sec:ARNP} below, we describe
briefly the rational Newton-Puiseux algorithm of Duval \cite{DeDiDu85}
and its improved version of \cite{PoRy15} due to the so-called
Abhyankar trick. In section \ref{sec:phi}, we show how to recover the
edge data of $F$ from its $\Phi$-adic expansion, where $\Phi$ is the
collection of minimal polynomials of the truncated Puiseux series of
$F$. We show in
Section \ref{sec:psi} that $\Phi$ can be replaced by a collection
$\Psi$ of approximate roots of $F$ which can be computed
in the aimed complexity bound. Section \ref{sec:absolute} is dedicated to the absolute case, and a new proof of Abhyankhar's irreducibility criterion is given. In Section \ref{ssec:PseudoIrr}, we allow residual polynomials to be square-free, leading to the notion of pseudo-irreducible polynomials. 
Section
\ref{sec:comp} is dedicated to
complexity issues and to the proofs of Theorems \ref{thm:main}, \ref{thm:absolute} and \ref{thm:nonsqrfree}. We show in Section \ref{sec:equising} that a polynomial is pseudo-irreducible if and only if  its absolutely irreducible factors are equisingular and have same sets of pairwise intersection sets (balanced polynomials), in which case we give explicit formulas for characteristic exponents and intersection multiplicities in terms of the edges data, thus proving Theorem \ref{thm:main2}. We conclude in Section \ref{sec:conc} with ongoing researches about factorization of polynomials over general local fields of arbitrary residual characteristic.


\section{The Newton-Puiseux algorithm and Abhyankar
  trick}\label{sec:ARNP}

\paragraph{Classical definitions.} We first recall classical
definitions that play a central role for our purpose, namely the
Newton polygon and the residual polynomial. In the following, we
denote $F=\sum_{i=0}^\dy a_i(x)\,y^i$ and $\val$ the usual
$x$-valuation of $\Ki[[x]]$.
\begin{dfn}\label{dnf:NP}
  The \textit{Newton polygon} of $F$ is the lower convex hull $\NP(F)$
  of the set of points $(i,\val(a_i))$ for $i=0,\ldots,\dy$. The
  \textit{principal Newton polygon} $\NPb(F)$ is the union of edges of
  negative slopes of $\NP(F)$.
\end{dfn}

Note that $\NPb(F)=\NP(F)$ if $F$ is Weierstrass.  The Newton polygon
is used at the first call of our main algorithms, while the principal
Newton polygon is used for recursive calls. It is well known that
irreducibility in $\Ki[[x]][y]$ (resp. in $\Ki[[x,y]]$) implies
straightness of $\NP(F)$ (resp. $\NPb(F)$), a single point being
straight by convention. However, straightness condition is not
sufficient.

\begin{dfn}\label{dfn:edgePol}
  Given the (principal or not) Newton polygon $\NP$ of $F$, we call
  $\bar{F}:=\sum_{(i,j)\in \NP} a_{ij} x^j y^i$ the \emph{\edgepoly{}}
  of $F$.
\end{dfn}
  
\begin{dfn}\label{dfn:degeneracy}
  We say that $F$ is \emph{degenerated} over $\Ki$ with respect to
  $\NP$ if its \edgepoly{} $\bar{F}$ is the power of an irreducible
  quasi-homogeneous polynomial.
\end{dfn}

In other words, $F$ is degenerated if and only if $\NP$ is straight of
slope $-m/q$ with $q,m$ coprime, $q>0$, and if
\begin{equation}\label{eq:quasihom}
  \bar{F}=c \left(P\left( \frac{y^q}{x^{m}} \right) \,
    x^{\varepsilon m\deg(P)}\right)^N
\end{equation}
with $c\in \Ki^\times$, $N\in \Ni$ and $P\in \Ki[Z]$ monic and
irreducible, and where $\varepsilon=1$ if $m\ge 0$ and $\varepsilon=0$
otherwise.  We call $P$ the \emph{residual polynomial}\footnote{In the
  Montes algorithm \cite[Definition 1.9, page 368]{GuMoNa12}, the
  residual polynomial would rather design $P^{qN}$ in our context}
of $F$. We call the tuple $(q,m,P,N)$ the \emph{edge data} of the
degenerated polynomial $F$ and denote \Edata{} an algorithm computing
this tuple.

\paragraph{Newton-Puiseux irreducibility test.}
If $F$ is irreducible in $\Ki[[x]][y]$, it is degenerated. The
converse holds if $N=1$. If $N>1$, the Newton-Puiseux algorithm
ensures that $F$ is irreducible if and only if all the successive
so-called Puiseux transforms of $F$ (line \ref{RNP:puiseux} of \RNP{})
are degenerated until we reach an edge data with $N=1$.

We let $\ell:=\deg(P)$ and $\Ki_P:=\Ki[Z]/(P(Z))$. We denote by
$z\in \Ki_P$ the residue class of $Z$. Finally, we let
$s,t$ be the unique integers such that the B\'ezout relation $qs-mt=1$
holds with $0\leq t < q$. The rational version of the Newton-Puiseux
algorithm of Duval \cite{Du89} induces the following irreducibility
test. Therein, we check degeneracy with respect to $\NP(F)$ at the
first call and with respect to $\NPb(F)$ at the recursive calls.
\vspace{2mm}

\begin{algorithm}[H]
  \nonl\TitleOfAlgo{\RNP($F,\Ki$)\label{algo:RNP}}%
  \KwIn{$F\in\Ki[[x]][y]$ of degree $\dy>0$.}%
  \KwOut{\True{} if $F$ is irreducible in $\Ki[[x]][y]$, and \False{}
    otherwise.}%
  $N\gets d$\;%
  \While{$N>1$}{%
    \lIf{$F$ is not degenerated over $\Ki$}{\Return{\False}}\label{RNP:degener}%
    $(q,m,P,N)\gets \Edata{}(F)$\label{RNP:data}\;%
    $F\gets F(z^t x^q,x^m(y+z^s))/x^{qm\ell N}$\label{RNP:puiseux}\tcp*{Puiseux transform}
    $\Ki\gets \Ki_P$\;%
  }%
  \Return{\True}\;%
\end{algorithm}

The transform performed in \RNP{} differs slightly from the
classical Newton-Puiseux transform $F(x^q,x^m(y+z^{1/q}))$. This
trick due to Duval avoids to introduce useless field extension
$\Ki[z^{1/q}]$ of $\Ki[z]=\Ki_P$ inherent to ramification. The number
of iterations is bounded by $\vF$ and powers of $x$ can be truncated
modulo $x^{\vF+1}$, leading to a complexity $\Ot(\dy\,(\vF+1)^2)$
\cite[Lemma 4, page 213]{PoRy11}\footnote{In \cite{PoRy11}, $\Ki$ is
  assumed to be a finite field. This result remains correct on any
  perfect field if one uses dynamic evaluation instead of univariate
  factorization or irreducibility test.}.

\paragraph{The Abhyankar trick.} 
At each recursive call, the Weierstrass Preparation
Theorem ensures that the current polynomial $F$ of line
\ref{RNP:degener} equals a Weierstrass polynomial
$G=\sum_{i=0}^N g_i(x) y^i$ times a unit of $\Ki[[x,y]]$, and we can
compute $G$ up to an arbitrary precision via Hensel lifting. The Abhyankar trick consists to replace the current polynomial $F$ 
by the \textit{Abhyankhar shift} $H$ of its Weierstrass polynomial $G$:
\begin{equation}\label{eq:coefc}
  H(x,y)\gets G\big(x,y+c(x)\big),\qquad c(x):=-\frac{g_{N-1}(x)}{N}.
\end{equation}
At the first call, we assume $F$ monic and we rather consider $G=F$ in \eqref{eq:coefc}. We call $H$ the \textit{Abhyankar transform} of $F$, and denote $\abhyankar{}(F)$ the subroutine
computing it (with infinite precision in what follows, but with a
suitable finite precision in practice, see \cite[Section
3.3]{PoWe17}). The polynomial $H$ has now no terms of degree $N-1$, ensuring $q\ell >1$ at line \ref{RNP:data}. This leads to the following variant of \cite[Algorithm
\NPA{}]{PoRy15}, where we stop computations if we find out that
$F$ is reducible (therein, we keep the notations $\Ki_P$ and $z$ asociated to $P$).
\begin{algorithm}
  \nonl\TitleOfAlgo{\NPA($F,\Ki$)\label{algo:ARNP}}%
  \KwIn{$F\in\Ki[[x]][y]$ monic of degree $\dy>0$.}%
  \KwOut{\True{} if $F$ is irreducible in $\Ki[[x]][y]$, and \False{}
    otherwise.}%
  $N\gets d$\;%
  \While{$N>1$}{%
    $H\gets \abhyankar{}(F)$ \label{arnp:abhyankar}\;
    \lIf{$H$ is not degenerated over $\Ki$}{\Return{\False}}%
    $(q,m,P,N)\gets \Edata{}(H)$\;%
    $F\gets H(z^t x^q,x^m(y+z^s))/x^{qm\ell N} $ \label{arnp:puiseux}\;%
    $\Ki\gets \Ki_P$\;%
  }%
  \Return{\True}\;%
\end{algorithm}
\begin{rem}\label{rem:CharMonomials}
  The $(q,m)$-sequence of \NPA{} is not the same as the
  $(q,m)$-sequence of \RNP{}, but can be deduced from it \cite[Remark
  5 and Example 2]{PoWe17}. It contains enough information
  for computing the characteristic exponents of $F$; see Section
  \ref{sec:equising}.
\end{rem}
Since $H$ has degree $N$ with no terms of degree $N-1$, either it is
not degenerated, either its edge data satisfies $q \ell \ge 2$.  The
product of these invariants over all iterations satisfies
$\prod_k q_k \ell_k\le \dy$ (with equality if and only if $F$ is
irreducible), and the number of calls is less than $\log(\dy)$. See
\cite[Section 4]{PoRy15} for details.  If $F$ is irreducible, we have
moreover $\val(F_y(S))=\frac\vF\dy$ for any Puiseux series $S$ of
$F$. Then, \cite[Lemma 6]{PoWe17} and \cite[Corollary 4]{PoWe17} prove
that computations can be made modulo $x^{\frac{2\vF}{\dy}+1}$, leading
to an expected number of operations over $\Ki$ bounded by
$\O((\vF+1)\,\dy)$ \cite[Proposition 18]{PoWe17}. Moreover, as
mentionned in the conclusion of \cite{PoWe17}, this complexity
estimates is sharp. This is mainly due to the fact that despite the
input data being of size $\vF$ after truncation, the Puiseux
transform of line \ref{arnp:puiseux} generates a polynomial $F$
that can be of size $\Omega(\dy\,\vF)$ (which is then again reduced to
size $\O(\vF)$ at line \ref{arnp:abhyankar}). The approach we propose
in this paper avoids this intermediate increased size thanks to the
theory of approximate roots.

\paragraph{Notations.} If $\NP(F)$ is not straight, then $F$ is
reducible. If $\NP(F)$ is straight with positive slope, we replace $F$
by its reciprocal polynomial. The leading coefficient is now
invertible. Consequently, we assume in the remaining of this paper
that $F$ is monic.

We denote by $H_0:=\abhyankar{}(F)$ and let $N_0:=\deg(H_0)=\dy$. If
$N_0=1$ or $H_0$ is not degenerated, we let $g=0$. Otherwise, we
denote by $H_0,\ldots,H_{g-1}$ the successive degenerated polynomials
encountered at line \ref{arnp:abhyankar} when running
\NPA{}$(F,\Ki)$. We collect their respective edge data in a list
\[
  \D(F):=\big((q_1,m_1,P_1,N_1),\ldots,(q_{g},m_{g},P_{g},N_{g})\big).
\]
The monic polynomial $H_0$ might have horizontal slope (in which case $q_1=1$ and $m_1=0$) while $H_k$ is Weierstrass and $m_{k+1}>0$ for $1\leq k < g$. We
include the $N_k$'s in the list for convenience, although they can be
deduced from the remaining data by \eqref{eq:Nk}. This data is
closely related to what is called \emph{a type} in \cite{GuMoNa12}. The
integer $g$ is defined in such a way that we have either $N_g=1$ and
$F$ is irreducible, either the next Weierstrass polynomial $H_{g}$ is not degenerated and $F$ is
reducible. If $F$ is irreducible, we can deduce from $\D(F)$ the
characteristic monomials (exponents and coefficients) of any of its
conjugated Puiseux series, see Section \ref{sec:equising}.

We let $\Ki_0=\Ki$ and we denote $\Ki_{k}=\Ki_{k-1}[Z_k]/(P_{k}(Z_k))$
the field extension of $\Ki_{k-1}$ generated by $P_k$, where $Z_k$ is
a new undeterminate. It is a finite extension of $\Ki$ of degree
$f_k:=\ell_1\cdots \ell_k$, where $\ell_k:=\deg(P_k)$; it represents
the part of the residual extension discovered so far. We let
$z_k\in \Ki_k$ be the residue class of $Z_k\mod P_k$.

For all $1\le k \le g$, we have $H_k\in \Ki_k[[x]][y]$. The integer $N_k$ satisfies
\begin{equation}\label{eq:Nk}
  N_k=\deg(H_k)\,\,\text{ and}\,\, N_{k-1}=N_k q_k \ell_k.
\end{equation}
As $\ell_k q_k >1$ for all
$1\le k\le g$, the sequence $N_0,\ldots,N_g$ is a strictly
decreazing sequence of integers with $N_{k}$ dividing $N_{k-1}$.

We associate to $F$ the maps
\begin{equation}\label{eq:tausigma}
  \begin{cases}
    \tau_k(x,y)=(x,y+c_k(x)),\quad\quad\quad\quad\quad\quad\, 0\le k
    \le g,
    \\
    \sigma_k(x,y)=(z_k^{t_k} x^{q_k},x^{m_k}(y+z_k^{s_k})),
    \quad\quad\, 1\le k \le g
  \end{cases}
\end{equation}
respectively defined to be the successive Abhyankhar shifts
\eqref{eq:coefc} and rational Newton-Puiseux transforms performed
while running \NPA{}$(F,\Ki)$: $\tau_k$ performed at line
\ref{arnp:abhyankar} on the Weierstrass polynomial of $F$ (or directly on $F$ for $k=0$) and
$\sigma_k$ at line \ref{arnp:puiseux}, with $(s_k,t_k)$ the B\'ezout
co-factors of $(q_k,m_k)$. Note that $c_k(0)=0$ if $k\ge 1$ by
\eqref{eq:coefc}. Defining $\pi_0=\tau_0$, and
$\pi_k=\pi_{k-1}\circ \sigma_k\circ \tau_k$ for $1\leq k\leq g$, we get - see
Lemma \ref{lem:pik} for an explicit formula in terms of $\D(F)$:
\begin{equation}\label{eq:pikxy}
  \pi_k (x, y ) = (\mu_k x^{e_k},\alpha_k x^{r_k} y + S_{k} (x)),
\end{equation}
where $e_k:=q_1\cdots q_k$ (the ramification index discovered so far),
$\mu_k,\alpha_k \in \Ki_k^{\times}$, $r_k\in \Ni$ and
$S_k \in \Ki_k[[x]]$ satisfies $\val(S_k)\le r_k$. The pair
$R_k=(\mu_k x^{e_k}, S_k \mod x^{r_k+1})$ is called in \cite[Section
3.2]{PoWe17} a \emph{truncated rational Puiseux expansion}. We can deduce
from $R_k$ all the roots of $F$ (seen as Puiseux series) truncated up
to precision $\frac{r_k}{e_k}$, that increases with $k$.

By construction, there exists an integer $v_k(F)\in \Ni$ such that
\begin{equation}\label{eq:pikHk}
  \pi_k^* F = x^{v_k(F)} U_k H_k \in \Ki_k[[x]][y],
\end{equation}
where $U_k (0, 0) \in \Ki_k^\times$. This key point will be used
several times in the sequel. Also, note that the coefficient of
$y^{N_k-1}$ in $H_k$ is $0$ from the Abhyankar shift
\eqref{eq:coefc}.

\paragraph{Minimal polynomials.} There exists a unic monic irreducible
polynomial $\phi_k\in \Ki[[x]][y]$ (in practice in $\Ki[x][y]$ when
truncating powers of $x$) such that
\begin{equation}\label{eq:phik}
  \phi_k (\mu_k x^{e_k}, S_{k}) = 0\text{ and }d_k:=\deg(\phi_k)=e_k f_k.
\end{equation}
In particular, $\phi_0=y-c_0(x)$ has degree $1$. We call $\phi_k$ the
\emph{$k^{th}$ minimal polynomial} of $F$. We deduce from
\eqref{eq:Nk} $\dy=N_k\,d_k$ for all $k=0,\ldots,g$. By construction,
the function call \NPA{}($\phi_k$) generates the
same transformations $\tau_i$, $\sigma_i$ for $i\le k$ and we have
\begin{equation}\label{eq:dataphik}
  \D(\phi_k)=\left((q_1,m_1,P_1,N'_1),\ldots,(q_k,m_k,P_k,N_k'=1)\right)
  \text{ with } N'_i:=N_i/N_k.
\end{equation}
Note that up to some constant $c$,
$\phi_k$ may be computed as a multivariate resultant
\begin{equation}\label{eq:phikRes}
  \phi_k(x,y)=c\,\Res_{\Z,T}(x-\mu_k(\Z) T^e,y-S_k(\Z,T),P_1(Z_1),\ldots,P_k(Z_1,\ldots,Z_k)),
\end{equation}
where we consider here any liftings of the coefficients of
$\mu_k,S_k,P_1,\ldots,P_k$ from $\Ki_k=\Ki[z_1,\ldots,z_k]$ to the
polynomial ring $\Ki[\Z]:=\Ki[Z_1,\ldots,Z_k]$.


\section{Edge data from the $\Phi$-adic expansion}
\label{sec:phi}

Let us fix an integer $0\le k \le g$ and assume that $N_k>1$. Given
the edges data $(q_1,m_1,P_1,N_1),\ldots,(q_k,m_k,P_k,N_k)$ and the
minimal polynomials $\phi_0,\dots,\phi_k$, we want to decide if the
next Weierstrass polynomial $H_k$ is degenerated and if so, to compute
its edge data $(q_{k+1},m_{k+1},P_{k+1},N_{k+1})$.  

In the following, we will omit for readibility the index $k$ for the sets $\Phi$, $\Bc$, $V$ and $\Lambda$ defined below.

\subsection{Main results}\label{ssec:phi-main}

\paragraph{$\Phi$-adic expansion.} Let $\phi_{-1}:= x$ and denote
$\Phi = (\phi_{-1} ,\phi_0, \ldots,\phi_k)$. Let
\begin{equation}\label{eq:Bc}
  \Bc := \{(b_{-1} ,\ldots,b_k)\in \Ni^{k+2} \,\,, \,
  \, b_{i-1}<q_i\,\ell_i\,, i=1,\ldots,k\}
\end{equation}
and denote $\Phi^B:=\prod_{i=-1}^k\phi_i^{b_i}$. Thanks to the
relations $\deg(\phi_i) = \deg(\phi_{i-1})q_i \ell_i$ for all
$1 \le i \le k$, an induction argument shows that $F$ admits a unique
expansion
\[
  F = \sum_{B\in \Bc} f_B \Phi^B, \quad f_B\in \Ki.
\]
We call it the $\Phi$-adic expansion of $F$. Note that we have
necessarily $b_k\leq N_k$ while we do not impose any \textit{a priori}
condition to the powers of $\phi_{-1}=x$ in this expansion. The aim of
this section is to show that one can extract the edge data of $H_k$
from the $\Phi$-adic expansion of $F$.

\paragraph{Newton polygon.}
Consider the semi-group homomorphism
\[
  \begin{array}{rcl}
    v_k:\ (\Ki[[x]][y],\times)&\to& (\Ni\cup\{\infty\},+)\\%
    H & \mapsto & v_k (H) := \val (\pi_k^* H),
  \end{array}
\]
From \eqref{eq:pikxy}, we deduce that the pull-back morphism $\pi_k^*$
is injective, so that $v_k$ defines a discrete valuation. This is a
valuation of transcendence degree one, thus an augmented valuation
\cite[Section 4.2]{Ru14}, in the flavour of MacLane valuations
\cite{Ma36a,Ma36b,Ru14} or Montes valuations
\cite{Mo99,GuMoNa12}. Note that $v_0(H)=\val(H)$. We associate to
$\Phi$ the vector
\[
  V := (v_k(\phi_{-1}), \ldots, v_k(\phi_k)),
\]
so that $v_k(\Phi^B)=\langle B , V\rangle$, where $\langle \,,\,\rangle$
stands for the usual scalar product. For all $i\in \Ni$, we define the
integer 
\begin{equation}\label{eq:wj}
  w_i := \min \left\{\langle B, V\rangle, \,\, b_k=i,\,\,f_B \ne 0\right\}-v_k(F)
\end{equation}
with convention $w_i := \infty$ if the minimum is taken over the empty
set. 

\begin{thm}
  \label{thm:NPphi}
  The Newton polygon of $H_k$ is the lower
  convex hull of $(i,w_i)_{0\leq i\leq N_k}$.
\end{thm}

This result leads us to introduce the sets 
$$\Bc(i):=\{B\in\Bc ; b_k=i\}\quad \textrm{and}\quad
\Bc(i,w) := \{B \in \Bc(i) \ | \ \, \langle B, V \rangle = w\}$$ for
all $i\in\Ni$ and all $w\in \Ni\cup\{\infty\}$, with convention
$\Bc(i,\infty)=\emptyset$.

\paragraph{Boundary polynomial.} Consider the semi-group homomorphism
\[
\begin{array}{rcl}
\lambda_k:\ (\Ki[[x]][y ],\times)&\to& (\Ki_k,\times)\\%
H & \mapsto & \lambda_k (H) := \tc{\left(\frac{\pi_k^{*}
  (H)}{x^{v_k(H)}}\right)_{|x=0}}
\end{array}
\]
with convention $\lambda_k(0)=0$, and where $\textrm{tc}_y$ stands for
the trailing coefficient with respect to $y$ (initial coefficient). We
associate to $\Phi$ the vector
\[
  \Lambda := (\lambda_k (\phi_{-1}),\ldots,\lambda_k(\phi_k))
\]
and denote $\Lambda^B:=\prod_{i=-1}^k \lambda_k(\phi_i)^{b_i}$. Note
that $\Lambda^B\in \Ki_k$ is non zero for all $B$. 

\begin{thm}\label{thm:EdgePoly} Let $B_0 := (0,\ldots,0, N_k)$. The
  \edgepoly{} $\bar H_k$ of $H_k$ equals
  \begin{equation}\label{eq:barHk}
    \bar{H}_k =\sum_{(i,w_i)\in \NP(H_k)} \left(\sum_{B\in \Bc(i,w_i+v_k(F))}f_B \Lambda^{B-B_0} \right)x^{w_i} y^i.
  \end{equation}
\end{thm}

\begin{xmp}\label{xmp:k0}
  If $k=0$, we have by definition $V=(1,0)$ and $\Lambda=(1,1)$ while
  $v_0(F)=\val(H_0)=0$. Assuming $H_0=\sum_{j=0}^\dy a_i (x) y^i$, we find $w_i=\val(a_i)$ and Theorem \ref{thm:NPphi} stands from
  Definition \ref{dnf:NP}. Moreover, $\Bc(i,w_i)$ is then
  reduced to the point $(i,w_i)$ and Theorem \ref{thm:EdgePoly}
  stands from Definition \ref{dfn:edgePol}.
\end{xmp}
\subsection{Key Proposition and proofs of Theorems \ref{thm:NPphi} and
  \ref{thm:EdgePoly}}

Let us first establish some basic properties of the minimal
polynomials $\phi_i$ of $F$. Given a ring $\Ai$, we denote by
$\Ai[[x,y]]^\times$ the set of all $U\in \Ai[[x,y]]$ for which
$U(0,0)\ne 0$.  If $\Ai$ is a field (which might not be the case in
Sections \ref{sec:comp} and \ref{sec:equising}), this is simply the
group of units of the ring $\Ai[[x,y]]$.  For $-1 \leq i \leq k$, we
introduce the notations
\[
  v_{k,i} := v_k (\phi_i )= \val(\pi_k^*(\phi_i)) \quad \text{ and }
  \quad \lambda_{k,i} := \lambda_k (\phi_i
  )=\tc{\left(\frac{\pi_k^*(\phi_i)}{x^{v_{k,i}}}\right)_{|x=0}}.
\]
\begin{lem}\label{lem:PikPhi} Let $-1\le i \le k$. There exists $U_{k,i}\in \Ki_k [[x,y]]^\times$ with
  $U_{k,i} (0, 0) = \lambda_{k,i}$ s.t.:
  \begin{eqnarray*}
    \begin{cases}\pi_k^* (\phi_i ) = U_{k,i} x^{v_{k,i}} \text{ if }  i < k \\
      \pi_k^* (\phi_k) = U_{k,k}x^{v_{k,k}} y,
    \end{cases}
  \end{eqnarray*}
\end{lem}

\begin{proof} As \NPA{}($\phi_k$) generates the same transform
  $\pi_k$, we deduce from \eqref{eq:pikHk}:
  \[
    \pi_k^*(\phi_k) = x^{v_k(\phi_k)}\,U(x,y)\,(y+\beta(x))
  \]
  with $U \in \Ki_k [[x,y]]^\times$ and $\beta\in \Ki_k [[x]]$. From
  \eqref{eq:pikxy} and \eqref{eq:phik}, we get
  $x^{v_{k,k}} U(x,0)\beta(x) = \phi_k(\mu_k x^{e_k},S_{k}) = 0$,
  i.e. $\beta = 0$. Second equality follows, since
  $U(0, 0) = \lambda_{k,k}$ by definition of $\lambda_k$. First
  equality follows from the second one by applying the pull-backs
  $(\sigma_j\circ \tau_j)^*$, $j=i+1,\ldots,k$ to
  $\pi_i^*(\phi_i)=x^{v_{ii}}\,y\,U_{ii}$.
\end{proof}

\begin{cor}\label{cor:res}
  With the standard notations for intersection multiplicities and
  resultants, we have
  \[
  v_k (\phi_i ) =\frac{(\phi_i,\phi_k)_0}{f_k}=
  \frac{\val (\Res_y(\phi_i,\phi_k))}{f_k}
  , \ -1 \le i \le k - 1.
  \]
\end{cor}

\begin{proof}
  By point 2 in Lemma \ref{lem:PikPhi}, we deduce that
  \[
  v_k (\phi_i) := \val(\pi_k^*(\phi_i)) = \val(\phi_i(\mu_k
  x^{e_k}, S_{k}(x))) \text{ since } \val(S_k)\leq r_k.
  \]
   But this last integer coincides
  with the intersection multiplicity of $\phi_i$ with any one of the
  $f_k$ conjugate plane branches (i.e. irreducible factor in
  $\algclos{\Ki}[[x]][y]$) of $\phi_k$. The first equality follows. The
  second is well known (the intersection multiplicity at $(0, 0)$ of
  two Weierstrass polynomials coincides with the $x$-valuation of
  their resultant).
\end{proof}

\begin{lem}\label{lem:vLambda} We have  initial conditions
  $v_{0,-1}=1$, $v_{0,0}=0$, $\lambda_{0,-1}=1$ and $\lambda_{0,0}=1$.
  Let $k\ge 1$. The following relations hold $($we recall
  $q_k s_k - m_k t_k = 1$ with $0\le t_k < q_k):$
  \begin{enumerate}
  \item $v_{k,k-1} = q_k v_{k-1,k-1} + m_k$
  \item $v_{k,i} = q_k v_{k-1,i}$ for all $-1\le i < k- 1$.
  \item
    $\lambda_{k,k-1} = \lambda_{k-1,k-1} z_k^{t_k v_{k-1,k-1}
      +s_{k}}$.
  \item $\lambda_{k,i} =\lambda_{k-1,i}z_k^{t_k v_{k-1,i}}$ for all
    $-1 \le i < k - 1$.
  \end{enumerate}
\end{lem}

\begin{proof} Initial conditions follow straightforwardly from the
  definitions. From point 1 of Lemma \ref{lem:PikPhi} and the
  definition of $\pi_k$, we have
  $\pi_k^*(\phi_{k-1})=\tau_k^*\circ \sigma_k^*\circ
  \pi_{k-1}^*(\phi_{k-1})$, i.e.%
\[
\pi_k^*(\phi_{k-1})=z_k^{t_k v_{k-1,k-1}} x^{q_k v_{k-1,k-1} +
  m_k}\,\tilde y U_{k-1,k-1}(z_k^{t_k}
x^{q_k},x^{m_k}\,\tilde y).
\]
where $\tilde y=y+z_k^{s_k}+c_k(x)$. As $c_k(0)= 0$,
$m_k > 0$ and $z_k\ne 0$, it follows that
\[
\pi_k^*(\phi_{k-1})=z_k^{t_k v_{k-1,k-1}+s_k} x^{q_k\,v_{k-1,k-1}+m_k}
\tilde{U}(x,y)
\]
with $\tilde{U}(0,0)=U_{k-1,k-1}(0,0)$, that is $\lambda_{k-1,k-1}$ by
point 1 of Lemma \ref{lem:PikPhi}. Items $1$ and $3$ follow. Similarly,
using point 2 of Lemma \ref{lem:PikPhi}, we get for all $i<k-1$
\[
\pi_k^*(\phi_i)=\tau_k^*\circ \sigma_k^*\circ
\pi_{k-1}^*(\phi_i)=z_k^{t_k v_{k-1,i}} x^{q_k v_{k-1,i}}
U_{k-1,i}(z_k^{t_k} x^{q_k},x^{m_k}(y+z_k^{s_k}+c_k(x))).
\]
As $U_{k-1,i}(0,0)=\lambda_{k-1,i}\ne 0$ once again by Point 2 of
Lemma \ref{lem:PikPhi}, items $2$ and $4$ follow.
\end{proof}

The proof of both theorems is based on the following key result:

\begin{prop}\label{prop:key} For all $i, w \in\Ni$, the family
  $\left(\Lambda^B , B \in \Bc(i,w)\right)$ is free over $\Ki$. In
  particular, $\Card ( \Bc(i,w))\le f_k$.
\end{prop}

\begin{proof} We show this property by induction on $k$. If $k = 0$,
  the result is obvious since $\Bc(i,w)=\{(i,w)\}$ and
  $\Lambda=(1,1)$. Suppose $k > 0$. As $\lambda_{k,k}$ is invertible
  and $b_k=i$ is fixed, we are reduced to show that the family
  $\left(\Lambda^B , B \in \Bc(0,w)\right)$ is free for all
  $w\in \Ni$. Suppose given a $\Ki$-linear relation
  \begin{equation}\label{eq:cBlambdaB}
    \sum_{B\in \Bc(0,w)} c_B \Lambda^B =   \sum_{B\in \Bc(0,w)} c_B \lambda_{k,-1}^{b_{-1}}\cdots\lambda_{k,k-1}^{b_{k-1}}= 0.
  \end{equation}
  Using $b_k = 0$, points 3 and 4 in Lemma \ref{lem:vLambda} give
  $\Lambda^B = \mu_B z_k^{N_B}$ where
  \[
    \mu_B = \prod_{j=-1}^{k-1} \lambda_{k-1,j}^{b_j} \in
    \Ki_{k-1}\quad \textrm{and}\quad N_B = b_{k-1} s_k +
    t_k\sum_{j=-1}^{k-1} b_j v_{k-1,j}.
  \]
  Points 1 ($q_k\,v_{k-1,k-1}=v_{k,k-1}-m_k$) and 2
  ($q_k\,v_{k-1,j} = v_{k,j}$) in Lemma \ref{lem:vLambda} give
  \begin{equation}\label{eq:qkNB}
    q_k N_B = b_{k-1} (q_k s_k - m_k t_k ) + t_k \sum_{j=-1}^{k-1} b_j
    v_{k,j} = b_{k-1} + t_k w,
  \end{equation}
  the second equality using $\langle B, V \rangle = w$ and
  $b_k=0$. Since $0\le b_{k-1} < q_k \ell_k$ and $N_B$ is an integer,
  it follows from \eqref{eq:qkNB} that $N_B=n+\alpha$ where
  $n=\lceil t_k w/q_k\rceil$ and $0\leq \alpha <\ell_k$. Dividing
  (\ref{eq:cBlambdaB}) by $z_k^n$, we get
  \[
    \sum_{\alpha=0}^{\ell_k -1} a_{\alpha} z_{k}^\alpha = 0, \quad
    \text{ where}\quad a_{\alpha} =\sum_{B\in \Bc(0,w) ,N_B =
      \alpha+n} c_{B} \mu_B .
  \]
  Since $a_{\alpha} \in \Ki_{k-1}$ and $z_k\in \Ki_k$ has
  minimal polynomial $P_k$ of degree $\ell_k$ over $\Ki_{k-1}$, this implies
  $a_{\alpha} = 0$ for all $0\leq \alpha < \ell_k$, i.e., using
  \eqref{eq:qkNB}:
  \[
    \sum_{\substack{B\in \Bc(0,w)\\ b_{k-1}=q_k\,(\alpha+n)-t_k\,w}}
    c_{B} \lambda_{k-1,-1}^{b_{-1}}\cdots \lambda_{k-1,k-1}^{b_{k-1}}
    = 0.
  \]
  By induction, we get $c_B=0$ for all $B\in \Bc(0,w)$, as required.
  The first claim is proved. The second claim follows immediately
  since $\Lambda^B \in \Ki_k$ is non zero for all $B$.
\end{proof}

\begin{cor}\label{cor:main}
  Consider $G=\sum_{B\in\Bc(i)} g_B \Phi^B$ non zero.  Then
  $\pi_k^*(G)=\Ut x^{w}\,y^i$ with $\Ut\in \Ki_k[[x,y]]^ \times$,
    $w=\min_{ g_B\ne 0} \langle B,V\rangle$ and
    $\Ut(0,0)=\sum_{B\in \Bc(i,w)}g_B\Lambda^B\ne 0$.
  In particular, $v_k(G)=w$ and $\lambda_k(G)=\Ut(0,0)$.
\end{cor}

\begin{proof}
  By linearity of $\pi_k^*$, denoting
  $U = (U_{k,-1} ,\ldots , U_{k,k} )$ with $U_{k,i}$ defined in Lemma
  \ref{lem:PikPhi}, we have
  \[
    \pi_k^*(G) = \left(\sum_{B\in\Bc(i)}g_B U^B x^{<B,V>}\right)y^i
    \text{ with } U(0,0)=\Lambda.
  \]
  Letting $w=\min_{ g_B\ne 0} \langle B,V\rangle$, we deduce
  \[
    \pi^{*}_k(G) =\left(\sum_{B\in \Bc(i,w)} g_B \Lambda^B + R \right)
    x^{w} y^i \quad \text{ where } R\in \Ki_k[[x,y]] \text{ satisfies } R(0,0)=0.
  \]
  As $\sum_{B\in\Bc(i,w)}g_B \Lambda^B \neq 0$ by Proposition
  \ref{prop:key}, the first two equalities follows. The last two
  equalities follow from the definitions of $v_k(G)$ and
  $\lambda_k(G)$.
\end{proof}

\paragraph{Proof of Theorems \ref{thm:NPphi} and \ref{thm:EdgePoly}.} We prove both theorems
simultaneously. We may write $F=\sum_{i=0}^{N_k} \sum_{B\in \Bc(i)} f_B\Phi^B$. Hence, Corollary \ref{cor:main} combined with the definition
of $w_i$ and the linearity of $\pi_k^*$ implies
\begin{equation}\label{eq:Fk}
  F_k :=\frac{\pi^{*}_k(F)}{x^{v_k(F)}}
  =\sum_{i=0}^{N_k} \, \Ut_i\,x^{w_i}\, y^i
\end{equation}
where $\Ut_i\in \Ki_k[[x,y]]$ is $0$ if $w_i=\infty$, and
$\Ut_i(0,0)=\sum_{B\in \Bc(i,w_i+v_k(F))} f_B \Lambda^B\ne 0$
otherwise.  Let $\NP$ stands for the lower convex hull of the set
$((i,w_i),i=0,\ldots,N_k)$ and let $\NPb$ be the subset of negative
slopes.  If $k=0$, then $H_0=F_0$ and we have moreover $\Ut_i\in \Ki$
for all $i$ (use that $\pi_0^*(\phi_0)=y$). It follows immediately
from \eqref{eq:Fk} that $\NP(H_0)=\NP$, as required. If $k>0$, then we
deduce from \eqref{eq:Fk} that $\NPb(F_k)=\NPb$.  Combined with
\eqref{eq:pikHk}, we get $\NPb(H_k)=\NPb$. As $k\ge 1$, $H_k$ is
Weierstrass of degree $N_k$, 
which forces $\NP(H_k)=\NP$ as required. This proves Theorem \ref{thm:NPphi}. Combined with
(\ref{eq:pikHk}), we deduce more precisely that there exists
$\mu \in \Ki_k^\times$ such that
\begin{equation}\label{eq:muH}
  \mu\bar{H}_k=\sum_{(i,w_i )\in\NP} \left(\sum_{B\in \Bc(i,w_i+v_k(F))} f_B \Lambda^B \right) x^{w_i} y^i.
\end{equation}
Since $\bar{H}_k$ is monic of degree $N_k$, we get $w_{N_k} = 0$ and
$w_i \ge 0$ for $i< N_k$ and we deduce from \eqref{eq:muH} that
\begin{equation}\label{eq:mu}
  \mu=\sum_{B\in \Bc(N_k,v_k(F))} f_B \Lambda^B .
\end{equation}
But $F$ and $\phi_k$ being monic of respective degrees $\dy$ and
$d_k$, the vector $B_0 = (0,\ldots 0, N_k) \in \Bc$ is the unique
exponent in the $\Phi$-adic expansion of $F$ with last coordinate
$b_k = N_k=d/d_k$ and we have moreover $f_{B_0} = 1$.  Since
$\Bc(N_k,v_k(F))$ is non empty by construction, this forces
$\Bc(N_k,v_k(F)) = \{B_0\}$ and (\ref{eq:mu}) becomes
$\mu = \Lambda^{B_0}$. Theorem \ref{thm:EdgePoly} follows.
$\hfill\square$

\subsection{Formulas for $\lambda_k(\phi_k)$ and $v_k(\phi_k)$}
In order to use Theorems \ref{thm:NPphi} and \ref{thm:EdgePoly} for
computing the edge data of $H_k$, we need to compute
$v_{k,k}:=v_k(\phi_k)$ and $\lambda_{k,k}:=\lambda_k(\phi_k)$ in terms
of the previously computed edges data
$(q_1,m_1,P_1,N_1),\ldots,(q_k,m_k,P_k,N_k)$ of $F$. We begin with the
following lemma:

\begin{lem}\label{lem:FPhi}
  Let $0\le k \le g$. We have $v_k(F)= N_k v_{k,k}$ and
  $\lambda_k(F) =\lambda_{k,k}^{N_k}$.
\end{lem}

\begin{proof}
  We have shown during the proof of Theorems \ref{thm:NPphi} and
  \ref{thm:EdgePoly} that $\Bc(N_k,v_k(F))=\{B_0\}$ where
  $B_0=(0,\ldots,0,N_k)$. By definition of $\Bc(N_k,v_k(F))$, we get
  the first point. By definition of $\lambda_k$, we have
  $\lambda_k(F)=\tc{F_k(0,y)}=\tc{\bar{F_k}(0,y)}$ and we have shown
  that $\bar{F_k}(0,y)=\Lambda^{B_0}\bar{H_k}(0,y)$. Since $\bar{H_k}$
  is monic, we deduce that $\tc{\bar{F_k}(0,y)}=\Lambda^{B_0}$, giving
  the second claim.
\end{proof}

\begin{prop}\label{thmvkk}
  For any $1\le k \le g$, we have the equalities
  $$v_{k,k}= q_k \ell_k v_{k,k-1}\quad and \quad \lambda_{k,k}=q_k
  z_k^{1-s_k-\ell_k}P_k'(z_k)\lambda_{k,k-1}^{q_k \ell_k}.$$
\end{prop}

\begin{proof}
  To simplify the notations of this proof, let us denote
  $w=v_{k-1}(\phi_k)$, $\gamma=\lambda_{k-1}(\phi_k)$ and
  $(m,q,s,t,\ell,z)=(m_k,q_k,s_k,t_k,\ell_k,z_k)$.  By definition of
  $\phi_k$, both $\phi_k$ and $F$ generate the same transformations
  $\sigma_i$ and $\tau_i$ for $i\leq k$. As in (\ref{eq:pikHk}), there
  exists $\Ut_{k-1}\in\Ki[[x,y]]^\times$ satisfying
  $\Ut_{k-1}(0,0)=\gamma$ and $\tilde{H}_{k-1}\in\Ki[[x]][y]$
  Weierstrass of degree $q\,\ell$ such that
  $\pi_{k-1}^{*}(\phi_k)= x^w \tilde{H}_{k-1} \Ut_{k-1}$, where
  \begin{equation*}\label{eqtildeHk}
    \tilde H_{k-1}(x,y)=P_k(x^{-m}\,y^q)x^{m \ell}+\sum_{m i + q j > m q \ell} h_{ij} x^j y^i.
  \end{equation*}
  We deduce that there exists $R_0,R_1,R_2\in \Ki_k[[x,y]]$ such that
  \begin{eqnarray*}
    \pi_{k}^{*} (\phi_k) & := & (\pi_{k-1}^{*} (\phi_k)) (z^t x^q,x^m(y+z^s+c_k(x))) \\
                         & =  &z^{tw}x^{qw}\left(z^{tm\ell}\,x^{mq\ell}(G_k+x R_0)\right)\,\left(\gamma+x\,R_1+y\,R_2\right)%
  \end{eqnarray*}%
  where we let
  $G_k(x,y):=P_k(z^{-tm}(y+z^s+c_k(x))^q)\in\Ki_k[[x]][y]$. It follows
  that there exists $R\in \Ki_k[[x,y]]$ such that
  \begin{equation}\label{eqpikstar}
    \pi_{k}^{*} (\phi_k) = \gamma\,z^{t(w+m\ell)}\,x^{q\,(w+m\ell)} \left(G_k(1+y R_2)+x\,R\right)
  \end{equation}      
  As $G_k(0,y)$ is not identically zero, we deduce from
  \eqref{eqpikstar} that $v_k(\phi_k)=q\,(w+m\ell)$. Using Lemma
  \ref{lem:FPhi} for $F=\phi_k$ and the valuation $v_{k-1}$, together
  with Point 1 of Lemma \ref{lem:vLambda}, we have
  $w+m\ell=\ell v_{k,k-1}$, which implies $v_{k,k}=q\,\ell\,v_{k,k-1}$
  as expected.  Using $c_k(0)=0$ and the relation $sq-tm=1$, we see
  that $G_k(0,0)=P_k(z_k)=0$ and
  $\partial_y G_k (0,0)=q z^{1-s} P'_k(z)$ is non zero. Combined with
  \eqref{eqpikstar}, we get that
  \[
    \lambda_{k,k}=\gamma z^{t\,\ell\,v_{k,k-1}} \left(q z^{1-s}
      P'_k(z)\right) =
    \gamma\,z^{q\,\ell\,t\,v_{k-1,k-1}+\ell\,t\,m+1-s}\,q\,P_k'(z),
  \]
  the second equality using Point 1 of Lemma \ref{lem:vLambda} once
  again. Now, using Lemma \ref{lem:FPhi} for $F=\phi_k$ and the
  morphism $\lambda_{k-1}$, we get $\gamma=\lambda_{k-1,k-1}^{\ell q}$
  so that
  $\lambda_{k,k} =q P'_k(z)\lambda_{k,k-1}^{q \ell} z^{1-s-\ell}$ as
  expected.
\end{proof}

\begin{rem}\label{rem:zkdiv0}
  If $q_k=1$, the second formula simplifies as
  $\lambda_{k,k}=P_k'(z_k)\lambda_{k-1,k-1}^{\ell_k}$. Moreover, we
  then have $t=0$ and $s=1$ so that no division by $z_k$ is done in
  the above proof. This remark is used in Sections \ref{sec:comp} and
  \ref{sec:equising} where $z_k$ might be a zero divisor when $q_k=1$.
\end{rem}

\subsection{Simple formulas for $V$ and $\Lambda$.}\label{subsec:resume}
For convenience to the reader, let us summarize the formulas which
allow to compute in a simple recursive way both lists
$V=(v_{k,-1},\ldots, v_{k,k})$ and
$\Lambda=(\lambda_{k,-1},\ldots,\lambda_{k,k})$.

If $k=0$, we let $V=(1,0)$ and $\Lambda=(1,1)$. Assume $k\ge 1$. Given
the lists $V$ and $\Lambda$ at rank $k-1$ and given the $k$-th edge
data $(q_k,m_k,P_k,N_k)$, we update both lists at rank $k$ thanks to
the formul\ae{}:
\begin{equation}\label{eq:update}
  \begin{cases}
    v_{k,i} = q_k v_{k-1,i} \quad {\small\text{$-1\le i<k-1$}}\\%
    v_{k,k-1} = q_k v_{k-1,k-1}+m_k\\%
    v_{k,k} = q_k\ell_k v_{k,k-1}
  \end{cases}
  \ 
  \begin{cases}
    \lambda_{k,i} = \lambda_{k-1,i}z_k^{t_k v_{k-1,i}}\quad
    {\small\text{$-1\le i<k-1$}}\\%
    \lambda_{k,k-1} = \lambda_{k-1,k-1} z_k^{t_k v_{k-1,k-1}+s_{k}}\\%
    \lambda_{k,k} = q_k
    z_k^{1-s_k-\ell_k}P_k'(z_k)\lambda_{k,k-1}^{q_k
      \ell_k}
  \end{cases}
\end{equation}
where $q_ks_k-m_k t_k=1$, $0\le t_k< q_k$ and $z_k=Z_k\mod P_k$.


\section{From minimal polynomials to approximate roots}\label{sec:psi}

Given $\Phi=(\phi_{-1},\ldots,\phi_k)$ and $F=\sum f_B \Phi^B$ the
$\Phi$-adic expansion of $F$, the updated lists $V$ and $\Lambda$ then
allow to compute in an efficient way the \edgepoly{}
$\bar H_k$ thanks to the formulas \eqref{eq:wj} and \eqref{eq:barHk}.
Unfortunately, the computation of the minimal polynomials $\phi_k$ is
up to our knowledge too expensive to fit in our aimed complexity
bound. For instance, it requires to know the $y^{N_k-1}$ coefficients
of the Puiseux transform of $H_{k-1}$ up to some suitable precision,
and computing this Puiseux transform might be costly, as explained in
Section \ref{sec:ARNP}.

In this section, we show that the main conclusions of all previous
results remain true if we replace $\phi_k$ by the
$N_k^{th}$-approximate root $\psi_k$ of $F$, with the great advantage
that these approximate roots can be computed in the aimed complexity
(see Section \ref{sec:comp}). Up to our knowledge, such a strategy was
introduced by Abhyankar who developped in \cite{Ab89} an
irreducibility criterion in $\algclos{\Ki}[[x,y]]$ which avoids to perform
any Newton-Puiseux type transforms.

\subsection{Approximate roots and main result} 

\paragraph*{Approximate roots.} The approximate roots
of a monic polynomial $F$ are defined thanks to the following proposition: 

\begin{prop}\label{prop:approx-root} $($see e.g. {\upshape
    \cite[Proposition 3.1]{Pop02}).} Let $F\in\Ai[y]$ be monic of
  degree $\dy$, with $\Ai$ a ring whose characteristic does not divide
  $\dy$. Let $N\in \Ni$ dividing $\dy$. There exists a unique
  polynomial $\psi\in \Ai[y]$ monic of degree $d/N$ such that
  $\deg(F-\psi^{N}) < d-d/N$. We call it the $N^{th}$ approximate roots of $F$. 
\end{prop}

A simple degree argument implies that $\psi$ is the $N^{th}$-approximate
root of $F$ if and only if the $\psi$-adic expansion
$\sum_{i=0}^N a_i \psi^i$ of $F$ satisfies $a_{N-1}=0$. For instance,
if $F=\sum_{i=0}^\dy a_i y^i$, the $\dy^{th}$ approximate root
coincides with the Tschirnhausen transform of $y$
\[
  \tau_F(y)=y+\frac{a_{\dy-1}}{\dy}.
\]
More generally, the $N^{th}$ approximate root can be constructed as
follows. Given $\phi\in R[y]$ a monic of degree $\dy/N$ and
given $F=\sum_{i=0}^N a_i \phi^i$ the $\phi$-adic expansion of $F$, we
consider the new polynomial
\[
 \tau_F(\phi):=\phi+\frac{a_{N-1}}{N}
\]
which is again monic of degree $\dy/N$. It can be shown that the
resulting $\tau_F(\phi)$-adic expansion $F=\sum a_i' \tau_F(\phi)^i$ satisfies
$\deg(a_{N-1}')< \deg(a_{N-1})< \dy/N$ (see e.g. \cite[Proof of
Proposition 6.3]{Pop02}). Hence, after applying at most $\dy/N$ times
the operator $\tau_F$, the coefficient $a_{N-1}'$ vanishes and the
polynomial $\tau_F\circ\cdots \circ \tau_F(\phi)$ coincides with the approximate root $\psi$ of
$F$. Although this is not the best strategy from a complexity point of
view (see Section \ref{sec:comp}), this construction will be used to
prove Theorem \ref{thm:HPsiPhi} below.

\paragraph{Main result.} We still consider $F\in \Ki[[x]][y]$
monic of degree $\dy$ and keep notations from Section
\ref{sec:ARNP}. We denote $\psi_{-1}:=x$ and, for all $k=0,\ldots,g$,
we denote $\psi_k$ the $N_k^{th}$-approximate root of $F$. Fixing
$0\leq k\leq g$, we denote $\Psi=(\psi_{-1},\psi_0,\ldots,\psi_{k})$,
omitting once again the index $k$ for readibility.

Since $\deg \Psi=\deg \Phi$ by definition, the exponents of the
 $\Psi$-adic expansion
\[
  F=\sum_{B\in \Bc} f_B' \Psi^B, \quad f_B'\in \Ki
\]
take their values in the same set $\Bc$ introduced in
(\ref{eq:Bc}). In the following, we denote by $w_i'\in\Ni$ the new
integer defined by (\ref{eq:wj}) when replacing $f_B$ by $f'_B$ and we
denote $\bar{H}_k'$ the new polynomial obtained when replacing $w_i$
by $w_i'$ and $f_B$ by $f_B'$ in (\ref{eq:barHk}).

\begin{thm}\label{thm:HPsiPhi}
  We have $\bar{H}_k=\bar{H}'_k$ for $0\leq k< g$ and the \edgepoly{}
  $\bar{H}_g$ and $\bar{H}'_g$ have same restriction to their Newton polygon's
  lower edge.
\end{thm}

In other words, Theorems \ref{thm:NPphi} and \ref{thm:EdgePoly} still
hold when replacing minimal polynomials by approximate roots, up to a
minor difference when $k=g$ which has no impact for degeneracy tests.

\paragraph{Intermediate results.} The proof of Theorem
\ref{thm:HPsiPhi} requires several steps.  We denote by
$-m_{g+1}/q_{g+1}$ the slope of the lower edge of $H_g$.
\begin{lem}\label{lem:vpsiphi}
  We have $v_k(\psi_k-\phi_k) > v_k (\phi_k)+m_{k+1}/q_{k+1}$ for all
  $k=0,\ldots,g$.
\end{lem}

\begin{proof}
  Let $(q,m)=(q_{k+1},m_{k+1})$. Since the lemma is true if
  $\psi_k=\phi_k$ and since $\psi_k$ is obtained after successive
  applications of the operator $\tau_F$ to $\phi_k$, it is sufficient
  to prove
  \begin{equation}\label{eq:vkphiphik}
    v_k(\phi-\phi_k)> v_k (\phi_k)+m/q \Longrightarrow
    v_k(\tau_F(\phi)-\phi_k) > v_k (\phi_k)+m/q
  \end{equation}
  for any $\phi\in \Ki[[x]][y]$ monic of degree $d_k$. Suppose given
  such a $\phi$ and consider the $\phi$-adic expansion
  $F=\sum_{j=0}^{N_k} a_j \phi^{j}$. Then \eqref{eq:vkphiphik} holds
  if and only if
  \begin{equation}\label{eq:vkaNk}
    v_k(a_{N_k-1})>v_k (\phi_k)+m/q.
  \end{equation}  
  \noindent
  $\bullet$ \emph{Case $\phi=\phi_k$.} Then we have
  $v_k(a_{N_k-1})\geq v_k(F)+m/q=N_k\,v_k(\phi_k)+m/q$ from Theorem
  \ref{thm:NPphi} and Lemma \ref{lem:FPhi}. Suppose first that
  $v_k(\phi_k)=0$. This may happen only when $k=0$, or $k=1$ if
  $m_1=0$. As $\phi_0=\psi_0$, we do not need to consider the case
  $k=0$. If $m_1=0$, the first slope is horizontal and the Abhyankhar
  transform \eqref{eq:coefc} ensures that the coefficient of
  $y^{N_1-1}$ has no constant term, so that $v_1(a_{N_1-1})>0$ as
  required. Suppose now $v_k(\phi_k)>0$. If $N_k>1$, we are done. If
  $N_k=1$, we must have $k=g$ and $H_g=y$, so that $v_g(a_0)=\infty$
  from Theorem \ref{thm:NPphi} and the claim follows.

  \vspace{2mm}
  \noindent
  $\bullet$ \emph{Case $\phi\neq \phi_k$.} First note that
  $v_k(\phi-\phi_k) > v_k(\phi_k)$ implies $v_k(\phi) =v_k(\phi_k)$.
  As $\deg(\phi-\phi_k)<d_k$, we deduce from Corollary \ref{cor:main}
  (applied to $G=\phi-\phi_k$ and $i=0$) and Lemma
  \ref{lem:PikPhi} that
  \[
    \pi_k^*(\phi)=\pi_k^*(\phi-\phi_k)+\pi_k^*(\phi_k)=x^{v_k(\phi)}U(y
    + x^{\alpha} \Ut)
  \]
  where $\alpha:=v_k(\phi-\phi_k)-v_k(\phi_k)>m/q$ (hypothesis) and
  for some units $U,\Ut\in \Ki_k[[x,y]]^\times$. As $a_i$ has also
  degree $<d_k$, we deduce again from Corollary \ref{cor:main} that
  when $a_i\neq 0$,
  \begin{equation}\label{eq:ajphij}
    \pi_k^*(a_i\phi^i)=x^{\alpha_i}U_i (y + x^{\alpha} \Ut)^i,
  \end{equation}
  where $\alpha_i:=v_k(a_i\phi^i)$ and $U_i\in \Ki[[x,y]]^\times$.  As
  $\alpha>m/q$, this means that the lowest line with slope $-q/m$
  which intersects the support of $\pi_k^* (a_i\phi^{i})$ intersects
  it at the unique point $(i,\alpha_i)$. Since
  $ \pi_k^*(F)=\sum_{i=0}^{N_k} \pi_k^* (a_i\phi^{i})$, we deduce that
  the edge of slope $-q/m$ of the Newton polygon of $\pi_k^*(F)$
  coincides with the edge of slope $-q/m$ of the lower convex hull of
  $\left((i,\alpha_i)\;; a_i\neq 0, 0\leq i\leq
    N_k\right)$. Thanks to \eqref{eq:pikHk} combined with
  $v_k(F)=N_k v_{k}(\phi_k)$ (Lemma \ref{lem:FPhi}) and
  $v_k(\phi_k)=v_k(\phi)$ (hypothesis), we deduce that the lower edge
  $\Delta$ of $H_k$ with slope $-q/m$ coincides with the edge of slope
  $-q/m$ of the lower convex hull of the points
  $\left((i,v_k(a_i)+(i-N_k)\,v_k(\phi))\;; a_i\neq 0, 0\leq i\leq
    N_k\right)$.  Since $H_k$ is monic of degree $N_k$ \textit{with no
    terms of degree $N_{k}-1$}, we deduce that $(N_k,0)\in \Delta$
  while $(N_{k}-1,v_k(a_{N_k-1})-v_k(\phi))$ must lie above
  $\Delta$. It follows that
  $
    m N_k < m(N_{k}-1)+q (v_k(a_{N_k-1})-v_k(\phi)),
$
  leading to the required inequality $v_k(a_{N_k-1}) > v_k(\phi)+m/q$.
  The lemma is proved. 
\end{proof}

\begin{prop}\label{prop:psi}
  We have $v_k(\Psi)=v_k(\Phi)$ and $\lambda_k(\Psi)=\lambda_k(\Phi)$
  for all $k=0,\ldots,g$.
\end{prop}

\begin{proof}
  We show this result by induction. If $k=0$, we are done since
  $\psi_0=\tau_F(y)=\phi_0$. Let us fix $1\le k\le g$ and assume that
  Proposition \ref{prop:psi} holds for all $k'\leq k-1$. We need to
  show that $v_k(\psi_i)=v_{k}(\phi_i)$ and
  $\lambda_k(\psi_i)=\lambda_{k}(\phi_i)$ for all $i\le k$. Case $i=k$
  is a direct consequence of Lemma \ref{lem:vpsiphi}. For $i=k-1$,
  there is nothing to prove if $\phi_{k-1}=\psi_{k-1}$. Otherwise,
  using the linearity of $\pi_{k-1}^*$, Corollary \ref{cor:main}
  (applied at rank $k-1$ with $G=\phi_{k-1}-\psi_{k-1}$ and $i=0$) and
  Lemma \ref{lem:vpsiphi} give
  $ \pi_{k-1}^*(\psi_{k-1})=\pi_{k-1}^*(\phi_{k-1})+x^{\alpha} \Ut $
  with $\alpha>v_{k-1}(\phi_{k-1})+m_{k}/q_{k}$ and
  $\Ut\in \Ki_{k-1}[[x,y]]^\times$. As
  $\pi_k^*(\psi_{k-1})=(\sigma_k\circ
  \tau_k)^*(\pi_{k-1}^*(\psi_{k-1}))$, it follows that
  \[
    \pi_k^*(\psi_{k-1})=\pi_k^*(\phi_{k-1})+x^{q_k\alpha}U_\alpha
  \]
  with $U_\alpha\in \Ki_k[[x,y]]^\times$. As
  $q_k\,\alpha>v_{k}(\phi_{k-1})$ using Lemma \ref{lem:vLambda}
  ($q_k\,v_{k-1,k-1}+m_k=v_{k,k-1}$), we deduce
  $v_k(\psi_{k-1})=v_k(\phi_{k-1})$ and
  $\lambda_k(\psi_{k-1})=\lambda_k(\phi_{k-1})$. Finally, for $i<k-1$,
  as $\deg(\psi_i)<d_{k-1}$, Corollary \ref{cor:main} (applied at rank
  $k-1$ with $G=\psi_i$ and $i=0$) gives
  \[
    \pi_{k-1}^*(\psi_i)=x^{v_{k-1}(\psi_i)}\lambda_{k-1}(\psi_i) U_i =
    x^{v_{k-1}(\phi_i)}\lambda_{k-1}(\phi_i) U_i ,
  \]
  where $U_i(0,0)=1$ (the second equality using the induction
  hypothesis). Applying $(\tau_k\circ\sigma_k)^*$ and using Lemma
  \ref{lem:vLambda}, we conclude $v_k(\psi_{i})=v_k(\phi_{i})$ and
  $\lambda_k(\psi_{i})=\lambda_k(\phi_{i})$.
\end{proof}

\begin{cor}\label{cor:gprime}
  Let $G$ of degree less than $d_k$ and with $\Psi$-adic expansion
  $G=\sum g'_B \Psi^B$. Then
  \[
    v_k(G)=\min (\langle B, V \rangle , g_B'\ne 0)\quad and \quad
    \lambda_k(G)=\sum_{B\in \Bc(0,v_k(G))} g'_B \Lambda^B.
  \]
  In particular, if $G$ has $\Phi$-adic expansion $\sum g_B \Phi^B$,
  then $g_B=g_B'$ when $\langle B, V \rangle=v_k(G)$.
\end{cor}

\begin{proof}
  As already shown in the proof of Proposition \ref{prop:psi}, from
  Corollary \ref{cor:main}, if $i<k$, we have
  $\pi_k^*(\psi_i)=x^{v_{k,i}}\lambda_{k,i}\,U_i$ with
  $U_i(0,0)=1$. As $\deg(G)< d_k$, we deduce
  \[
    \pi_k^*(G)=\sum g'_B \Lambda^B x^{\langle B, V \rangle} U_B
  \]
  with $U_B(0,0)=1$. This shows the result, using Proposition
  \ref{prop:key}.
\end{proof}

\paragraph{Proof of Theorem \ref{thm:HPsiPhi}.} Write
$F=\sum_i a_i\psi_k^i$ the $\psi_k$-adic expansion of
$F$. 
Similarly to (\ref{eq:ajphij}), when $a_i\neq 0$, Corollary
\ref{cor:main} and Lemma \ref{lem:vpsiphi} imply:
\begin{equation}
  \label{eq:ajpsij}
  \pi_k^*(a_i\psi_k^i)=x^{v_k(a_i\psi_k^i)}U (y + x^{\alpha} \Ut)^i,
\end{equation}
with $\alpha>m_{k+1}/q_{k+1}$, $U, \Ut\in \Ki_k[[x,y]]^\times$ and
$U(0,0)=\lambda_k(a_i\,\psi_k^i)$.  Applying the same argument than in
the proof of Lemma \ref{lem:vpsiphi}, we get that each point
$(i, w_i=N_k-i\,m_{k+1}/q_{k+1})$ of the lower edge $\NP_{k}^*$ of the
Newton polygon of $H_k$ (hence the all polygon if $k<g$) is actually
$(i,v_k(a_i\,\psi_k^i)-v_k(F))$, that is $(i,w_i')$ from Corollary
\ref{cor:gprime} (applied to $G=a_i$) and Proposition
\ref{prop:psi}. This shows that we may replace $w_i$ by $w_i'$ in
\eqref{eq:wj}. More precisely, it follows from \eqref{eq:ajpsij} that the restriction
$\bar{H}_{k}^*$ of $\bar{H}_k$ to $\NP_{k}^*$ is uniquely determined
by the equality
\[
\lambda_k(F)x^{v_k(F)}\bar{H}_{k}^* = \sum_{(i,w_i')\in \NP_{k}^*}
\lambda_k(a_i\psi_k^i)x^{v_k(a_i\psi_k^i)} y^i.
\]
Using again Corollary
\ref{cor:gprime} and Proposition \ref{prop:psi}, we get 
\[
\bar{H}_{k}^*  = \sum_{(i,w_i')\in \NP_{k}^*}
\left(\sum_{B\in\Bc(i,w_i'+v_k(F))} f_B' \Lambda^{B-B_0}\right) x^{w_i'} y^i,
\]
as required. \hfill$\square$

\begin{rem}
  Theorem \ref{thm:HPsiPhi} would still hold when replacing $\psi_k$ by any monic polynomial $\phi$ of same degree
  for which $\pi_k^*(\phi)=U x^{v_{k,k}}(y+\beta(x))$ with
  $\val(\beta)> m_{k+1}/q_{k+1}$. 
\end{rem}

\subsection{An Abhyankar type irreducibility test for Weierstrass
  polynomials}

Theorem \ref{thm:HPsiPhi} leads to the following sketch of algorithm.
Subroutines \AppRoot{}, \Expand{} and
\Edgepoly{} respectively compute the approximate roots, the
$\Psi$-adic expansion and the current lower \edgepoly{} (using
\eqref{eq:wj} and \eqref{eq:barHk}). They are detailed in Section
\ref{sec:comp}. Also, considerations about truncation bounds is
postponed to Section \ref{ssec:trunc}.

\begin{algorithm}[H]
  \nonl\TitleOfAlgo{\irreducible$(F,\Li)$\label{algo:irreducible}}%
  \KwIn{$F\in\Ki[[x]][y]$ monic with $\dy=\deg(F)$ not divisible by
    the characteristic of $\Ki$ ; $\Li$ a field extension of $\Ki$.}%
  \KwOut{\True{} if $F$ is irreducible in $\Li[[x]][y]$, and \False{}
    otherwise.}%
  $N\gets \dy$, $V\gets (1,0)$, $\Lambda\gets (1,1)$,
  $\Psi\gets(x)$\;%
  \While{$N > 1$}{%
    $\Psi\gets \Psi\cup \AppRoot{}(F,N)$\;%
    $\sum_{B} f_B \Psi^{B}\gets \Expand{}(F,\Psi)$\;%
    $\bar{H}\gets \Edgepoly{}(F,\Psi)$\;%
    \lIf{$\bar{H}$ is not degenerated over $\Li$}{\Return{\False}}%
    $(q,m,P,N)\gets \Edata(\bar{H})$\;%
    Update the lists $V,\Lambda$ thanks to formula
    \eqref{eq:update}\;%
    $\Li\gets \Li_P$ }%
  \Return{\True}
\end{algorithm}

\begin{thm}\label{thm:Irr}
  Algorithm \irreducible{} returns the correct answer.
\end{thm}

\begin{proof}
  This follows from Theorem \ref{thm:NPphi}, \ref{thm:EdgePoly}
  and \ref{thm:HPsiPhi}, together with the correctness of \NPA{}.
\end{proof}

Let us illustrate this algorithm on two simple examples.

\begin{xmp}\label{exKuo}
  Let $F(x,y)=(y^2-x^3)^2-x^7$. This example was suggested by Kuo who wondered if we could show that $F$ is
  reducible in $\algclos{\Qi}[[x]][y]$
  without performing Newton-Puiseux type tranforms. Abhyankhar solved this challenge in \cite{Ab89} thanks to approximate roots. Let us show that
  we can prove further that $F$ is reducible in $\Qi[[x]][y]$ without
  performing Newton-Puiseux type tranforms.

  \textit{Initialisation.} Start from $\psi_{-1}=x$, $N_0=\dy=4$,
  $V=(1,0)$ and $\NP=(1,1)$.

  \textit{Step k=0.} The $4$-th approximate root of $F$ is
  $\psi_{0}=y$. So $H_0=F$ and we deduce from \eqref{eq:barHk} (see
  Exemple \ref{xmp:k0}) that $\bar{H}_0=(y^2-x^3)^2$. Hence, $F$ is
  degenerated with edge data $(q_1,m_1,P_1,N_1)=(2,3,Z_1-1,2)$ and we
  update $V=(2,3,6)$ and $\Lambda=(1,1,2)$ thanks to
  \eqref{eq:update}, using here $z_1=1\mod P_1$.

  \textit{Step k=1.} The $2$-th approximate root of $F$ is
  $\psi_1=y^2-x^3$ and $F$ has $\Psi$-adic expansion
  $F=\psi_1^2-\psi_{-1}^7$. We have $v_1(\psi_1^2)=2v_{1,1}=12$,
  $\lambda_1(\psi_1^2)=\lambda_{1,1}^2=4$ while
  $v_1(\psi_{-1}^7)=7v_{-1,1}=14$ and
  $\lambda_1(\psi_{-1}^7)=\lambda_{-1,1}^7=1$. We deduce from
  \eqref{eq:barHk} that $\bar{H}_1=y^2-\frac{1}{4} x^2$. As the
  polynomial $Z_2^2-\frac{1}{4}$ is reducible in
  $\Qi_{P_1}[Z_2]=\Qi[Z_2]$, we deduce that $F$ is reducible in
  $\Qi[[x]][y]$.
\end{xmp}

\begin{xmp}\label{ex:2}
  Consider $F=((y^2-x^3)^2+4 x^8)^2+x^{14}(y^2-x^3)$ (we assume that
  we only know its expanded form at first).

  \textit{Initialisation.} We start with $\psi_{-1}=x$, $N_0=\dy=8$,
  $V=(1,0)$ and $\NP=(1,1)$.

  \textit{Step k=0.} The $8$-th approximate root of $F$ is
  $\psi_{0}=y$.  The monomials reaching the minimal values
  \eqref{eq:wj} in the $\Psi=(\psi_{-1},\psi_0)$-adic expansion
  of $F$ are
  $\psi_0^8$ $-4\psi_{-1}^3 \psi_0^6$, $6\psi_{-1}^6 \psi_0^4$
  ,$-4\psi_{-1}^9\psi_0^2$, $\psi_{-1}^{12}$
  and we deduce from \eqref{eq:barHk} that
  $\bar{H}_0=(y^2-x^3)^4$. Hence, $(q_1,m_1,P_1,N_1)=(2,3,Z_1-1,4)$
  and we update $V=(2,3,6)$ and $\Lambda=(1,1,2)$ thanks to
  \eqref{eq:update}, using here $z_1=1\mod P_1$.

  \textit{Step k=1.} The $4$-th approximate root of $F$ is
  $\psi_1=y^2-x^3$ and we get the current $\Psi$-adic expansion
  $F=\psi_1^4+8\psi_{-1}^8 \psi_1^2+\psi_{-1}^{14}
  \psi_1+16\psi_{-1}^{16}$.
  The monomials reaching the minimal values \eqref{eq:wj} are
  $\psi_1^4$, $8\psi_{-1}^8 \psi_1^2$, $16\psi_{-1}^{16}$ and we deduce
  from \eqref{eq:barHk} that $\bar{H}_1=(y^2+x^4)^2$. Hence
  $(q_2,m_2,P_2,N_2)=(1,2,Z_2^2+1,2)$ and we update $V=(2,3,8,16)$ and
  $\Lambda=(1,1,2z_2,8 z_2)$ thanks to \eqref{eq:update}, where
  $z_2=Z_2\mod P_2$ and using the B\'ezout relation $q_2 s_2-m_2 t_2=1$
  with $(s_2,t_2)=(1,0)$. Note that we know at this point that $F$ is
  reducible in $\algclos{\Qi}[[x]][y]$ since $P_2$ has two distinct roots in $\algclos{\Qi}$.

  \textit{Step k=2.} The $2$-th approximate roots of $F$ is
  $\psi_2=(y^2-x^3)^2+4x^8$ and we get the current $\Psi$-adic
  expansion $F=\psi_2^2+\psi_{-1}^{14}\psi_1$.  The monomials reaching
  the minimal values \eqref{eq:wj} are $\psi_2^2$, $\psi_{-1}^{14}\psi_1$
  and we deduce from \eqref{eq:barHk} that
  $\bar{H}_2=y^2+(32 z_2)^{-1} x$ (note that $z_2$ is invertible in
  $\Qi_{P_2}$).  Hence $\bar{H}_2$ is degenerated with edge data
  $(q_3,m_3,P_3,N_3)=(2,1,Z_3+(32 z_2)^{-1},1)$. As $N_3=1$, we deduce
  that $F$ is irreducible in $\Qi[[x]][y]$ ($g=3$ here).
\end{xmp}

\begin{rem}\label{rem}
  Note that for $k\ge 2$, we really need to consider the $\Psi$-adic
  expansion: the $(x,y,\psi_k)$-adic expansion is not enough to
  compute the next data. At step $k=2$ in the previous example, the
  $\psi_2$-adic expansion of $F$ is $F=\psi_2^2+a$ where
  $a=x^{14} y^2-x^{17}$. We need to compute $\val[2](a)$. Using the
  $\Psi$-adic expansion $a=\psi_{-1}^{14} \psi_1$, we find
  $\val[2](a)=14 \times 2+ 8 = 36$.  Considering the $(x,y)$-adic
  expansion of $a$ would have led to the wrong value
  $\val[2](x^{14} y^2)= \val[2](x^{17}) =34 < 36$. 
\end{rem}


\section{Absolute irreducibility}\label{sec:absolute}

We say that $F\in \Ki[[x]][y]$ is absolutely irreducible if it is
irreducible in $\overline{\Ki}[[x]][y]$, that is if
\irreducible($F,\overline{\Ki}$) returns \True{}. In particular, in
this context, we always have $\ell_k=1$. As already mentionned,
Abhyankhar's absolute irreducibility test avoids any Newton-Puiseux
type transforms or Hensel type liftings. In fact, it is even
stronger as it does not even compute the \edgepoly{}s $\bar{H}_k$, but
only their Newton polygon. Although we don't need this improvement
from a complexity point of view (see Subsection \ref{ssec:proofs}), we
show how to recover this result in our context for the sake of
completness. We will use the following alternative characterizations
of valuations and polygons:

\begin{lem}\label{lem:samevk}
  Suppose that $H_0,\ldots,H_{k-1}$ are degenerated.
  \begin{enumerate}
  \item Write $F=\sum c_i\psi_k^i$ the $\psi_k$-adic expansion of
    $F$. Then $\val[k](F)=\min_i \val[k](c_i\psi_k^i)$ and
    \begin{equation}\label{defpolygk}
      \NP_k(F):=\NPb(\pi_k^*(F))=\Conv\left(\left(i,\val[k](c_i\psi_k^i)\right)+(\Ri^+)^2, \, c_i\ne 0\right).
    \end{equation}
  \item Let $k\ge 1$ and $G\in \Ki[[x]][y]$ with $\psi_{k-1}$-adic
    expansion $G=\sum_i a_i\psi_{k-1}^i$. We have
    \begin{equation}\label{newdefvk}
      \val[k](G)=\min_i \left(q_k v_{k-1}(a_i\psi_{k-1}^i)+i m_k \right).
    \end{equation}
  \end{enumerate}
\end{lem}

\begin{proof}
  1. Equality \eqref{defpolygk} is a direct consequence of Corollary
  \ref{cor:gprime} with Theorems \ref{thm:NPphi} and
  \ref{thm:HPsiPhi}. Also, from \eqref{eq:ajpsij}, $\pi_k^*(c_i\psi_k^{i})$ has a term of lowest
  $x$-valuation of shape
  $ux^{\val[k](a_i\,\psi_k^i)}\,y^i$ for some $u\in \Ki_k^\times$ and it follows that $\val[k](F)=\min_i \val[k](c_i\psi_k^i)$, as required.

  2. By \eqref{eq:ajpsij} applied to rank $k-1$, we get $\pi_{k-1}^*(a_i\psi_{k-1}^i)=x^{v_{k-1}(a_i\psi_{k-1}^i)}U_i(y+x^{\alpha}\tilde{U}_i)$, where $\alpha> m_k/q_k$, and $U_i,\tilde{U}_i$ are units. Suppose $m_k>0$. Then $V_i=U_i(z_k^{s_k}x^{q_k},x^{m_k}(y+z_k^{t_k}+c_k(x))$ is a unit such that $V_i(0,y)=U_i(0,0)\in \Ki_k^*$ is constant and
  a straightforward computation shows that $\pi_{k}^*(a_i\psi_{k-1}^i)=x^{q_k v_{k-1}(a_i\psi_{k-1}^i)+i m_k}P_i(y)+h.o.t$, where $P_i\in \Ki[y]$ has degree exactly $i$. Equality \eqref{newdefvk} follows. The case $m_k=0$ may occur only when $k=1$, $q_1=1$. In such a case, we have $\pi_1^*(G)=\sum_i a_i(x)(y+z_1+c(x))^i$ with $\val(c)>0$ and the same conclusion holds. 
\end{proof}

\begin{rem}\label{rem:refMontes}
Point 2 in Lemma \ref{lem:samevk} shows that our valuations coincide with the extended valuations used in Montes algorithm over general local fields, see for instance \cite[point (3) of Proposition 2.7]{GuMoNa12}.
\end{rem}

Hence, we may take \eqref{defpolygk} and \eqref{newdefvk} as
alternative recursive definitions of valuations and Newton
polygons. This new point of view has the great advantage to be
independent of the map $\pi_k$, hence of the Newton-Puiseux algorithm. In particular, it can be generalized
at rank $k+1$ without assuming that $H_k$ is degenerated.

\begin{dfn}\label{def:alternative}
  Suppose that $H_0,\ldots,H_{k-1}$ are degenerated and let $-m_{k+1}/q_{k+1}$ be the slope of the
  lowest edge of $H_k$. We still define the valuation $v_{k+1}$ and
  the Newton polygon $\NP_{k+1}(F)$ by formulas \eqref{newdefvk}
  and \eqref{defpolygk} applied at rank $k+1$.
\end{dfn}

We obtain the following absolute irreducibility test which only
depends on the geometry of the successive Newton polygons.

\begin{algorithm}[H]
  \nonl\TitleOfAlgo{\AbhyankarMoh($F$)\label{algo:AbMoh}}%
  \KwIn{$F\in\Ki[[x]][y]$ Weierstrass s.t. \Char($\Ki$) does not divide $\dy=\deg(F)$.}%
  \KwOut{\True{} if $F$ is irreducible in $\overline{\Ki}[[x]][y]$,
    \False{} otherwise.}%
  $N\gets \dy$, $v=\val$\;%
  \While{$N > 1$}{
    $\psi\assign \AppRoot{}(F,N)$\;%
    $\sum c_i\psi^i \assign \Expand{}(F,\psi)$\;%
    Compute the current polygon $\NP(F)$ with \eqref{defpolygk}\;%
    \If{$(N,v(F))\notin \NP(F)$ or $\NP(F)$ is not straight or
      $q=1$}{%
      \Return{\False}%
    }%
    $N\assign N/q$\;%
    Update $v$ with \eqref{newdefvk}\;%
  }%
  \Return{\True}\;%
\end{algorithm}

\begin{prop}\label{prop:abhyankhar}
  Algorithm \AbhyankarMoh{} works as specified.
\end{prop}

\begin{proof}
  Suppose that $F$ is not absolutely irreducible. Let us abusively
  still denote by $g$ be the first index $k$ such that $H_k$ is not
  degenerated over $\overline{\Ki}$ or $N_k=1$: so both algorithms
  \AbhyankarMoh{}($F$) and \irreducible ($F,\overline{\Ki}$) compute
  the same data $\psi_0,\ldots,\psi_{g-1}$ and
  $(q_1,N_1),\ldots,(q_g,N_g)$. If $N_g=1$, then $F$ is absolutely
  irreducible, and both algorithms return \True{} as required. If
  $N_g>1$, then \irreducible ($F,\overline{\Ki}$) returns \False. Note
  that $\NP(H_g)$ equals $\NP_g(F)-(0,v_g(F))$ by definition of
  $\NP_g$. As $H_g$ is Weierstrass of degree $N_g$, we have
  $(N_g,v_g(F))\in \NP_g(F)$ at this stage. If $\NP_g(F)$ is not
  straight or $q_{g+1}=1$, then so does $\NP(H_g)$
  and \AbhyankarMoh{}($F$) returns \False{} as required. There remains
  to treat the case where $\NP_g(F)$ is straight with $q_{g+1}>1$
  (still assuming $N_g>1$ and $H_g$ not degenerated over
  $\overline{\Ki}$). In such a case, \AbhyankarMoh{}($F$) computes the
  next $N_{g+1}^{th}$ approximate roots $\psi_{g+1}$ of $F$ where
  $N_{g+1}=N_g/q_{g+1}$. We will show that
  $(N_{g+1},v_{g+1}(F))\notin \NP_{g+1}(F)$ so that \AbhyankarMoh{}
  returns \False{} at this step.

  Let $F=\sum_{i=0}^{N_{g+1}} c_i \psi_{g+1}^{i}$ be the
  $\psi_{g+1}$-adic expansion of $F$. By hypothesis, we know that
  \[
    \pi_{g}^*(F)=x^{v_g(F)}\, H_g\, U,\text{ with } U(0,0)\ne 0
  \]
  where $\bar{H}_g=\prod_{Q(\zeta)=0}(y^{q_{g+1}}-\zeta x^{m_{g+1}})$,
  with $Q\in \Ki[Z]$ of degree $N_{g+1}:=N_{g}/q_{g+1}$ having at
  least two distinct roots. In particular, $\bar{H}_g$ is not the $N_{g+1}$-power of a
  polynomial and it follows that $\pi_g^*(\psi_{g+1}^{N_{g+1}})$ and
  $\pi_g^*(F)$ can not have the same \edgepoly{}s. We deduce that
  there is at least one index $i<N_{g+1}$ such that
  $\NP_g(c_i \psi_{g+1}^{i})$ has a point on or below
  $\NP_g(F)$. Consider the $\psi_g$-adic expansions
  $c_i \psi_{g+1}^{i}=\sum_{j} a_{j} \psi_g^j$ and
  $F=\sum_j \alpha_j\psi_g^j$. Thanks to \eqref{defpolygk}, there
  exists at least one index $j$ such that
  $(j,v_g(a_{j} \psi_g^j))\in \NP_g(c_i \psi_{g+1}^i)$. By \eqref{defpolygk},
  $\NP_g(F)$ is the lower convex hull of
  $(j,v_g(\alpha_{j} \psi_g^j))$, which is by assumption straight of
  slope $-q_{g+1}/m_{g+1}$. It follows that
  \[
    \min_{j} (q_{g+1} v_g(a_j\psi_g^j)+m_{g+1} j)\le \min_j (q_{g+1}
    v_g(\alpha_j\psi_g^j)+m_{g+1} j).
  \]
  Thanks to Definition \ref{def:alternative}, this implies
  $v_{g+1}(c_i \psi_{g+1}^{i})\le v_{g+1}(F)$ which in turns forces
  $(N_{g+1},v_{g+1}(F))\notin \NP_{g+1}(F)$.
\end{proof}

\begin{rem}At step $k+1$ of the algorithm, we know that
  $H_{0},\ldots,H_{k-1}$ are degenerated over $\overline{\Ki}$.  Hence
  the recursive definition of the map $v_{k+1}$ is equivalent to
  \begin{equation}\label{eq:vk+1G}
    v_{k+1}(G)=\min_{g_B\ne 0} \left(q_{k+1}\langle B, V \rangle +m_{k+1}
      b_k\right)
  \end{equation}
  where $G$ has $(\psi_{-1},\ldots,\psi_k)$-adic expansion
  $G=\sum g_B\Psi^B$ and $V=(v_{k,-1},\ldots, v_{k,k})$. This is the
  approach we shall use in practice for valuations updates. 
\end{rem}


\section{Pseudo-irreducibility}\label{ssec:PseudoIrr}

As mentionned in the introduction, performing too many irreducibility
tests might be costly. We therefore relax the degeneracy condition by
allowing square-freeness of the involved residual polynomial
$P_1,\ldots,P_g$, and eventually check if $\Ki_g$ is a field. This
leads to what we call a pseudo-irreducibility test. Despite of its complexity interest, we
will show in Section \ref{sec:equising} that this modification allows to
characterise balanced polynomials, thus proving Theorem
\ref{thm:main2}.

If we allow the $P_k$'s to be square-free, the fields $\Ki_k$'s become
ring extensions of $\Ki$ isomorphic to a direct product of fields and
we have to take care of zero divisors. Let
$\Ai=\Li_0\oplus \cdots \oplus \Li_r$ be a direct product of
perfect fields. We say that a (possibly multivariate) polynomial $H$ defined
over $\Ai$ is \textit{square-free} if all its projections under the
natural morphisms $\Ai\to \Li_i$ are square-free (in the usual sense
over a field). If the polynomial is univariate and monic, this exactly
means that its discriminant is not a zero divisor in $\Ai$.

In the following, we call the lower \edgepoly{} of $F$ the restriction of $F$ to the lower
edge of its Newton polygon, that we abusively still denote by $\bar{F}$. 

\begin{dfn}\label{def:pseudodeg}
  We say that a monic polynomial $F\in \Ai[[x]][y]$ is
  pseudo-degenerated if its \emph{lower \edgepoly{}} is the power
  of a \emph{square-free} quasi-homogeneous polynomial of shape 
  \begin{equation}\label{eq:pseudoquasihom}
  \bar{F}=\left(P\left(\frac{y^q}{x^{m}}\right)x^{m\deg(P)}\right)^N,
\end{equation}
with $P\in \Ai[Z]$ square-free and $P(0)\in \Ai^\times$ if $q>1$. 
\end{dfn}

We still call $P$ the \emph{residual polynomial} of $F$
and $(q,m,N,P)$ the \emph{edge data} of $F$ (with convention
$(q,m)=(1,0)$ if the Newton polygon is reduced to a point).

\begin{rem}\label{rem:P}
  If $q>1$, then $\NP(F)$ is straight, and Definition
  \ref{def:pseudodeg} is the analoguous of Definition
  \ref{dfn:degeneracy} of degenerated polynomials (square-freeness
  replacing irreducibility). However, in contrast to degenerated
  polynomials, we authorize here $P(0)=0$ (or more generally a
  zero-divisor) if $q=1$. In such a case, $\NP(F)$ may have several
  edges. For instance, $F=(y-x)(y-x^2)$ has two edges but is pseudo-degenerated (we get as $\bar{F}=y^2-xy$ and $P=Z^2-Z$). A more complicated example is $F=(y^2-x^2)^2(y-x^2)(y-x^3)+x^{10}$ which has three
  edges but is pseudo-degenerated (we get $\bar{F}=(y^3-x^2y)^2$ and $P=Z^3-Z$). Although having several edges implies reducibility, this definition will make sense when considering balanced polynomials (Remark
  \ref{rem:deg} and Example \ref{ex1}).
\end{rem}

\begin{dfn}\label{def:pseudo}
  We call \PNPA{} and \PIrr{} the new
  algorithms obtained when replacing degenerated tests by
  pseudo-degenerated tests respectively in algorithms \NPA{} and
  \irreducible.
\end{dfn}

\begin{rem} Note that any assertions in previous sections of type
  $a \ne 0$ still has to be read as such ($a$ is non zero), while any
  assertion of type $a\in \Ki_k^\times$ still has to be read as such
  (meaning now $a$ is not a zero divisor). In particular, given
  $a\in\Ai[[x]]$, $\val(a)$ is computed via the smallest monomial with
  \emph{non-zero} coefficient (and not ``non zero divisor''). This
  remark also applies to formula \eqref{eq:wj}. 
\end{rem}

\begin{prop}\label{prop:pseudobordel}
  Algorithms \PIrr{} and \PNPA{} are
  well-defined. Moreover, they give the same output and compute the
  same edges data. We say that a monic polynomial $F\in \Ki[[x]][y]$ is
  pseudo-irreducible if this output is \True{}.
\end{prop}

\begin{proof}
  We need to show that both algorithms are well-defined and that all
  results of Section \ref{sec:phi} and Section \ref{sec:psi} still
  hold when considering pseudo-degeneracy.  We have to take care of
  the fact that $z_k$ might be now a zero divisor. It is however
  sufficient to prove that $\lambda_{k,k}\in \Ki_k^\times$. Indeed,
  then the computation of the Weierstrass polynomial $H_k$ is
  possible, the statement of Theorem \ref{thm:EdgePoly} is correct,
  and so is the proof of Proposition \ref{prop:key}. This also implies
  that the functorial properties $v_k(a\phi_k^j)=v_k(a)+jv_k(\phi_k)$
  and $\lambda_k(a\phi_k^j)=\lambda_k(a)\lambda_k(\phi_k)^j$ hold for
  all $a\in \Ki[[x]][y]$. As $v_k$ still obeys to the triangular inequality, all these
  implications mean that all results of Sections \ref{sec:psi} hold too.

  We now prove that $\lambda_{k,k}\in\Ki_k^\times$ by induction. The claim
  is obvious when $k=0$. Now, let $k>0$ and assume that
  $\lambda_{i,i}\in\Ki_i^\times$ for $i<k$. If $z_k\in\Ki_k^\times$,
  then we are done. Otherwise, we must have $q_k=1$ and $t_k=0$ so
  that the definition of $\phi_k$ makes sense by \eqref{eq:tausigma}, \eqref{eq:pikxy} and \eqref{eq:phik} (we use $\mu_k\in \Ki_k^\times$) and Lemma \ref{lem:PikPhi} and
\ref{lem:vLambda} hold up to rank $k$. The proof of Proposition \ref{prop:key} remains valid (we use
  $n=\lfloor w t_k/q_k\rfloor =0$ in that case) and the proof of
  Proposition \ref{thmvkk} remains valid too (we use $P_k'(z_k)\in \Ki_k^\times$ since by assumption $P_k$ is square-free over a product of perfect fields). This implies
  $\lambda_{k,k}\in \Ki_k^\times$ from Remark
  \ref{rem:zkdiv0}. Finally, note that some splittings might appear
  during the algorithm (see Example \ref{xmp:splitting}), not changing
  the output of the algorithm (see \cite[Section 5]{PoWe17} for
  details).
\end{proof}

\begin{rem}\label{rem:z1}
  Note that $\lambda_{k,i}$ might be a zero divisor when $i<k$. For
  instance, the polynomial $F=(y^3-x^2 y)^N+\cdots$ is
  pseudo-degenerated with residual polynomial $P_1=Z^3-Z$, so that
  $z_1:=Z \mod P_1\in \Ki_1$ is not invertible.  We compute
  $\pi_1^*(\phi_0)=\pi_1^*(y)=x(y+z_1)$ from which it follows that
  $\lambda_{1,0}=z_1$ is a zero-divisor. The key point is that the families
  $\Lambda^B, B\in \Bc(w,j)$ remain free over $\Ki$.
\end{rem}

\begin{rem}\label{phikresultant}
  The polynomials $\phi_k$ are no longer irreducible (nor the
  $\psi_k$'s) when considering algorithm \PNPA{}, but only
  pseudo-irreducible. However, they still obey to
  equalities $\deg(\phi_k)=e_k f_k$ and
  $\pi_k^*(\phi_k)=x^{v_{k,k}} U_{k,k}(x,y)\, y$, with
  $U_k(0,0)=\lambda_{k,k} \in \Ki_k^{\times}$.
\end{rem}

\begin{cor}\label{cor:pseudoIrrvsIrr}
  A square-free monic polynomial $F\in \Ki[[x]][y]$ is irreducible over
  $\Ki$ if and only if it is pseudo-irreducible and $\Ki_g$ is a
  field.
\end{cor}

\begin{proof} This follows immediately 
from Definition \ref{def:pseudodeg} and Proposition
\ref{prop:pseudobordel}.
\end{proof}  

We will check irreducibility using Corollary \ref{cor:pseudoIrrvsIrr},
thus avoiding to perform too many univariate irreducibility tests.
Besides this advantage, testing pseudo-degeneracy will
allow us to characterize a larger
class than irreducible polynomials in $\Ki[[x]][y]$, namely the class
of balanced polynomials. In particular, if $F$ is pseudo-irreducible, we can compute easily the
ramification index and the residual degrees of all its irreducible
factors in $\Ki[[X]][Y]$ (see Example
\ref{xmp:splitting}), and the characteristic exponents and pairwise intersection multiplicities of all its absolutely irreducible factors (see Examples \ref{ex1}, \ref{ex2}, \ref{ex:nonweierstrass}).

\section{Complexity. Proof of Theorems \ref{thm:main} and
  \ref{thm:absolute}}\label{sec:comp}

\subsection{Complexity model}
\label{ssec:model}
We use the algebraic RAM model of Kaltofen \cite[Section 2]{Ka88},
counting only the number of arithmetic operations in our base field
$\Ki$.  Most subroutines are deterministic; for them, we consider the
worst case. However, computation of primitive elements in residue
fields uses a probabilistic algorithm of Las Vegas type, and we
consider then the average running time. We denote by $\M(d)$ the
number of arithmetic operations for multiplying two polynomials of
degree $d$. We use fast multiplication, so that $\M(d)\in \Ot(d)$ and
$d'\M(d)\le \M(d'd)$, see \cite[Section 8.3]{GaGe13}. We denote by
$\I(d)$ the number of arithmetic operations for testing irreducibility
of a degree $d$ polynomial over $\Ki$. We assume that $d\in \I(d)$ and
$d'\I(d)\le I(dd')$, which is consistent with the known bounds for
$\I(d)$ (see e.g. \cite[Theorem 14.37]{GaGe13} for $\Ki=\Fi_q$ and
\cite[Theorem 15.5]{GaGe13} for $\Ki=\Qi$). We use the classical
notations $\O()$ and $\Ot()$ that respectively hide constant and
logarithmic factors (\cite[Chapter 25, Section 7]{GaGe13}). In
particular, we will abusively denote $\Ot(\vF)$ a complexity result as
$\vF\log(\dy)$ (which is bounded by $\vF\log(\vF)$ only when $F$ is
Weierstrass).  

\paragraph{Primitive representation of residue rings.}
The $\Ki$-algebra $\Ki_k$ is given inductively as a tower extension of
$\Ki$ defined by the radical triangular ideal
$(P_1(Z_1),\ldots,P_k(Z_1,\ldots,Z_k))$. It turns out that such a
representation does not allow to reduce a basic operation in $\Ki_k$
to $\Ot(f_k)$ operations over $\Ki$ (see \cite{PoWe17} for
details). To solve this problem, we compute a primitive representation
of $\Ki_k$, introducing the notation $\Ki_Q:=\Ki[T]/(Q(T))$.

\begin{prop}\label{prop:Prim}
  Let $Q\in \Ki[T]$ and $P\in \Ki_Q[Z]$ square-free, and assume that
  $\Ki$ has at least $(\deg_T(Q)\,\deg_Z(P))^2$ elements.  There
  exists a Las Vegas algorithm \Primitive{} that returns $(Q_1,\tau)$
  with $Q_1\in \Ki[W]$ square-free and
  $\tau:\Ki[T,Z]/(Q,P) \rightarrow \Ki[W]/(Q_1)$ an isomorphism. It
  takes an expected $\O((\deg_T(Q)\,\deg_Z(P))^{(\omega+1)/2})$
  operations over $\Ki$. Given $\alpha\in\Ki[T,Z]/(Q,P)$, one can
  compute $\tau(\alpha)$ in less than $\Ot(\deg_T(Q)^2\,\deg_Z(P))$.
\end{prop}

\begin{proof}
  Use \cite[Proposition 15]{PoWe17} with
  $I=(Z_1,Q(Z_2))$ (see notations therein).
\end{proof}

In the following, we use that an operation in
$\Ki_k$ costs $\Ot(f_k)$ operations in $\Ki$.
\begin{rem}
  Another way to deal with tower extensions would be the recent
  preprint \cite{HoLe18}. This would make all algorithms
  deterministic, with a cost $\O(\vF^{1+o(1)})$ instead of $\Ot(\vF)$.
  Note also \cite{HoLe19} for dynamic evaluation. 
\end{rem}

\subsection{Truncation bounds}
\label{ssec:trunc}

In order to estimate the complexity in terms of arithmetic operations
in $\Ki$, we will compute approximate roots and $\Psi$-adic expansions
modulo a suitable truncation bound for the powers of $\psi_{-1}=x$. We
show here that the required sharp precision is the same than the one
obtained in \cite[Section 3]{PoWe17} for the Newton-Puiseux type
algorithm. Note also \cite[Theorem 2.3, page 144]{BaNaSt13} that
provides similar results in the context of irreducibility test. In the
following, when we say that we truncate a polynomial with precision
$\tau\in\Qi$, we mean that we keep only powers of $X$ less or equal
than $\tau$.

The successive polynomials generated by \PNPA{}($F$) are still
denoted $H_0,\ldots,H_g$, and we let $(q_{g+1},m_{g+1})$ stand for the
slope of the lower edge of $H_g$ ($(q_{g+1},m_{g+1}):=(1,0)$ if
$N_g=1$). As $\deg(H_k)=N_k$ and $\NP(H_k)$ has a lower edge of slope
$-m_{k+1}/q_{k+1}$, the computation of the lower \edgepoly{}
$\bar{H}_k$ only depends on $H_k$ truncated with precision
$N_k m_{k+1}/q_{k+1}$.  Combined with \eqref{eq:pikHk}, and using
$\val(\pi_k^*(x))=e_k$, we deduce that the $k^{th}$-edge data only
depends on $F$ truncated with precision
\begin{equation}\label{eq:sigmak}
  \eta_k:=\frac{v_k(F)}{e_k}+ N_k \frac{m_{k+1}}{e_{k+1}}.
\end{equation}
Denoting $\displaystyle{} \eta(F):=\max_{0\leq k\leq g}(\eta_k)$, we
deduce that running \PIrr{} modulo $x^{\eta(F)+1}$ return the correct
answer, this bound being sharp by construction.

\begin{lem}\label{lem:sigmaF}
  We have $\eta_{k}=\eta_{k-1}+\frac{N_k m_{k+1}}{e_{k+1}}$. In
  particular,
  $\eta (F)= \eta_g=\sum_{k=1}^{g+1} \frac{N_{k-1} m_k}{e_k}$.
\end{lem}

\begin{proof}
  As $v_k(F)=N_k\, v_{k,k}$ from Lemma \ref{lem:FPhi}, we get
  \begin{equation}\label{eq:sigmak2}
    \eta_k=\frac{N_{k} v_{k,k}}{e_{k}}+\frac{N_{k} m_{k+1}}{e_{k+1}}
  \end{equation}
  for all $0\leq k\leq g$. If $k=0$, we have $\eta_0=N_0\, m_1/q_1$ as
  required, as $v_{0,0}=0$. Suppose $k\ge 1$. Applying
  \eqref{eq:sigmak2} at rank $k-1$, we obtain for $k=1,\ldots, g$ the
  relations
  \begin{equation}\label{eq:sigmak3}
    \eta_{k-1}=\frac{N_{k-1} v_{k,k-1}}{e_{k}}=\frac{N_{k} v_{k,k}}{e_{k}},
  \end{equation}
  first equality using Point 1 of Lemma \ref{lem:vLambda}
  ($v_{k,k-1}=q_k v_{k-1,k-1}+m_k$) and second equality using
  $N_{k-1}= q_{k} \ell_{k} N_{k}$ and equality
  $v_{k,k}=q_{k} \ell_{k} v_{k,k-1}$ of Theorem \ref{thmvkk}.  Hence
  \eqref{eq:sigmak2} and \eqref{eq:sigmak3} give
  $\eta_k = \eta_{k-1}+\frac{N_k m_{k+1}}{e_{k+1}}$ as required. The
  formula for $\eta(F)$ follows straightforwardly.
\end{proof}

\begin{rem} We have the formula
  $\eta_k
  =\frac{\val(\pi_k^*F(x,0))}{e_{k+1}}=\frac{(F,\phi_k)_0}{d_k}$
  for $k\le g-1$, the first equality following again from
  \eqref{eq:pikHk} and the second equality from Corollary
  \ref{cor:res}. Since $(F,\phi_k)_0=(F,\psi_k)_0$, we deduce in
  particular that the sequence of integers
  $(N_0,d_0\eta_0,\ldots,d_{g-1} \eta_{g-1})$ form a minimal set of
  generators of the semi-group of $F$ when $F$ is irreducible in
  $\algclos{\Ki}[[x]][y]$ ; see e.g. \cite[Proposition 4.2 and Theorem
  5.1]{Pop02}.
\end{rem}

\begin{prop}\label{prop:sigma}
  Let $F\in\Ki[[x]][y]$ be monic and separable of degree $\dy$, with
  discriminant valuation $\vF$. Then $\eta(F) \le \frac{2\vF}{\dy}$.
  If moreover $F$ is pseudo-irreducible, then $\eta(F)\ge \vF/\dy$.
\end{prop}

\begin{proof}
  It follows from Lemma \ref{lem:sigmaF} that $\eta(F)$ is smaller or
  equal than the quantity ``$N_i$'' defined in \cite[Subsection
  3.3]{PoWe17} (take care of notations, these $N_i$ are not the same
  as those defined here), with equality if $F$ is
  pseudo-irreducible. From \cite[Corollary 4]{PoWe17}, we deduce
  $\eta(F) \le 2 v_i$ for $i=1,\ldots,\dy$, where
  $v_i:=\val(\partial_y F(y_i))$, $y_i$ denoting the roots of $F$. As
  $\vF=\sum v_i$, we have $\min v_i\le \vF/\dy$ and the upper
  bound for $\eta(F)$ follows. If $F$ is pseudo-irreducible, then we
  have also $v_i \le \eta(F)=N_i$ by \cite[Corollary 4]{PoWe17}. As
  all $v_i$'s are equal in that case, the lower bound follows too.
\end{proof}

\begin{rem}[\emph{Dealing with the precision}]\label{rem:precision} As
  $\vF$ is not given, we do not have an \textit{a piori} bound for the
  precision $\eta(F)$. To deal with this problem, we start from some low
  precision, and double it each time the computed lower edge of the Newton polygon is
  not "guaranted" \cite[Definition 8 and Figure 1.b]{PoWe17}, which can be checked thanks to Lemma \ref{lem:sigmaF} (lines \ref{Pirr:etaprim} and \ref{Pirr:callback} of algorithm \PIrr{} below). We could use also relaxed computations \cite{Ho02}. In both solutions, this only multiply the complexity
  result by at most a logarithm factor.
\end{rem}

\subsection{Main subroutines}
\label{ssec:subroutines}

\paragraph{Computing approximate roots and $\Psi$-adic expansion.}
\begin{prop}\label{prop:approx}
  There exists an algorithm $\AppRoot{}$ which given
  $F\in \Ai[y]$ a degree $d$ monic polynomial defined over a ring of
  characteristic not dividing $d$ and given $N$ which divides $d$,
  returns the $N^{th}$ approximate root of $F$ with $\M(d)$ operations
  over $\Ai$.
\end{prop}

\begin{proof}
  $\psi$ can be computed as follows. Let
  $G=y^{d}F(1/y)$ be the reciprocal polynomial of $F$. So $G(0)=1$ and
  there exists a unique series $S\in \Ai[[y]]$ such that $S(0)=1$ and
  $G=S^N$. Then $\psi$ is the reciprocal polynomial of the truncated
  series $\tronc{S}{\frac d N}$ (see e.g. \cite[Proposition 3.4]{Pop02}). The
  serie $S$ is solution of the equation $Z^N-G=0$ in $\Ai[[y]][Z]$ and
  can be computed up to an arbitrary precision $\sigma$ with
  $\M(N\sigma)$ operations by quadratic Newton iteration \cite[Theorem
  9.25]{GaGe13}, hence $\M(d)$ operations with $\sigma=d/N$.
\end{proof}

\begin{prop}\label{prop:expand}
  There exists an algorihm \Expand{} which, given $F\in \Ai[y]$
  of degree $d$ and $\Psi=(\psi_{0},\ldots,\psi_k)$ a
  collection of monic polynomials $\psi_i\in \Ai[y]$ of strictly
  increazing degrees $d_0<\cdots < d_k\le d$ returns the reduced
  $\Psi$-adic expansion of $F$ within $\O((k+1)\M(d)\log(d))$ arithmetic
  operations over $\Ai$.
\end{prop}

\begin{proof}
  The $\psi_k$-adic expansion of $F=\sum a_i \psi_k^i$ requires
  $\O(\M(d)\log(d))$ operations by \cite[Theorem 9.15]{GaGe13}. If $k>0$,
  we recursively compute the $(\psi_0,\ldots,\phi_{k-1})$-adic
  expansion of $a_i$ in $\O(k\M(\deg\,a_i)\log(\deg\,a_i))$
  operations. Since $\deg(a_i)< d_k$, summing over all
  $i=0,\ldots,\lfloor d/d_k\rfloor $ gives $\O(k\M(d)\log(d))$
  operations.
\end{proof}

\paragraph{Computing \edgepoly{}s.}
\begin{prop}\label{prop:computeHbark}
  Given $F$ and $\Psi=(\psi_{-1},\ldots,\psi_{k})$ modulo
  $x^{\eta(F)+1}$, $V=(v_{k,-1},\ldots,v_{k,k})$ and
  $\Lambda=(\lambda_{k,-1},\ldots,\lambda_{k,k})$, there exists an
  algorithm \Edgepoly{} that computes the lower \edgepoly{}
  $\bar{H}_k\in \Ki_k[x,y]$ within $\Ot(\vF+f_k^2)$ operations over
  $\Ki$.
\end{prop}

\begin{proof}
  First compute the $\Psi$-adic expansion $F=\sum f_B \Psi^B$ modulo
  $x^{\eta+1}$, with $\eta:=\eta(F)$. As $\eta\le 2\vF/\dy$, this is
  $\Ot(\vF)$ by Proposition \ref{prop:expand} applied with
  $\Ai=\Ki[x]/(x^{\eta+1})$.  The computation of the lower edge of
  $\bar{H}_k$ is done with Theorem \ref{thm:EdgePoly} and take no
  arithmetic operations (this takes $\Ot(\delta)$ bit operations as $\langle B,V\rangle \in \O(\delta)$ and there are at most $e_kf_kN_k\eta\le 2\delta$ such scalar products to compute). It remains to compute the coefficient of each
  monomial $x^{w_i} y^i$ of $\bar{H}_k$, which is (Theorem
  \ref{thm:EdgePoly}):
  \[
  c_{k,i}:=\sum_{B\in \Bc(i,w_i+v_k(F))} f_B \Lambda^{B-B_0}.
  \]
  Note first that computing $\Lambda^{B_0}=\lambda_{k,k}^{N_k}$ takes
  $\O(\log(d))$ operations over $\Ki_k$ via fast exponentiation.
  Then, there are at most $f_k$ monomials $\Lambda^B$ to compute from
  Proposition \ref{prop:key}. Each of them can be computed in
  $\O(k\,\log(\vF))$ operations in $\Ki_k$ via fast exponentiation on each
  $\lambda_{k,i}$ (we have
  $w_i\le \val(H_k(x,0))= N_{k}m_{k+1}/q_{k+1}$, thus
  $w_i+v_k(F)\le e_k\eta_k\le 2\vF$ from \eqref{eq:sigmak} and
  Proposition \ref{prop:sigma}). This concludes.
\end{proof}

\paragraph{Testing pseudo-degeneracy and computing edge data.}
There remains to test the pseudo-degeneracy of the \edgepoly{}s.
\begin{prop}\label{prop:pseudodeg}
  Given $Q\in \Ki[Z]$ square-free and $\bar{H}\in \Ki_Q[x,y]$ monic in
  $y$ and quasi-homogeneous, there exists an algorithm \PDegenerated{}
  that returns \False{} if $\bar{H}$ is not pseudo-degenerated, and
  the edge data $(q,m,P,N)$ of $\bar{H}$ otherwise. It takes at most
  $\Ot(\deg_Z(Q) \deg(H))$ operations over $\Ki$.
\end{prop}

\begin{proof} 
  As $\bar{H}$ is quasi-homogeneous, we have
  $\bar{H}=y^rP_0(y^q/x^m)x^{m\deg(P_0)}$ for some coprime integers
  $q,m\in \Ni$, $0\le r<q$ and some $P_0\in \Ki_{Q}[T]$ which can be computed with the aimed complexityIf $r\ne 0$, then $H$ is not pseudo-degenerated. Otherwise, there remains to
  check if $P_0=P^{N}$ for some $N\in \Ni$ and $P\in\Ki_Q[T]$
  square-free $($i.e $(Q,P)$ radical ideal in $\Ki[Z,T])$, and that
  $P(0)\notin\Ki_Q^\times$ when $q>1$.  The first task is a special
  case of \cite[Proposition 14]{PoWe17} and fits in the aimed
  bound. Second one is just a gcd computation, bounded by
  $\Ot(\deg_Z(Q))$.
\end{proof}
  
\begin{rem}\label{rem:sqrfree}
  We might discover that $Q$ factors and perform some splittings of the ring $\Ki_Q$ in course of the square-free test (see Example \ref{xmp:splitting}). In such a case, we do not necessarily return \False{}: although this is the natural
  option when testing irreducibility, we don't want to stop the
  algorithm if $Q$ factors when testing balancedness (see
  Section \ref{sec:equising}). 
\end{rem}

\subsection{The main algorithm. Proofs of Theorems \ref{thm:main}, \ref{thm:absolute} 
and \ref{thm:nonsqrfree}}\label{ssec:proofs}


\begin{algorithm}[H]
  \nonl\TitleOfAlgo{\PIrr{}($F,\eta=1$)\label{algo:PIrr}}%
  \KwIn{$F\in\Ki[[x]][y]$ monic of degree $\dy$ not divisible by the characteristic of $\Ki$}%
  \KwOut{\False{} if $F$ is not pseudo-irreducible, and $(\D,Q)$ otherwise, with $\D$ the edges data of $F$ and $\Ki_g=\Ki_Q$.}%
  $F\gets F\mod x^{\eta}$ \tcp*{All computations modulo $x^\eta$}%
  $N\gets \dy$, $V\gets [1,0]$, $\Lambda\gets [1,1]$, $\Psi\gets[x]$, $Q\gets Z$, $(e,\eta')\gets (1,0)$, $\D\gets[\,]$\;%
  \While{$N > 1$}{%
    $\Psi\gets \Psi\cup \AppRoot{}(F,N)$\label{Pirr:approx}\;%
    $\bar H\gets \Edgepoly{}(F,\Psi,V,\lambda)$ \label{Pirr:edge}\tcp*{$\bar H\in \Ki_Q[x,y]$}%
    $e\gets q\,e$ ; $\eta'\gets \eta'+ \frac{N\, m}{e}$ \label{Pirr:etaprim}\tcp*{$(q,m)$ lower edge of $\bar H$}%
    \lIf{$\eta \le\eta'$}{\Return \PIrr{}($F,2\eta$)}\label{Pirr:callback}%
    $(Bool,(q,m,P,N))\gets$ \PDegenerated$(\bar H,Q)$ \label{Pirr:degen}\;%
    \lIf{$Bool=$ \False}{\Return{\False}}%
    $\D\gets \D\cup (q,m,P,N)$\;%
    Update the lists $V,\Lambda$ thanks to formula \eqref{eq:update}\;%
    $(Q,\tau)\gets$ \Primitive{}$(Q,P)$\label{Pirr:prim}\;%
    $\Lambda\gets \tau(\Lambda)$\label{Pirr:updatelambda}\;%
  }%
  \Return $(\D,Q)$\;%
\end{algorithm}

\begin{prop}\label{prop:fastIrr}
  Running \PIrr{}$(F)$ returns the correct
  output. If $F$ is Weierstrass, it takes at most $\Ot(\vF)$
  operations over $\Ki$. If $F$ is monic, it takes $\Ot(\vF+\dy)$
  operations, assuming a slight change of line \ref{Pirr:prim} and a bivariate representation $\Ki_g=\Ki_{P_1,Q}$ (see the proof below).
\end{prop}

\begin{proof}
  The polynomial $\bar H$ at line \ref{Pirr:degen} is the correct lower \edgepoly{} thanks to Lemma \ref{lem:sigmaF} (see also Remark \ref{rem:precision}). Then correctness follows from Theorem \ref{thm:Irr} and
  \ref{prop:pseudo}, together with Proposition \ref{prop:sigma}. As
  $q_k\ell_k \ge 2$, we have $g\le \log_2(\dy)$, while recursive calls of line \ref{Pirr:callback}
  multiplies the complexity by at most a logarithm
  too. Considering one iteration, lines \ref{Pirr:approx},
  \ref{Pirr:edge}, \ref{Pirr:degen}, \ref{Pirr:prim} and
  \ref{Pirr:updatelambda} cost respectively $\Ot(\vF)$,
  $\Ot(\vF+f_k^2)$, $\Ot(f_k N_k)\subset \Ot(\dy)$,
  $\Ot(f_k^{(\omega+1)/2})$ and $\Ot(f_k)$ from respectively
  Propositions \ref{prop:approx} (used with precision
  $\eta(F)\leq 2\,\vF/\dy$), \ref{prop:computeHbark},
  \ref{prop:pseudodeg}, \ref{prop:Prim} and \ref{prop:Prim} once
  again.  Summing up, we get a total cost $\Ot(\vF+\dy+f^2)$. Lemma
  \ref{lem:sigmaF} gives $\eta(F)\ge \dy m_1/q_1$. If $F$ is
  Weierstrass, we have $m_1>0$. Hence $\eta(F)\ge \dy/q_1 \ge f$ from
  which it follows that $\dy f\le 2\vF$. Therefore, both $d$ and
  $f^2\le fd$ belong to $\O(\vF)$, proving Proposition
  \ref{prop:fastIrr} for Weierstrass polynomials. If $F$ is monic and
  not Weierstrass, then $m_1=0$ and the inequality $\dy f\le 2\vF$ doesn't hold anymore. We thus modify the algorithm as follows: we do not
  compute primitive elements of $\Ki_k$ over the field $\Ki$ but only
  over the next residue ring $\Ki_1=\Ki_{P_1}$. We thus get a
  representation $\Ki_k=\Ki[Z_1,Z_2]/(P_1(Z_1),Q_k(Z_1,Z_2))$ for all
  $k\ge 2$, with $Q_k\in \Ki_1[Z_2]$ square-free of degree
  $f_k/\ell_1$. Given $P_{k+1}$, we compute then $Q_{k+1}$ such that
  $(P_1,Q_{k+1})=(P_1,Q_k,P_k)$, thus dealing with at most trivariate
  triangular sets. Propositions \ref{prop:Prim} and
  \ref{prop:pseudodeg} need to be adapted slightly: the base field
  $\Ki$ has to be replaced by the ring $\Ki_1=\Ki_{P_1}$. This is
  possible thanks to \cite[Propositions 14 and 15]{PoWe17}, computing
  now $Q_k$ with smaller complexity
  $\O(d_{P_1}(d_{Q_{k-1}} d_{P_k})^{(\omega+1)/2})\subset \O(\ell_1
  (f_k/\ell_1)^{(\omega+1)/2})$
  and still checking pseudo-degeneracy of $\bar{H}_k$ with
  $\Ot(d_{P_1}d_{Q_{k}} N_k)=\O(\dy)$.  As $m_2>0$, Lemma
  \ref{lem:sigmaF} gives
  $\eta(F) \ge N_1 m_2/q_2 = \dy/\ell_1 q_1 q_2 \ge f_k/\ell_1$ from
  which it follows that
  $\ell_1(f_k/\ell_1)^2 = \ell_1 f_k^2\le 2f_k\eta(F)\le 2\vF$.  The
  all complexity of this slightly modified algorithm becomes
  $\Ot(\vF+\dy)$, as required.
\end{proof}

\paragraph{Proof of Theorem \ref{thm:main}.} 
If $F$ is monic, then it is irreducible in
$\Ki[[x]][y]$ if and only if it is pseudo-irreducible and $\Ki_g$ is a
field (Corollary \ref{cor:pseudoIrrvsIrr}). If $F$ is Weierstrass, we have $\Ki_g=\Ki[Z]/(Q(Z))$ and we check if $\Ki_g$ is a field  with a univariate irreducibility test in $\Ki[Z]$ of degree
$\deg(Q)=f\le \dy$ (cost $\I(\dy)$). If $F$ is monic but not Weierstrass, we have $\Ki_g=\Ki_{P_1}[Z]/(Q(Z))$  and $\Ki_g$ is a field if and only if $P_1$ and $Q$ are irreducible. This cost $\I(\ell_1)\subset \I(\dy)$ operations in $\Ki$ for $P_1$ and $\I(f/\ell_1)$ operations in $\Ki_1$ for $Q$ (assuming $P_1$ irreducible), that is $\Ot(\I(\dy))$ operations over $\Ki$ (use assumption $d\I(n)\le \I(dn)$). If $F$ is not monic
and its leading coefficient has valuation $0$, we simply invert it. Otherwise, either its Newton polygon has more than one slope and $F$ is reducible, either $F(0)$ has valuation $0$ and we can run the algorithm on
the reciprocal polynomial of $F$. We are thus done from Proposition
\ref{prop:fastIrr}. $\hfill\square$

\paragraph{Proof of Theorem \ref{thm:absolute}.} The polynomial $F$ is
absolutely irreducible if and only if it is pseudo-irreducible and
$f_g=1$. We thus apply algorithm \PIrr{}, except that we return
\False{} if we find out that $\ell_k>1$. We thus have $\Ki_k=\Ki$ for
all $k$, and we need not to deal with the Las-Vegas subroutine
\Primitive{}, nor with univariate irreducibility tests. We obtain a
deterministic algorithm running with $\Ot(\vF+d)$ operations over
$\Ki$, which is $\Ot(\vF)$ is $F$ is Weierstrass (or more generally if
$F(0,y)$ has a unique root: in such a case, we have $\ell_1=1$ which
implies $q_1>1$ and $m_1>0$ so that the inequality $f_g\,d=d\le 2\vF$
holds). Note that the non-monic case is handled in the same way than in
the proof of Theorem \ref{thm:main}. Also, we could have used
algorithm \AbhyankarMoh{} with suitable precisions for the same
cost. $\hfill\square$

\paragraph{Proof of Theorem \ref{thm:nonsqrfree}.} If the input
$F\in \Ki[x,y]$ is monic in $y$ with partial degrees $\dx:=\deg_x(F)$
and $\dy=\deg_y(F)$, we run algorithm \PIrr{} with parameters $F$ and
$4\,\dx$, except that we return \False{} whenever the Newton polygon test of line \ref{Pirr:callback} fails. If $F$ is square-free, we have the
well known inequality $\vF\le 2\dx \dy$ so that $\eta(F)\le 4 \dx$:
the algorithm will return the correct answer with at most
$\Ot(\dx \dy)$ operations over $\Ki$ as required, and so without performing recursive calls of line \ref{Pirr:callback} (see Remark \ref{rem:precision} once again). If $F$ is not
square-free, then $\bar{H}_k$ is never square-free. Hence, we will
never reach the case $N_k=1$, and either the pseudo-degenarcy test of line \ref{Pirr:degen} or the Newton polygon test of line \ref{Pirr:callback} will fail at some point during the algorithm, in which case we return \False{} within $\Ot(\dx \dy)$ as required. If $F$ is not
monic, the result is similar, applying the same strategy as in the
proof of Theorems \ref{thm:main} and
\ref{thm:absolute}.
 $\hfill\square$

\begin{xmp}\label{xmp:splitting}
  Let us illustrate algorithm \PIrr{} on a simple example. Consider
  $F=(Y^4-X^2)^4+Y^6X^{11}-Y^4X^{12}-Y^2X^{13} +X^{14}+X^{16}\in
  \Qi[X,Y]$.

  \emph{Initialisation.} We have $N_0=\dy=16$, and we let $\psi_{-1}=X$,
  $V=(1,0)$ and $\lambda=(1,1)$.

  \emph{Step $0$.} The $16^{th}$-approximate roots of $F$ is
  $\psi_0=Y$ and we find $\bar{H}_0=(Y^4-X^2)^4$, meaning that $H_0$
  is pseudo-degenerated with edge data
  $(q_1,m_1,P_1,N_1)=(2,1,Z_1^2-1,4)$. Accordingly to
  \eqref{eq:update}, we update $V=(2,1,4)$ and
  $\lambda=(z_1,z_1,4z_1)$, with $z_1=Z_1\mod P_1$ (i.e.
  $z_1^2=1$). Note that $Z_1$ is coprime to $P_1$ so that
  $\lambda_{1,1}=4\,z_1^3=4\,z_1$ is invertible in $\Qi_1$ as
  predicted by the proof of Proposition \ref{prop:pseudobordel}.

  \emph{Step $1$.} As $N_1=4$, we compute the $4^{th}$-approximate
  root of $F$, getting $\psi_1=Y^4-X^2$. $F$ has
  $\Psi=(\psi_{-1},\psi_0,\psi_1)$-adic expansion
  $F=\psi_1^4 +\psi_{-1}^{11}\psi_0^2\psi_1
  -\psi_{-1}^{12}\psi_1+\psi_{-1}^{16}$.
  All involved monomials reach the minimal values \eqref{eq:wj}, and
  we deduce from \eqref{eq:barHk} and equality $z_1^2=1 $ that
  $\bar{H}_1=Y^4+\frac{(z_1-1)}{(4z_1)^3}
  X^{12}Y+\frac{1}{(4z_1)^4}X^{16}$.
  Here, $\bar{H}_1$ is quasi-homogeneous with slope
  $(q_2,m_2)=(1,4)$. As $4\,z_1$ is a unit of $\Qi_1$, $H_1$
  is  pseudo-degenerated if and only if the univariate polynomial
  $Q(Z_2)=Z_2^4+(z_1-1) Z_2+1$ is the power of a square-free
  polynomial $P_2$ in $\Qi_1[Z_2]$. To check this, we apply
  the euclidean algorithm to compute the gcd between $Q$ and its
  derivative $Q'$. The first euclidean division gives
  $Q=\frac{Z_2}{4}\,Q'+R$ with $R=\frac 3 4\,(z_1-1)\,Z_2+1$. As $Z_1-1$
  divides $P_1$, the leading coefficient of the remainder $R$ is a
  zero divisor in $\Qi_1$. Hence, performing the next euclidean
  division of $Q'$ by $R$ requires to split $Q$ accordingly to the decomposition of the current residue ring
  $\Qi_1=\Qi_{1,1}\oplus \Qi_{1,2}$ induced by the factorization
  $P_1=(Z_1-1)(Z_1+1)$ discovered so far. Then we continue the
  euclidean algorithm in each fields summands.  We find here
  that both reductions $Q_1\in\Qi_{1,1}[Z_2]$ and $Q_2\in\Qi_{1,2}[Z_2]$
  of $Q$ are square-free, from which it follows by definition that $Q$
  is square-free in $\Qi_1[Z_2]$. Hence, $H_1$ is pseudo-degenerated
  with edge data $(q_2,m_2,P_2,N_2)=(1,4,Z_2^4+(z_1-1) Z_2+1,1)$. As
  $N_2=1$, we deduce that $F$ is pseudo-irreducible.

  Note that $F$ is thus balanced (see Section
  \ref{sec:equising}). In particular, it has $\ell_1\ell_2=8$
  irreducible factors in $\algclos{\Qi}[[X]][Y]$ all with same
  ramification index $q_1q_2=2$ and whose characteristic exponents and
  intersection multiplicities can be deduced from the edges data
  thanks to Theorem \ref{prop:pseudo} and formula \eqref{eq:BkMk}. If
  we want furthermore to compute the number of irreducible factors in
  $\Qi[[X]][Y]$ together with their residual degrees, there only
  remains to compute the decomposition of the last residue ring
  $\Qi_2$ into fields summand. We find here the field decomposition:
  \[
  \Qi_2\simeq\frac{\Qi[Z_1,Z_2]}{(Z_1-1,Z_2^4+1)}\oplus
  \frac{\Qi[Z_1,Z_2]}{(Z_1+1,Z_2-1)}\oplus
  \frac{\Qi[Z_1,Z_2]}{(Z_1+1,Z_2^3+Z_2^2+Z_2-1)}.
  \]
  It follows that $F$ has three irreducible factors in $\Qi[[X]][Y]$ of
  respective residual degrees $4,1,3$ (which are given together with
  their residue fields) and ramification index $2$. In particular, they
  have respective degrees $8,2,6$.
\end{xmp}

\begin{rem}\label{rem:splitting}
  This example was chosen to illustrate that it might be necessary to perform some
  splittings in course of the involved square-free tests, as mentionned in Remark \ref{rem:sqrfree}. In case of pseudo-degeneracy, the splittings
  recombine thanks to the Chinese Remainder Theorem, and we 
  pursue the algorithm over a single residue field: in the previous
  example, if we would have found $Q_1=P_{2,1}^{N_2}$ and
  $Q_2=P_{2,2}^{N_2}$ with $P_{2,i}\in \Qi_{1,i}[Z_2]$ square-free and
  with \emph{the same} exponent $N_2>1$ for $i=1,2$, we would have
  continue the algorithm over the single current residue ring
  $\Qi_2=\Qi_1[Z_2]/(P_2)$, with $P_2\in\Qi_1[Z_2]$ square-free, recovered from its
  reductions $P_{2,1}$ and $P_{2,2}$. 
\end{rem}

The reader will find more examples in Subsection \ref{ssec:examples}.


\section{Pseudo-irreducible means balanced}\label{sec:equising}

We show in this section that a polynomial is pseudo-irreducible if and only if  its absolutely irreducible factors are equisingular and have same sets of pairwise intersection multiplicities (balanced polynomials). In this case, we give explicit formulas for the characteristic exponents and the intersection multiplicities in terms of the edges data. This leads us to the proof of Theorem \ref{thm:main2}. 

\subsection{Balanced polynomials} 

\paragraph{Characteristic exponents.} Let $F\in \algclos{\Ki}[[x]][y]$
be an irreducible polynomial of degree $e$ satisfying $F(0,0)=0$. We still assume that $\Char(\Ki)$ is zero or greater than $e$ and we let $(T^e,\sum a_i T^i)$ be the Puiseux parametrization of 
the germ of plane curve $(F,0)$ (or branch) defined by $F$. The
\textit{characteristic exponents} of $F$ are those exponents $i$ for
which a non trivial factor of the ramification index is
discovered. Namely, they are defined as
\[
  \beta_0=e,\quad \beta_k=\min\left(i\,\,\text{ s.t.} \,\, a_i\ne 0,\,
  gcd(\beta_0,\ldots,\beta_{k-1}) \not| i \right), \quad k=1,\ldots,G,
\]
where $G$ is the least integer for which
$gcd(\beta_0,\ldots,\beta_{G})=1$ (characteristic exponents are sometimes refered to the rational numbers $\beta_i/e$ in the litterature).  It is wellknown that the data
\[
  \C(F)=(\beta_0;\beta_1,\ldots,\beta_G)
\]
determines the \emph{equisingularity class} of the germ $(F,0)$. The equisingular equivalence of two germs of plane curves was developped by Zariski in \cite{Za65}. There are several equivalent definitions, a usual one being in terms of the multiplicity sequences of the infinitely near points of the singularity.  This notion is particularly important as it agrees with the topological class when $\Ki=\Ci$ (see e.g. \cite{Wa04}). Conversely, two equisingular germs of curves which are not tangent to the $x$-axis have same characteristic exponents \cite[Corollary 5.5.4]{Ca00}. If
tangency occurs, we rather need to consider ``generic characteristic
exponents'', which form a complete set of
equisingular (hence topological if $\Ki=\Ci$) invariants. The set $\C(F)$ and the set of generic characteristic
exponents determine each others assuming that we are given $\beta_0$ (contact order with $x$-axis)
\cite[Proposition 4.3]{Pop02} or \cite[Corollary 5.6.2]{Ca00}. Note that a data equivalent to $\C(F)$ is given
by the list of intersection multiplicities of $F$ with its
characteristic approximate roots $\psi_{-1},\psi_0,\ldots,\psi_g$
\cite[Cor. 5.8.5 and 5.9.11]{Ca00}, or equivalently with its
characteristic minimal polynomials $\phi_{-1},\ldots,\phi_g$ (use
e.g. Proposition \ref{prop:approx} and Lemma \ref{lem:vLambda}, or see
\cite{Pop02}).

More generally, if $F\in \algclos{\Ki}[[x]][y]$ is irreducible, it
defines a unique germ of irreducible curve on the line $x=0$, with
center $(0,c)$, $c\in \algclos{\Ki}\cup \{\infty\}$. We define then the
characteristic exponents of $F$ as those of the
shifted polynomial $F(x,y+c)$ if $c\in \algclos{\Ki}$ or of the
reciprocal polynomial $\tilde{F}=y^{\dy} F(x,y^{-1})$ if $c=\infty$.

\paragraph{Intersection sets.}
If we want to determine the equisingularity class of a
reducible polynomial $(F,0)$, we need to consider also the intersection
multiplicities between the branches of $F$. The intersection
multiplicity between two coprime polynomials $G,H\in \Ki[[x]][y]$ is
defined as
\[
  (G,H)_0:=\val(\Res_y(G,H))=\dim_{\algclos{\Ki}}\frac{\algclos{\Ki}[[x]][y]}{(G,H)},
\]
the right hand equality following from classical properties of the
resultant. The intersection multiplicity is zero if and only if $G$
and $H$ do not have branches with same center. Suppose that $F$ has
(distinct) irreducible factors
$F_1,\ldots,F_{f}\in \algclos{\Ki}[[x]][y]$. We introduce \textit{the
  intersection sets} of $F$, defined for $i=1,\ldots,f$ as
\[
  \Gamma_i(F):=\big((F_i,F_j)_0,\, 1\le j \le f, j\ne i\big).
\]
By convention, we take into account repetitions, $\Gamma_i(F)$ being
considered as an unordered list with cardinality $f-1$.  If $F$ is
Weierstrass, the equisingular class (hence the topological class if $\Ki=\Ci$) of the
germ $(F,0)$ is uniquely determined by
the characteristic exponents and the intersections sets of the
branches of $F$ \cite{Za86}. Note that the set $\C(F_i)$ only depends on $F_i$ while $\Gamma_i(F)$ depends on $F$.

 \paragraph{Balanced polynomials.} Theorem \ref{thm:main2} asserts  that in some ``balanced'' situation, we can compute in quasi-linear time characteristic exponents and intersection sets of
 some reducible polynomials.

\begin{dfn}\label{def:balanced}
  We say that $F$ is \emph{balanced} if $\C(F_i)=\C(F_j)$ and
  $\Gamma_i(F)=\Gamma_j(F)$ for all $i,j$. In such a case, we denote
  simply these sets by $\C(F)$ and $\Gamma(F)$.
\end{dfn}

Thus, if $F$ is a balanced Weierstrass polynomial, its absolutely irreducible factors are equisingular and have same sets of pairwise intersection multiplicities, and the converse holds if no branch is tangent to the $x$-axis or all branches are tangent to the $x$-axis.

\begin{xmp}\label{rem:deg}
  Let us illustrate this definition with some  basic examples. Note that the second and third  examples show in particular that no condition implies the other in  Definition \ref{def:balanced}.

$\bullet$ If $F\in \Ki[[x]][y]$ is irreducible, a Galois argument shows that it is balanced (follows from Theorem \ref{prop:pseudo} below). The converse doesn't hold:
  $F=(y-x)(y+x^2)$ is reducible, but it is
  balanced. This example also shows that balancedness does not imply
  straightness of the Newton polygon.
  
$\bullet$ The polynomial $F=(y^2-x^3)(y^2+x^3)(y^2+x^3+x^4)$ is not balanced. It has $3$ absolutely irreducible factors with same sets of characteristic exponents $\C(F_i)=(2;3)$ for all $i$, but $\Gamma_1(F)=(6,6)$ while $\Gamma_2(F)=\Gamma_3(F)=(6,8)$. 

$\bullet$ The polynomial $F=(y-x-x^2)(y-x+x^2)(y^2-x^3)$ is not balanced. It has $3$ absolutely irreducible factors with same sets of pairwise intersection multiplicities $\Gamma_i(F)=(2,2)$, but $\C(F_1)=\C(F_2)=(1)$ while $\C(F_3)=(2;3)$. 

$\bullet$ The polynomial $F=(y-2x^2)^2-8yx^5-2x^8$ has four
  irreducible factors in $\algclos{\Qi}[[x]][y]$, namely
$F_1=y-\sqrt{2}x-\sqrt[4]{2}x^2$, $F_2=y-\sqrt{2}x+\sqrt[4]{2}x^2$, $F_3=y+\sqrt{2}x-i\sqrt[4]{2}x^2$ and $F_4=y-\sqrt{2}x+i\sqrt[4]{2}x^2$.
We have $C(F_i)=(1)$ and $\Gamma_i(F)=(1,1,2)$ for all $i$ so $F$ is
balanced. Note that this example shows that balancedness does not imply that all factors intersect each others with the same multiplicity. 

$\bullet$ The polynomial $F=(y^2-x^3)(y^3-x^2)$ is not balanced. However, it defines two equisingular germs of plane curve (one is tangent to the $x$-axis while the other is not). 
\end{xmp}

%

\paragraph{Noether-Merle's Formula.} If $F,G\in \algclos{\Ki}[[x]][y]$
are two irreducible Weierstrass polynomials of respective degrees
$e_F$ and $e_G$, their intersection multiplicty $(F,G)_0$ at the
origin is closely related to the characteristic exponents
$(\beta_0,\ldots,\beta_G)$ of $F$. Let us denote by
\[
  \Cont(F,G):=e_F\,\max_{\substack{F(y_F)=0\\G(y_G)=0}}\val(y_F-y_G)
\]
the \textit{contact order} of the branches $F$ and $G$. Then
Noether-Merle's \cite{Me77} formula states
\begin{equation}\label{eq:NMformula}
  (F,G)_0=\frac{e_G}{e_F}\left(\sum_{k\le K} (E_{k-1}-E_{k})\beta_k+E_K \Cont(F,G)\right),
\end{equation}
where $E_k:=gcd(\beta_0,\ldots,\beta_k)$ and $K=\max(k\,|\,\beta_k\le \Cont(F,G))$. A proof can be found in
\cite[Proposition 6.5]{Pop02} (or references therein), where a formula is
given in terms of the semi-group generators, which turns out to be
equivalent to \eqref{eq:NMformula} thanks to \cite[Proposition 4.2]{Pop02}.

\paragraph{$PGL_2(\Ki)$-invariance of characteristic exponents and
  intersection sets.} As we will consider also non monic polynomials, we will need the following lemma in order to reduce to the monic case.

\begin{lem}\label{lem:invariance}
  Let $a,b,c,d\in \Ki$ such that $ad-bc\ne 0$. Then $F$ and
  $\tilde{F}:=(cy+d)^{\dy}F\left(\frac{ay+b}{cy+d}\right)$ have same
  number of irreducible factors in $\Ki[[x]][y]$, same numbers of
  irreducible factors in $\algclos{\Ki}[[x]][y]$, same sets of
  characteristic exponents and same intersection sets. In particular,
  $F$ is balanced if and only if $\tilde{F}$ is.
\end{lem}

\begin{proof}
  As $\tilde{F}$ is obtained after a projective change of coordinate
  over $\Ki$, it's clear that $F$ and $\tilde{F}$ have same number of
  factors over any given field extension of $\Ki$. If $F$ has
  irreducible factors $F_1,\ldots,F_{f}$ over $\algclos{\Ki}$, then
  $\tilde{F}$ has irreducible factors
  $\tilde{F}_i=(cy+d)^{\dy(F_i)}F_i((ay+b)/(cy+d))$,
  $i=1,\ldots,f$. It follows immediately that
  $\C(F_i)=\C(\tilde{F}_i)$ as the characteristic exponents are
  computed after translation to $y=0$. We have also
  $\Gamma_i(F)=\Gamma_i(\tilde{F})$ as the $x$-valuation of the
  resultant is invariant under projective change of the $y$ coordinate
  (see e.g. \cite[Chapter 12]{GeKaZe94}).
\end{proof}

\subsection{Balanced is equivalent to pseudo-irreducible}\label{ssec:pseudo=balanced}

\paragraph*{Notations and main results.} Let $F\in \Ki[[x]][y]$ be a
monic and square-free polynomial. As usual we let
$(q_1,m_1,P_1,N_1),\ldots,(q_g,m_g,P_g,N_g)$ its edge data computed by
algorithm \PIrr{}. We denote $e=e_g=q_1\cdots q_g$
and let $\hat{e}_{k}=e/e_k$.  We define $f=f_g=\ell_1\cdots \ell_g$
and $\hat{f}_{k}=f/f_k$ in the analoguous way, where as usual
$\ell_k=\deg(P_k)$. For all $k=1,\ldots,g$, we define
\begin{equation}\label{eq:BkMk}
  B_k=m_1 \hat{e}_1+\cdots +m_k\hat{e}_k
  \quad \textrm{and} \quad
  M_k=m_1 \hat{e}_0\hat{e_1}+\cdots+ m_k\hat{e}_{k-1}\hat{e}_k
\end{equation}
and we let $B_0=e$. They are positive integers related by the formula
\begin{equation}\label{eq:Mk}
  M_k=\sum_{i=1}^k(\hat{e}_{i-1}-\hat{e}_i)B_i +\hat{e}_k B_k.
\end{equation}
Note that $0\le B_1\le\cdots \le B_g$ and $B_0\le B_g$. We have  $B_1>0$ if and only if $m_1 >0$,
or equivalently, $F$ is Weierstrass. In such a case, the inequality $B_0\le B_1$ is equivalent to that $q_1\le m_1$, meaning that the germ  $(F,0)$ is not tangent to the $x$-axis. We check easily that
$\hat{e}_k=gcd(B_0,\ldots,B_k)$. In particular,
$gcd(B_0,\ldots,B_g)=1$. 

\begin{thm}\label{prop:pseudo}
  A monic polynomial $F\in \Ki[[x]][y]$ is balanced if and
  only if it is pseudo-irreducible. In such a case, it has $f$
  irreducible factors in $\algclos{\Ki}[[x]][y]$ of degree
  $e$ and
  \begin{enumerate}
  \item $\C(F)=(B_0;B_k \, |\, q_k >1)$ 
  \item $\Gamma(F)=(M_k \, |\, \ell_k >1)$, where $M_k$ appears
    $\hat{f}_{k-1}-\hat{f}_{k}$ times.
  \end{enumerate}
\end{thm}

Note that taking into account repetitions, the intersection set has
cardinality $\sum_{k=1}^g (\hat{f}_{k-1}-\hat{f}_{k})=f-1$, as
required. Of course, it is empty if and only if $f=1$, that is if $F$
is irreducible in $\algclos{\Ki}[[x]][y]$.

\begin{cor}\label{cor:disc}
  Let $F\in \Ki[[x]][y]$  balanced and monic. Then,
  the discriminant of $F$ has valuation
  \[
    \vF= f \Big(\sum_{\ell_k > 1}
    (\hat{f}_{k-1}-\hat{f}_{k})M_k+\sum_{q_k >1}
    (\hat{e}_{k-1}-\hat{e}_{k})B_k\Big)
  \]
  and the discriminants of the absolutely irreducible factors of $F$
  all have same valuation
  $\bar{\vF}=\sum_{q_k >1} (\hat{e}_{k-1}-\hat{e}_{k})B_k$.
\end{cor}

\begin{proof} \textit{(of Corollary \ref{cor:disc})} Suppose that $F$
  is balanced. Then it has $f$ irreducible factors
  $F_1,\ldots,F_f$ of same degree $e$, with discriminant valuations
  say $\vF_1,\ldots,\vF_f$. The multiplicative property of the
  discriminant gives the well-known formula
  \begin{equation}\label{eq:disc}
    \vF=\sum_{1\le i \le f} \vF_i + \sum_{1\le i\ne j\le f}(F_i,F_j)_0.
  \end{equation}
  Let $y_1,\ldots,y_e$ be the roots of $F_i$. Thanks to
  \cite[Proposition 4.1.3 (ii)]{Wa04} combined with point 1 of Theorem
  \ref{prop:pseudo}, we deduce that for each fixed $a=1,\ldots,e$, the
  list $(\val(y_a-y_b),b\ne a)$ consists of the values $B_k/e$
  repeated $\hat{e}_{k-1}-\hat{e}_{k}$ times for $k=1,\ldots,g$. Since
  $\vF_i=\sum_{1\le a \ne b \le e} \val(y_a-y_b)$, we deduce that
  $\bar{\vF}:=\vF_1=\cdots=\vF_f$ satisfies the claimed formula. The
  formula for $\vF$ then follows straightforwardly from
  \eqref{eq:disc} combined with point 2 in Theorem \ref{prop:pseudo}.
\end{proof}


%

The proof of Theorem \ref{prop:pseudo} requires some intermediate
results. We begin by investigating the relations between
(pseudo)-rational Puiseux expansions and the Puiseux series of $F$.

\paragraph{Structure of the pseudo-rational Puiseux expansion.}

Contrarly to \cite{PoWe17} where dynamic evaluation is also
considered, we do not necessarily ``split'' the Newton-Puiseux type
algorithm when we meet several edges. However, we show that this has
no impact for our purpose and that algorithm \PNPA{} still
allow to recover all the Puiseux series of a pseudo-irreducible
polynomial. To this aim, we need to study in more details the
so-called pseudo-rational Puiseux expansion (pseudo-RPE for short)
\[
  (\mu_k T^{e_k},S_k(T)):=\pi_k(T,0)
\]
computed when running algorithm \PNPA{}. As an induction
argument is used, we need some further notations.

\textit{Exponents data.} For all $0\le i\le k \le g$, we define
$Q_{k,i}=q_{i+1}\cdots q_k$ with convention $Q_{k,k}=1$ and let
\begin{equation}\label{eq:Bki}
  B_{k,i}=m_{1} Q_{k,1}+\cdots+m_i Q_{k,i}
\end{equation}
with convention $B_{k,0}=0$. We have $Q_{k+1,i}=Q_{k,i} q_{k+1}$ and
$B_{k+1,i}=q_{k+1} B_{k,i}$ for all $i\le k$ and
$B_{k+1,k+1}=q_{k+1} B_{k,k} + m_{k+1}$.

\noindent
\textit{Coefficients data.} For all $0\le i\le k \le g$, we define
$ \mu_{k,i}:=z_{i+1}^{t_{i+1}Q_{i,i}}\cdots z_{k}^{t_{k}Q_{k-1,i}} $
with convention $\mu_{k,k}=1$ and let
\begin{equation}\label{eq:muki}
  \alpha_{k,i}:=\mu_{k,1}^{m_1}\cdots \mu_{k,i}^{m_i},
\end{equation}
with conventions $\alpha_{k,0}=1$.  We have
$\mu_{k+1,i}=\mu_{k,i}z_{k+1}^{t_{k+1} Q_{k,i}}$ and
$\alpha_{k+1,i}=\alpha_{k,i}z_{k+1}^{t_{k+1}B_{k,i}}$ for all
$1\le i\le k$, and $\alpha_{k+1,k+1}=\alpha_{k+1,k}$.

\begin{rem}\label{rem:zerodiv}
  Note that $\mu_{k,0}$ is invertible in the product of fields
  $\Ki_k$. Namely, if $z_i\notin \Ki_i^\times$, then $P_i(0)$ is a
  zero divisor and we must have $q_i=1$ by definition of
  pseudo-degeneracy (see Remark \ref{rem:P}). In such a case, we have
  $t_{i}=0$ and $z_i$ does not appear as a factor of $\mu_{k,0}$.
\end{rem}

\begin{lem}\label{lem:pik} Let $z_0=0$. For all $k=0,\ldots,g$, we
  have the formula
  \[
    \pi_k(x,y)=\Bigg(\mu_{k,0} x^{Q_{k,0}}, \sum_{i=0}^{k} \alpha_{k,i}
    x^{B_{k,i}}\Big(z_i^{s_i}+c_i\big(\mu_{k,i}
    x^{Q_{k,i}}\big)\Big)+\alpha_{k,k} x^{B_{k,k}} y\Bigg).
  \]
\end{lem}

\begin{proof}
  This is correct for $k=0$: the formula becomes
  $\pi_0(x,y)=(x,y+c_0(x))$. For $k>0$, we conclude by induction,
  using the the relations \eqref{eq:Bki} and \eqref{eq:muki} above
  with definition
  $\pi_{k}(x,y) = \pi_{k-1}(z_{k}^{t_k} x^{q_k} , x^{m_k}
  (z_k^{s_k}+c_{k}(x)+y))$.
\end{proof}

Given $\alpha$ an element of a ring $\Li$, we denote by $\alpha^{1/e}$
the residue class of $Z$ in $\Li[Z]/(Z^e-\alpha)$. By Remark
\ref{rem:zerodiv} we know that $\mu_{k,0}\in \Ki_k$ is invertible for
all $k=0,\ldots,g$ and we introduce the ring extension
\[
  \Li_k:=\Ki_k[\theta_k]=\Ki[z_1,\ldots,z_k][\theta_k],\quad
  \textrm{where}\quad
  \theta_k:=\big(\mu_{k,0}^{-\hat{e}_k}\big)^{\frac{1}{e}}.
\]
Note that $\Li_0=\Ki$ and we check straightforwardly from the
definition that $\theta_{k}\in \Li_{k+1}$. In particular, we have a
natural strict inclusion $\Li_k\subset \Li_{k+1}$.

\begin{prop}\label{prop:CPE}
  Let $F\in \Ki[[x]][y]$ be Weierstrass and $\tilde{S}=S(\mu^{-1/e}T)$
  with $(\mu\, T^e,S(T)):=\pi_g(T,0)$. We have
  \[
    \tilde{S}(T)=\sum_{B > 0} a_B T^{B}\in \Li_g[[T]],
  \]
  where $gcd(B_0,\ldots,B_k)|B$ and $a_B\in \Li_k$ for all
  $B< B_{k+1}$ (with convention $B_{g+1}:=+\infty$). Moreover, we have
  for all $1\le k \le g$
  \begin{equation}\label{eq:aBk}
    a_{B_k}=\begin{cases} \varepsilon_{k}  (z_k\theta_{k-1}^{m_k})^{\frac{1}{q_k}} \quad \qquad \text{if}\quad  q_k>1 \\
      \varepsilon_{k} z_k\theta_{k-1}^{m_k}  +\rho_k \qquad \,\, \,\text{if} \quad q_k=1 
    \end{cases}
  \end{equation}
  where $\varepsilon_{k}\in\Li_{k-1}$ is invertible and
  $\rho_{k}\in \Li_{k-1}$. In particular
  $a_{B_k}\in \Li_{k}\setminus \Li_{k-1}$.
\end{prop}

\begin{proof} Note first that we have $\mu=\mu_{g,0}$ thanks to Lemma
  \ref{lem:pik}, so that $\tilde{S}(T)=S(\theta_g T)$ lies in
  $\Li_g[[T]]$ as required. Thanks to definitions \eqref{eq:Bki} and
  \eqref{eq:muki}, we compute
  \begin{equation}\label{eq:alpha}
    \mu_{g,k}\mu_{g,0}^{-\hat{e}_{k}/e}=\theta_k\in \Li_k\quad\textrm{and}\quad \alpha_{g,k}\mu_{g,0}^{-B_{k}/e}=\prod_{j=1}^k \Big(\mu_{g,j}\mu_{g,0}^{-\hat{e}_{j}/e}\Big)^{m_j} =\prod_{j=1}^k \theta_j^{m_j}\in\Li_k.
  \end{equation}
  Combined with Lemma \ref{lem:pik} applied to rank $k=g$, we deduce
  \begin{equation}\label{eq:CPE}
    \tilde{S}(T)=\sum_{k=0}^{g} U_k(\theta_k T^{\hat{e}_k}) T^{B_{k}},\qquad U_k(T):=\big(z_k^{s_k} +c_k(T)\big)\prod_{j=1}^k \theta_j^{m_j}\in \Li_k[[T]].
  \end{equation}
  As $\hat{e}_k=gcd(B_0,\ldots,B_k)$ divides both $\hat{e}_i$ and
  $B_i$ for all $i\le k$, this forces $gcd(B_0,\ldots,B_k)$ to divide
  $B$ for all $B< B_{k+1}$. In the same way, as $\Li_i\subset \Li_k$
  for all $i\le k$, we get $a_B\in \Li_k$ for all $B< B_{k+1}$. There
  remains to show the formula for $a_{B_k}$ for $k\ge 1$. As
  $c_k(0)=0$, we deduce that
  \begin{equation}\label{eq:Uk0}
    U_k(0)=z_k^{s_k}\prod_{j=1}^k \theta_j^{m_j}=(z_k\theta_{k-1}^{m_k})^{\frac{1}{q_k}} \prod_{j=1}^{k-1} \theta_j^{m_j},
  \end{equation}
  the second equality using the B\'ezout relation $s_k q_k - t_k m_k=1$.
  Note that
  $\varepsilon_k:=\prod_{j=1}^{k-1} \theta_j^{m_j}\in \Li_{k-1}$ is
  invertible by Remark \ref{rem:zerodiv}. In particular,
  $(z_k\theta_{k-1}^{m_k})^{\frac{1}{q_k}}\in \Li_k$ (although $z_k^{1/q_k}$ might not belong to $\Li_k$).  Let $\rho_k$ be the sum of the contribution
  of the terms $T^{B_i}U_i$ to the coefficient of the monomial
  $T^{B_k}$. So $a_{B_k}=U_k(0)+\rho_k$. As $B_1\le\cdots\le B_g$ and
  $k\ge 1$, we deduce that if $U_iT^{B_i}$ contributes to $T^{B_k}$,
  then $i<k$ so that $U_iT^{B_i}\in
  \Li_{k-1}[[T^{\hat{e}_{k-1}}]]$.
  We deduce that $\rho_k \in \Li_{k-1}$. Moreover, $\rho_k\ne 0$
  forces $\hat{e}_{k-1}$ to divide $B_k$. By definition
  \eqref{eq:BkMk} of $B_k$, and using that $m_{k}$ is coprime to
  $q_{k}$, we must have $q_k=1$, as required.
\end{proof}

\begin{rem}\label{rmk:Puiseux}
  While algorithm \PNPA{} allows to compute the all parametrization
  $\sum_B a_B T^B$ (up to some truncation bound), algorithm \PIrr{}
  precisely allows to compute the monomials
  $(a_{B_k}-\rho_k) T^{B_k}$, $k=0,\ldots,g$  (using \eqref{eq:aBk} and
  explicit formula of $\varepsilon_k$ in terms of edges data). As the
  remaining part of this section shows, this is precisely the minimal
  information required for testing balancedness. For instance, the
  Puiseux series of $F=(y-x-x^2)^2-2x^4$ are $S_1=T+T^2(-\sqrt{2}+1)$
  and $S_2=T+T^2(\sqrt{2}+1)$. While algorithm \PNPA{} allows to
  compute $S_1$ and $S_2$, algorithm \PIrr{} will compute only the
  ``essential monomials'' $-\sqrt{2} T^2$ and $\sqrt{2} T^2$ with
  approximate roots. Computing the singular part of the Puiseux series
  of a (pseudo)-irreducible polynomial in quasi-linear time remains an
  open challenge (see Section \ref{sec:conc} for some hints towards
  such a result).
\end{rem}

For all $\zeta\in W$, we denote by $\theta_g(\zeta)$ a $e^{th}$-roots
of $\mu(\zeta)^{-1}=\mu_{g,0}(\zeta)^{-1}$. Such a choice induces a
natural evaluation map
\[
  ev_{\zeta}:\Li_g=\Ki[z_1,\ldots,z_g][\theta_g]\to
  \Ki[\zeta_1,\ldots,\zeta_g][\theta_g(\zeta)]\subset \algclos{\Ki}
\]
and we denote for short $a(\zeta)\in \algclos{\Ki}$ the evaluation of
$a\in \Li_g$ at $\zeta$.


Let $\zeta'\in W$. By construction, when $\theta_g(\zeta')$ runs over
the $q^{th}$-roots of $\mu(\zeta')^{-1}$, then $\theta_k(\zeta')$ runs
over the $e_k=e/\hat{e}_k$ roots of
$\mu_{k,0}(\zeta_1',\ldots,\zeta_k')$. Hence it is always possible to
choose $\theta_g(\zeta')$ in such a way that
\begin{equation}\label{eq:hyp}
  (\zeta_1,\ldots,\zeta_k)=(\zeta'_1,\ldots,\zeta'_k)\quad \Longrightarrow \quad  \theta_k(\zeta)=\theta_k(\zeta'),
\end{equation}
We assume this from now. In such a case, we have $a(\zeta)=a(\zeta')$
for all $a\in \Li_k$. The following lemma is crucial for our
purpose. 

\begin{lem}\label{lem:abzeta} Let us fix $\omega$ such that $\omega^e=1$ and let
  $\zeta,\zeta'\in W$. For all $k=0,\ldots,g$, the following
  assertions are equivalent:
  \begin{enumerate}
  \item $a_{B}(\zeta)=a_{B}(\zeta')\omega^{B}$ for all $B\le B_k$.
  \item $a_{B}(\zeta)=a_{B}(\zeta')\omega^{B}$ for all $B< B_{k+1}$.
  \item $(\zeta_1,\ldots,\zeta_k)=(\zeta_1',\ldots,\zeta_k')$ and
    $\omega^{\hat{e}_k}=1$.
  \end{enumerate}
\end{lem}

\begin{proof}
  By Proposition \ref{prop:CPE}, we have $a_B\in \Li_k$ and
  $\hat{e}_k|B$ for all $B<B_{k+1}$ from which we deduce
  $3)\Rightarrow 2)$ thanks to hypothesis (\ref{eq:hyp}). As
  $2)\Rightarrow 1)$ is obvious, we need to show $1)\Rightarrow 3)$.
  We show it by induction. If $k=0$, the claim follows immediately
  since $\hat{e}_0=q$. Suppose $1)\Rightarrow 3)$ holds true at rank
  $k-1$ for some $k\ge 1$. Let us denote by
  $\zeta_k^{1/q_k}:=ev_{\zeta}(z_k^{1/q_k})$.  If
  $a_{B}(\zeta)=a_{B}(\zeta')\omega^{B}$ for all $B\le B_k$, then this
  holds true for all $B\le B_{k-1}$.  As
  $\varepsilon_k,\rho_k\in\Li_{k-1}$, the induction hypothesis
  combined with (\ref{eq:hyp}) gives
  $\varepsilon_k(\zeta)=\varepsilon_k(\zeta')\ne 0$ and
  $\rho_k(\zeta)=\rho_k(\zeta')$. Let us use now $a_{B_k}(\zeta)=a_{B_k}(\zeta')\omega^{B_k}$. Two cases occur:
  \begin{itemize}
  \item If $q_k>1$, we deduce from (\ref{eq:aBk}) that 
  $\left(\zeta_k\theta_{k-1}^{m_k}(\zeta)\right)^{\frac{1}{q_k}}=\left(\zeta_k'\theta_{k-1}^{m_k}(\zeta')\right)^{\frac{1}{q_k}}\omega^{B_k}.$
    Raising to the power $q_k$, and using that $\hat{e}_{k-1}|q_k B_k$ forces 
    $\omega^{q_k B_k}=1$ by induction hypothesis, we deduce that $\zeta_k\theta_{k-1}^{m_k}(\zeta)=\zeta_k'\theta_{k-1}^{m_k}(\zeta')$. As $\theta_{k-1}\in\Li_{k-1}^\times$, we get $\zeta_k=\zeta_{k}'$ thanks again to the induction hypothesis. Furthermore, as $\zeta_k=\zeta'_k$ implies $a_{B_k}(\zeta)=a_{B_k}(\zeta')$ thanks to (\ref{eq:hyp}), we have also $\omega^{B_k}=1$.
  \item If $q_k=1$, we deduce from (\ref{eq:aBk}) that
    $\zeta_k\theta_{k-1}^{m_k}(\zeta)+\rho_k(\zeta)=\omega^{B_k}(\zeta_k'\theta_{k-1}^{m_k}(\zeta')+\rho_k(\zeta'))$. As
    $q_k=1$ implies $\hat{e}_{k-1}=\hat{e_k}|B_k$ and $\rho_k\in \Li_{k-1}$, $\theta_{k-1}\in\Li_{k-1}^\times$, induction hypothesis gives again
    $\omega^{B_k}=1$ and $\zeta_k=\zeta'_k$.
  \end{itemize}
  To  conclude, use that $B_k=\sum_{s\le k} m_s \hat{e}_s$, so that induction hypothesis gives
  $(\omega^{\hat{e}_k})^{m_k}=1$. Since $m_k$ is coprime to $q_k$ and
  $(\omega^{\hat{e}_k})^{q_k}=\omega^{\hat{e}_{k-1}}=1$, this forces
  $\omega^{\hat{e}_k}=1$.
\end{proof}

In particular, Lemma \ref{lem:abzeta} above implies that
algorithm \PNPA{} still allow to recover all the Puiseux
series of a pseudo-irreducible polynomial, as required. 

\begin{cor}\label{cor:branches}
  Suppose that $F$ is pseudo-irreducible and Weierstrass. Then $F$
  admits exactly $f$ distinct monic irreducible factors
  $F_{\zeta}\in \algclos{\Ki}[[x]][y]$ indexed by $\zeta\in W$.  Each
  factor $F_{\zeta}$ has degree $e$ and defines a branch with
  classical Puiseux parametrizations $(T^e,\tilde{S}_{\zeta}(T))$
  where
  \begin{equation}\label{eq:SetParam}
    \tilde{S}_{\zeta}(T)=\sum_B a_{B}(\zeta)T^B.
  \end{equation} 
  The $e$ Puiseux series of $F_{\zeta}$ are given by
  $\tilde{S}_{\zeta}(\omega x^{\frac{1}{e}})$ where $\omega$ runs over
  the $e^{th}$-roots of unity and this set of Puiseux series does not
  depend of the choice of the $e^{th}$-roots $\theta_g(\zeta)$.
\end{cor}

\begin{proof}
  As $F$ is pseudo-irreducible, $H_g=y$ (Weierstrass polynomial of
  degree $N_g=1$ with no terms of degree $N_g-1$),thus
  $\pi_g^* F(x,0) = 0$. We deduce $F(T^e,\tilde{S}_{\zeta}(T))=0$ for
  all $\zeta\in W$.  By \eqref{eq:aBk}, we have $a_{B_k}(\zeta)\ne 0$
  for all $k$ such that $q_k>1$. Since
  $gcd(B_0=e,B_k \,|\, q_k>1))=gcd(B_0,\ldots,B_g)=\hat{e}_g=1$, the
  parametrization $(T^e,\tilde{S}_{\zeta}(T))$ is primitive, that is
  the greatest common divisor of the exponents of the series $T^e$ and
  $\tilde{S}_{\zeta}(T)$ equals one. Hence, this parametrization
  defines a branch $F_{\zeta}=0$, where
  $F_{\zeta}\in \algclos{\Ki}[[x]][y]$ is an irreducible monic factor
  of $F$ of degree $e$. Thanks to Lemma \ref{lem:abzeta}, these $f$
  branches are distinct when $\zeta$ runs over $W$. As $\deg(F)=e\,f$, we
  obtain in such a way all irreducible factors of $F$. The last claim
  follows straightforwardly.
\end{proof}

\paragraph{Pseudo-irreducible implies balanced.}  This is the easiest implication. Let us first consider the characteristic exponents. We get:

\begin{prop}\label{cor:charac}
  Let $F\in \Ki[[x]][y]$ be pseudo-irreducible. Then each branch
  $F_{\zeta}$ of $F$ has characteristic exponents $(B_0;B_k \, |\, q_k >1), \,k=1,\ldots,g)$.
\end{prop}

\begin{proof} Thanks to Corollary \ref{cor:branches}, all polynomials
  $F_{\zeta}$ have same first characteristic exponent $B_0=e$. Using
  again that $a_{B_k}(\zeta)\ne 0$ for all $k\ge 1$ such that $q_k>1$
  (by \eqref{eq:aBk}), it follows immediately from Proposition
  \ref{prop:CPE} and Corollary \ref{cor:branches} that the remaining
  characteristic exponents of $F_{\zeta}$ are those $B_k$ for which
  $k\ge 1$ and $q_k>1$.
\end{proof}

Concerning the intersection multiplicities, we get:
\begin{prop}\label{prop:series}
  Let $F\in \Ki[[x]][y]$ be pseudo-irreducible with at
  least two branches $F_{\zeta}, F_{\zeta'}$. We have
  \[
    (F_{\zeta},F_{\zeta'})_0=M_{\kappa},\quad \kappa:=\min
    \big(k=1,\ldots,g \, |\, \zeta_k \ne \zeta'_k\big).
  \]
  and this value is reached exactly
  $\hat{f}_{\kappa-1}-\hat{f}_{\kappa}$ times when $\zeta'$ runs over
  the set $W\setminus\{\zeta\}$.
\end{prop}

\begin{proof} Noether-Merle's formula \eqref{eq:NMformula} combined
  with Proposition \ref{cor:charac} gives
  \begin{equation}\label{eq:merle}
    (F_{\zeta},F_{\zeta'})_0=\sum_{k \le K} 
    (\hat{e}_{k-1}-\hat{e}_k) B_k +  \hat{e}_K \Cont(F_{\zeta},F_{\zeta'})
  \end{equation}
  with $K = \max\{k\,|\,\Cont(F_\zeta,F_{\zeta'}) \geq B_k\}$. Note
  that the $B_k$'s which are not characteristic exponents do not
  appear in the first summand of formula (\ref{eq:merle}) ($q_k=1$
  implies $\hat{e}_{k-1}-\hat{e}_k=0$). It is a classical fact that we
  can fix any root $y$ of $F$ for computing the contact order (see e.g. \cite[Lemma
  1.2.3]{Ga95}). Combined with Corollary \ref{cor:branches}, we obtain
  the formula
  \begin{equation}\label{eq:cont}
    \Cont(F_{\zeta},F_{\zeta'})=\max_{\omega^e=1}\left(v_T\left(\tilde{S}_{\zeta}(T)
    -\tilde{S}_{\zeta'}(\omega T)\right)\right).
  \end{equation}
  We deduce from Lemma \ref{lem:abzeta} that
  \[
  v_T\left(\tilde{S}_{\zeta}(T)-\tilde{S}_{\zeta'}(\omega
    T)\right)=B_{\bar\kappa},\quad \bar\kappa:=\min \left\{k=1,\ldots,g
    \,\, |\,\, \zeta_k \ne \zeta'_k \,\, \textrm{or}\,\,
    \omega^{\hat{e}_k}\ne 1\right\}.
  \]
  As $\omega=1$ satisfies $\omega^{\hat{e}_k}=1$ for all $k$, we
  deduce from the last equality that the maximal value in
  \eqref{eq:cont} is reached for $\omega=1$ (it might be reached for
  other values of $\omega$). It follows that
  $\Cont(F_{\zeta},F_{\zeta'})=B_{\kappa}$ with
  $\kappa=\min \left\{k \, |\, \zeta_k \ne \zeta'_k\right\}$. We thus
  have $K=\kappa$ and \eqref{eq:merle} gives
  $(F_{\zeta},F_{\zeta'})_0=\sum_{k=1}^{\kappa}(\hat{e}_{k-1}-\hat{e}_k)
  B_k + \hat{e}_{\kappa} B_{\kappa}=M_{\kappa}$,
  the last equality by \eqref{eq:Mk}. Let us fix $\zeta$. As said
  above, we may choose $\omega=1$ in \eqref{eq:cont}. We have
  $v_T(\tilde{S}_{\zeta}(T)-\tilde{S}_{\zeta'}(T))=B_{\kappa}$ if and
  only if $\zeta_k'=\zeta_k$ for $k<\kappa$ and
  $\zeta_{\kappa}\ne \zeta'_{\kappa}$. This concludes, as the number
  of possible such values of $\zeta'$ is precisely
  $\hat{f}_{\kappa-1}-\hat{f}_{\kappa}$.
  \end{proof}

If $F$ is pseudo-irreducible, then it is balanced and satisfies
both items of Theorem \ref{prop:pseudo} thanks to Proposition
\ref{cor:charac} and Proposition \ref{prop:series}. There remains to
show the converse.

\paragraph{Balanced implies pseudo-irreducible.} 
We need to show that then $N_g=1$ if $F$ is balanced. We denote more
simply $H:=H_g\in \Ki_g[[x]][y]$, and $\pi_g(T,0)=(\mu T^e,S(T))$. We
denote $H_{\zeta}, S_{\zeta},\mu_{\zeta}$ the images of $H, S,\mu$
after applying (coefficient wise) the evaluation map
$ev_{\zeta}:\Ki_g\to \algclos{\Ki}$.

\begin{lem}\label{lem:samedeg}
  Suppose that $F$ is balanced. Then all irreducible factors of
  all $H_{\zeta}$, $\zeta\in W$ have same degree.
\end{lem}

\begin{proof}
  Let $\zeta\in W$ and let $y_\zeta$ be a roots of $H_{\zeta}$. As
  $H_{\zeta}$ divides $(\pi_{g}^* F)_{\zeta}=\pi_{g,\zeta}^* F$ by
  \eqref{eq:pikHk}, we deduce from Lemma \ref{lem:pik} (use
  $B_{gg}=B_g$) that
  \[
    F(\mu_{\zeta} x^e,S_{\zeta}(x)+x^{B_g} y_\zeta(x))=0.
  \]
  Hence,
  $y_0(x):=\tilde{S}_{\zeta}(x^{\frac{1}{e}})+\mu_{\zeta}^{-\frac{B_g}{e}}x^{\frac{B_{g}}{e}}y_\zeta(\mu_{\zeta}^{-\frac{1}{e}}x^{\frac{1}{e}})$
  is a root of $F$ and we have moreover the equality
  \begin{equation}\label{eq:y}
    \deg_{\algclos{\Ki}((x))}(y_0)=e\deg_{\algclos{\Ki}((x))}(y_\zeta),
  \end{equation}
  where we consider here the degrees of $y_0$ and $y_{\zeta}$ seen as algebraic elements over the field $\algclos{\Ki}((x))$.
  As $F$ is balanced, all its
  irreducible factors - hence all its roots - have same
  degree. Combined with \eqref{eq:y}, this implies that all roots -
  hence all irreducible factors - of all $H_{\zeta}$, $\zeta\in W$
  have same degree.
\end{proof}

\begin{cor}\label{cor:Hg0}
  Suppose $F$ balanced and $N_g>1$. Then there exist some coprime
  positive integers $(q,m)$ and $Q\in \Ki_g[Z]$ monic with non zero
  constant term such that $H$ has lower \edgepoly{}
  \[
    \bar{H}(x,y)=Q\left(y^{q}/x^{m}\right)x^{m \deg(Q)}
  .\]
\end{cor}

\begin{proof}
  As $N_g>1$, the Weierstrass polynomial $H=H_{g}$ is not
  pseudo-degenerated and admits a lower slope $(q,m)$ (we can not have
  $H_g=y^{N_g}$ as $F$ would not be square-free). Hence, its lower
  \edgepoly{} may be written in a unique way
  \begin{equation}\label{eq:Hg0}
    \bar{H}(x,y)=y^r \tilde{Q}\left(y^{q}/x^{m}\right)x^{m \deg(\tilde{Q})}
  \end{equation}
  for some non constant monic polynomial $\tilde{Q}\in \Ki_g[Z]$ with
  non zero constant term and some integer $r\ge 0$.  Let $\zeta\in W$
  such that $\tilde{Q}_{\zeta}(0)\ne 0$ and suppose $r>0$. By applying
  the evaluation map $ev_{\zeta}$ to \eqref{eq:Hg0}, we deduce that
  the Newton polygon of $H_{\zeta}$ has a vertice of type $(r,i)$,
  $0\le r\le \dy$ and the Newton-Puiseux algorithm (over a field)
  implies that $H_{\zeta}$ admits two factors $A, B$ such that
  $\deg(A)=r$ and $\deg(B)=q\deg(\tilde{Q})$. By Lemma
  \ref{lem:samedeg}, this forces $q$ to divide $r$. Hence $r=nq$ for
  some $n\in \Ni$ and the claim follows by taking
  $Q(Z)=Z^n\tilde{Q}(Z)$.
\end{proof}

\begin{lem}\label{lem:Gparam}
  Suppose $F$ balanced and $N_g>1$. We keep
    notations $q$ and $Q$ from Corollary \ref{cor:Hg0}. Let
  $G\in \algclos{\Ki}[[x]][y]$ be an irreducible monic factor of
  $F$. Then $e\,q$ divides $n:=\deg(G)$ and there exists a unique
  $\zeta\in W$ and a unique root $\alpha$ of $Q_{\zeta}$ such that $G$ admits a parametrization
  $(T^n,S_G(T))$, where
  \begin{equation}\label{eq:SG}
    S_G(T)\equiv \tilde{S}_{\zeta}(T^{\frac{n}{e}}) + \alpha^{\frac{1}{q}}\mu_{\zeta}^{-\frac{B_g}{e}}T^{\frac{n}{e}B_g+\frac{nm}{eq}} \mod \,T^{\frac{n}{e}B_g+\frac{nm}{eq}+1},
  \end{equation}
  with $\alpha^{1/q}$ an arbitrary $q^{th}$-roots of $\alpha$ (we may
  \textit{a priori} have $\alpha=0$). Conversely, given $\zeta\in W$
  and $\alpha$ a roots of $Q_{\zeta}$, there exists at least one
  irreducible factor $G$ for which \eqref{eq:SG} holds.
\end{lem}

\begin{proof}
  Let $y_{\zeta}^{(i)}$, $i=1,\ldots,N_g$ be the roots of
  $H_{\zeta}$. Following the proof of Lemma \ref{lem:samedeg}, we know
  that each roots $y_{\zeta}^{(i)}$ gives rise to a family of $e$-roots of $F$
  \[
    y_{\zeta,\omega}^{(i)}:=\tilde{S}_{\zeta}(\omega
    x^{\frac{1}{e}})+\omega\mu_{\zeta}^{-\frac{B_g}{e}}
    x^{\frac{B_{g}}{e}}y_{\zeta}^{(i)}(\omega\mu_{\zeta}^{-\frac{1}{e}}x^{\frac{1}{e}}),
  \]
  where $\omega$ runs over the $e^{th}$ roots of unity.  As
  $H_{\zeta}$ has distinct roots and
  $\tilde{S}_{\zeta}(\omega x^{1/e})\ne \tilde{S}_{\zeta'}(\omega'
  x^{1/e})$ when $(\zeta,\omega)\ne (\zeta',\omega')$ (Lemma
  \ref{lem:abzeta}), we deduce that the $efN_g=\deg(F)$ Puiseux series
  $y_{\zeta,\omega}^{(i)}$ are distinct, getting all roots
  of $F$. As $e$ divides
  $n:=\deg_{\algclos{\Ki}((x))}(y_{\zeta,\omega}^{(i)})$ (use
  \eqref{eq:y}), the roots $y_{\zeta,\omega}^{(i)}$, $\omega^e=1$
  belong to the same orbit of the Galois group of the field extension
  $\algclos{\Ki}((x))\to \algclos{\Ki}((x^{1/n}))$. Thus, any
  irreducible factor $G$ of $F$ has degree $n$ and admits a root of
  type $y_{\zeta,1}^{(i)}$ for some pair $(\zeta,i)$. Hence $G$ admits
  a parametrization $(T^n,S_G(T))$, where
  $S_G(T):=y_{\zeta,1}^{(i)}(T^n)$. Since
  $y_{\zeta}^{(i)}(x)= \alpha^{1/q} x^{m/q}+h.o.t$ for some uniquely
  determined roots $\alpha$ of $Q_{\zeta}$ (use Corollary
  \ref{cor:Hg0}), we get the claimed formula. Since there exists at
  least one root $\alpha\ne 0$ of $Q_{\zeta}$, the fact that
  $S_G\in \algclos{\Ki}[[T]]$ forces $nm/eq\in \Ni$. Hence $eq$
  divides $n$ since $e$ divides $n$ and $q$ and $m$ are
  coprime. Conversely, if $\zeta\in W$ and $Q_{\zeta}(\alpha)=0$,
  there exists at least one root $y_{\zeta}^{(i)}$ of $H_{\zeta}$
  such that $y_{\zeta}^{(i)}(x)= \alpha^{1/q} x^{m/q}+h.o.t$ and by
  the same arguments as above, there exists at least one irreducible
  factor $G$ such that \eqref{eq:SG} holds.
\end{proof}

For any irreducible factor $G$ of $F$, we denote by
$(\zeta(G),\alpha(G))\in W\times \algclos{\Ki}$ the unique pair
$(\zeta,\alpha)$ such that \eqref{eq:SG} holds.

\begin{cor}\label{cor:expcarac}
  Suppose $F$ balanced and $N_g>1$. Let $n$ stands for the degree
  of any of its irreducible factor and let $q$ as in Lemma \ref{lem:Gparam}. Then the lists of the
  characteristic exponents of the irreducible factors of $F$ all begin
  as $\{n\}\cup \{\frac{n}{e}B_k,q_k>1, k=1,\ldots,g\}.$ Moreover the
  next characteristic exponent of any factor $G$ is greater or equal
  than $\frac{n}{e}B_g+\frac{nm}{eq}\in \Ni$, with equality if and
  only if $q>1$ and $\alpha(G)\ne 0$.
\end{cor}

\begin{proof}
  This follows straightforwardly from Lemma \ref{lem:Gparam} combined
  with Proposition \ref{prop:CPE} (similar argument than for
  Proposition \ref{cor:charac}).
\end{proof}


\begin{cor}\label{cor:interset}
  Suppose $F$ balanced with $N_g>1$ and with $\rho\geq 2$
  irreducible factors $G_1,\ldots,G_{\rho}\in \algclos{\Ki}[[x]][y]$. We
  have
  \[
    (G_i,G_j)_0 > \frac{n^2}{e^2}\left(M_g+\frac{m}{q}\right) \quad
    \Longleftrightarrow\quad
    (\zeta(G_i),\alpha(G_i))=(\zeta(G_j),\alpha(G_j)).
  \]
\end{cor}

\begin{proof}
  Using similar arguments than Proposition \ref{prop:series}, we get
  $\Cont(G_i,G_j)=v_T(S_{G_i}-S_{G_j})$ and we deduce from
  \eqref{eq:SG} and Lemma \ref{lem:abzeta} that
  $\Cont(G_i,G_j) >\frac{n}{e}B_g+\frac{nm}{eq}$ if and only if
  $\zeta(G_i)=\zeta(G_j)$ and $\alpha(G_i)=\alpha(G_j)$ . The claim
  then follows from Noether-Merle's formula \eqref{eq:NMformula}
  combined with Corollary \ref{cor:expcarac}.
\end{proof}

\begin{prop}\label{prop:equisingpseudoirr}
  If $F$ is balanced, then it is pseudo-irreducible.
\end{prop}

\begin{proof} We need to show that $N_g=1$. Suppose on the contrary
  that $N_g>1$. Let $\zeta\in W$ and let $G_i$ such that
  $\zeta(G_i)=\zeta$. Thanks to Lemma \ref{lem:Gparam}, we deduce from
  algorithm \NPA{} (over a field) that $\pi_{g,\zeta}^*(G_i)$ has an
  \edgepoly{} of shape $(y^q-\alpha(G_i)x^m)^{N(G_i)}$ where
  $eq N(G_i)=\deg(G_i)=n$. In particular, $N(G_i)=n/eq$ is constant
  for all $i=1,\ldots,\rho$. We deduce that
  $\bar{H}_{\zeta}=\prod_{i
    |\zeta(G_i)=\zeta}(y^q-\alpha(G_i)x^m)^{N(G_i)}$, hence
  \begin{equation}\label{eq:factoQzeta}
    Q_{\zeta}(Z)=\prod_{i |\zeta(G_i)=\zeta}(Z-\alpha(G_i))^{N(G_i)}.
  \end{equation}
  Let $\alpha$ be a root of $Q_{\zeta}$ and $j$ such that
  $(\zeta(G_j),\alpha(G_j))=(\zeta,\alpha)$. Denote
  $I_j:=\{i\ne j\,|\,
  (\zeta(G_i),\alpha(G_i))=(\zeta(G_j),\alpha(G_j))\}$.
  Thanks to \eqref{eq:factoQzeta}, we deduce that the root $\alpha$
  has multiplicity $N(G_j)+\sum_{i\in I_j}
  N(G_i)=(\Card(I_j)+1)n/eq$.
  As $F$ is balanced, all factors have same intersection sets and
  Corollary \ref{cor:interset} implies that all sets $I_j$ have same
  cardinality. It follows that all roots $\alpha$ of all polynomials
  $Q_{\zeta}$ have same multiplicity. In other words, $Q$ is the power
  of some square-free polynomial $P\in \Ki_g[Z]$. If $q=1$, this
  implies that $H=H_g$ is pseudo-degenerate (Definition
  \ref{def:pseudodeg}), contradicting $N_g>1$. If $q>1$, we need to
  show moreover that $P$ has invertible constant term. Since there
  exists at least one non zero root $\alpha$ of some $Q_{\zeta}$
  (Corollary \ref{cor:Hg0}), we deduce from Corollary
  \ref{cor:expcarac} that at least one factor $G_i$ has next
  characteristic exponent $\frac{n}{e}B_g +\frac{nm}{eq}$ (use
  $q>1$). As $F$ is balanced, it follows that all $G_i$'s have
  next characteristic exponent $\frac{n}{e}B_g +\frac{nm}{eq}$, which
  by Corollary \ref{cor:expcarac} forces all $\alpha(G_i)$ - thus all
  roots $\alpha$ of all $Q_{\zeta}$ by last statement of Lemma
  \ref{lem:Gparam} - to be non zero. Thus $P$ has invertible constant
  term and $H=H_g$ is pseudo-degenerate (Definition
  \ref{def:pseudodeg}), contradicting $N_g>1$. Hence $N_g=1$ and $F$
  is pseudo-irreducible.
\end{proof}

The proof of Theorem \ref{prop:pseudo} is complete.  $\hfill{\square}$

\subsection{Proof of Theorem \ref{thm:main2}.} If $F$ is monic, the
result follows immediately from Proposition \ref{prop:fastIrr} and
from Theorem \ref{prop:pseudo} (see Section
\ref{sec:equising}). Namely, $F$ is balanced if and only if it is
pseudo-irreducible and in such a case, the edges data allow to compute
characteristic exponents and pairwise intersection multiplicities as
well as the discriminant valuation $\vF$ (Corollary
\ref{cor:disc}). If $F=c(x) y^{\dy}+\cdots$ has invertible leading
coefficient $c\in \Ki[[x]]$, $c(0)\ne 0$, we invert $c$ after line $1$
up to precision $\eta\le \eta(F)$, for a cost
$\Ot(\eta)\subset \Ot(\vF)$, thus reducing in the aimed bound to the
monic case.  If the leading coefficient of $F$ is not invertible, we
can find $z\in \Ki$ such that $F(0,z)\ne 0$ with at most $\dy$
evaluation of $F(0,y)$ at $z=0,1,\ldots,\dy$ (use here that $\Ki$ has
at least $\dy$ elements). This costs at most $\Ot(\dy)$ using fast
multipoint evaluation \cite[Corollary 10.8]{GaGe13}. One such a $z$ is
found, we can apply previous strategy to the polynomial
$\tilde{F}:=y^{\dy}F\left(\frac{zy+1}{y}\right)\in \Ki[[x]][y]$ which
has by construction an invertible coefficient. We have
$\deg(F)=\deg(\tilde{F})$ and $\vF(F)=\vF(\tilde{F})$ \footnote{This
  equality explains why we consider the valuation of the resultant
  between $F$ and $F_y$ as main complexity indicator instead of the
  valuation of the discriminant which may vary under projective change
  of coordinates.} so the complexity remains the same. Moreover, Lemma
\ref{lem:invariance} shows that $F$ is balanced if and only if
$\tilde{F}$ is, and there is a one-to-one correspondance between the
irreducible factors of $F$ and $\tilde{F}$ in $\bar{\Ki}[[x]][y]$ such
that $F_i$ and $\tilde{F}_i$ have same characteristic exponents and
same sets of intersection multiplicities. Hence we are reduced to the
monic case, and so within the aimed complexity. Theorem
\ref{thm:main2} is proved.    $\hfill\square$

\subsection{Some examples}\label{ssec:examples}

\begin{xmp}[\textbf{balanced}]\label{ex1}
  Let
  $F=y^6-3x^3 y^4-2x^2 y^4+3x^6 y^2+x^4 y^2-x^9+2x^8-x^7\in
  \Qi[x,y]$.
  This small example is constructed in such a way that $F$ has $3$
  irreducible factors $(y-x)^2-x^3$, $(y+x)^2-x^3$, $y^2-x^3$ and we
  can check that $F$ is balanced, with $e=2$, $f=3$ and
  $\C(F_i)=(2;3)$ and $\Gamma_i(F)=(4,4)$ for all $i=1,2,3$. Let us
  recover this with algorithm \PIrr{}.

  \emph{Initialise.} We have $N_0=\dy=6$, and we let $\psi_{-1}=x$,
  $V=(1,0)$ and $\Lambda=(1,1)$.

  \emph{Step $0$.} The $6^{th}$-approximate roots of $F$ is $\psi_0=y$
  and we deduce that
  $\bar{H}_0=y^6-2x^2 y^4+x^4 y^2=(y(y^2-x^2))^2.$ Thus, $H_0$ is
  pseudo-degenerated with edge data
  $(q_1,m_1,P_1,N_1)=(1,1,Z_1^3-Z_1,2)$.  Accordingly to
  \eqref{eq:update}, we update $V=(1,1,1)$ and
  $\Lambda=(1,z_1,3z_1^2-1)$. Note that the Newton polygon $\NP$
  of $F$ is not straight. In particular, $P_1$ is reducible over $\Qi$
  and $F$ is reducible over $\Qi[[x]][y]$.

  \emph{Step $1$.} The $2^{th}$-approximate root of $F$ is
  $\psi_1=y^3-\frac32 x^3y-x^2y$ and $F$ has $\Psi$-adic expansion
  $F=\psi_1^2 -3\psi_0^2 \psi_{-1}^5+\frac34 \psi_0^2
  \psi_{-1}^6-\psi_{-1}^7+2\psi_{-1}^8- \psi_{-1}^9$.
  The monomials reaching the minimal values \eqref{eq:wj} are
  $\psi_1^2$ (for $j=2$) and $-3\psi_0^2 \psi_{-1}^5$ and
  $\psi_{-1}^7$ (for $j=0$). We deduce from \eqref{eq:barHk} that
  $\bar{H}_1=y^2-\alpha x$, where
  $\alpha=(3z_1^2+1)/(3z_1^2-1)^2$ is easily seen to be invertible in
  $\Qi_1$ (in practice, we compute $P\in \Qi[Z_1]$ such that
  $\alpha=P\mod P_1$ and we check $gcd(P_1,P)=1$). We deduce that
  $H_1$ is pseudo-degenerated with edges data
  $(q_2,m_2,P_2,N_2)=(2,1,Z_2-\alpha,1)$.  As $N_2=1$, we deduce that
  $F$ is balanced with $g=2$.

  \emph{Conclusion.} We deduce from Theorem \ref{prop:pseudo} that $F$
  has $f=\ell_1\ell_2=3$ irreducible factors over $\algclos{\Ki}[[x]][y]$
  of same degrees $e=q_1 q_2=2$.  Thanks to \eqref{eq:BkMk}, we
  compute $B_0=e=2$, $B_1=2$, $B_2=3$ and $M_1=4$, $M_2=6$.  We deduce
  that all factors of $F$ have same characteristic exponents
  $\C(F_i)=(B_0;B_2)=(2;3)$ and same intersection sets
  $\Gamma_i(F)=(M_1,M_1)=(4,4)$ (i.e. $M_1$ which appears
  $\hat{f}_0-\hat{f}_1=3-1$ times), as
  required. 
\end{xmp}

\begin{xmp}[\textbf{non balanced}]\label{ex2}
  Let
  $F={{y}^{6}}-{{x}^{6}} {{y}^{4}}-2 {{x}^{4}} {{y}^{4}}-2 {{x}^{2}}
  {{y}^{4}}+2 {{x}^{10}} {{y}^{2}}+3 {{x}^{8}} {{y}^{2}}-2 {{x}^{6}}
  {{y}^{2}}+{{x}^{4}} {{y}^{2}}-{{x}^{14}}+2 {{x}^{12}}-{{x}^{10}}\in
  \Qi[x,y]$.
  This second small example is constructed in such a way that $F$ has
  $6$ irreducible factors $y+x-x^2$, $y+x-x^2$, $y-x-x^2$, $y-x+x^2$,
  $y-x^3$ and $y+x^3$ and we check that $F$ is not balanced, as
  $\Gamma_i(F)=(1,1,1,1,2)$ for $i=1,\ldots,4$ while with
  $\Gamma_i(F)=(1,1,1,1,3)$ for $i=5,6$. Let us recover this with
  algorithm \PIrr{}.

  \emph{Initialise.} We have $N_0=\dy=6$, and we let $\psi_{-1}=x$,
  $V=(1,0)$ and $\Lambda=(1,1)$.

  \emph{Step $0$.} The $6^{th}$-approximate roots of $F$ is $\psi_0=y$
  and we deduce that
  $\bar{H}_0=y^6-2x^2 y^4+x^4 y^2=(y(y^2-x^2))^2.$ Thus, as in
  Example \ref{ex1}, $H_0$ is pseudo-degenerated with edge data
  $(q_1,m_1,P_1,N_1)=(1,1,Z_1^3-Z_1,2)$.  Accordingly to
  \eqref{eq:update}, we update $V=(1,1,1)$ and
  $\Lambda=(1,z_1,3z_1^2-1)$.

  \emph{Step $1$.} The $2^{th}$-approximate root of $F$ is
  $\psi_1=y^3-yx^2-yx^4-\frac12 yx^{6}$ and $F$ has $\Psi$-adic
  expansion
  $F=\psi_1^2 -\psi_{-1}^{10}+2
  \psi_{-1}^{12}-\psi_{-1}^{14}-4\psi_{-1}^{6}\psi_{0}^{2}+\psi_{-1}^{8}\psi_{0}^{2}+\psi_{-1}^{10}\psi_{0}^{2}-\frac14\psi_{-1}^{12}\psi_{0}^{2}$.
  The monomials reaching the minimal values \eqref{eq:wj} are
  $\psi_1^2$ (for $j=2$) and $-4\psi_{-1}^{6}\psi_{0}^{2}$ (for
  $j=0$). We deduce from \eqref{eq:barHk} that
  $\bar{H}_1=y^2-\alpha x^2$, where
  $\alpha=4z_1^2/(3z_1^2-1)^2$. As $z_1$ is a zero divisor in
  $\Qi_1=\Qi[Z_1]/(Z_1^3-Z_1)$ and $(3z_1^2-1)=P_1'(z_1)$ is invertible in
  $\Qi_1$, we deduce that $\alpha$ is a zero divisor. It follows that
  $\bar{H}_1$ is not the power of a square-free polynomial. Hence
  $H_1$ is not pseudo-degenerated and $F$ is not balanced (with
  $g=1$), as required. In order to desingularise $F$, we would need at
  this stage to split the algorithm accordingly to the discovered
  factorization $P_1=Z_1(Z_1^2-1)$ before continuing the process,
  as described in \cite{PoWe17}. 
\end{xmp}

\begin{xmp}[\textbf{non Weierstrass}]\label{ex:nonweierstrass}
  Let $F=(y+1)^6-3x^3(y+1)^4-2(y+1)^4+3x^6(y+1)^2+(y+1)^2-x^9+2x^6-x^3$. We have
  $F=((y+2)^2-x^3)((y+1)^2-x^3)(y^2-x^3)$ from which we deduce that
  $F$ is balanced with three irreducible factors with
  characteristic exponents $\C(F_i)=(2,3)$ and intersection sets
  $\Gamma_i(F)=(0,0)$. Let us recover
  this with algorithm \PIrr{}.

  \emph{Initialise.} We have $N_0=\dy=6$, and we let $\psi_{-1}=x$,
  $V=(1,0)$ and $\Lambda=(1,1)$.

  \emph{Step $0$.} The $6^{th}$-approximate roots of $F$ is
  $\psi_0=y+1$. We have
  $F=\psi_0^6-3\psi_{-1}^3 \psi_0^4- 2 \psi_0^4 + 3 \psi_{-1}^6
  \psi_0^2 + \psi_0^2 - \psi_{-1}^9 + 2 \psi_0^6 - \psi_{-1}^3$. By
  \eqref{eq:wj}, the monomials involved in the lower edge of $H_0$ are
  $\psi_0^6, -2\psi_0^4, \psi_0^2$. We deduce from \eqref{eq:barHk}
  that $\bar{H}_0=(y^3-y)^2$ so that $H_0$ is pseudo-degenerated with
  edge data $(q_1,m_1,P_1,N_1)=(1,0,Z_1^3-Z_1,2)$. Note that
  $m_1=0$. This is the only step of the algorithm where this may
  occur. Using \eqref{eq:update}, we update $V=(1,0,0)$ and
  $\Lambda=(1,z_1,3z_1^2-1)$.
  
  \emph{Step $1$} The $N_1=2^{th}$ approximate roots of $F$ is
  $\psi_1=(y+1)^3-3/2 x^3 (y+1)-(y+1)$ and $F$ has $\Psi$-adic
  expansion
  $F=\psi_1^2 -\psi_{-1}^3-3\psi_{-1}^3
  \psi_0^2+2\psi_{-1}^6-\psi_{-1}^9+3/4 \psi_{-1}^6\psi_0^2$.  We
  deduce that the monomials reaching the minimal values \eqref{eq:wj}
  are $\psi_1^2$ (for $j=2$) and $-\psi_{-1}^3$,
  $-3\psi_{-1}^3 \psi_0^2$ (for $j=0$). We deduce from
  \eqref{eq:barHk} that $\bar{H}_1=y^2-\alpha x^3$, where
  $\alpha=(\lambda_{1,-1}^3+3\lambda_{1,-1}^3\lambda_{1,0}^2)\lambda_{1,1}^{-2}=(3z_1^2+1)/(3z_1^2-1)^2$
  is easily seen to be invertible in $\Qi_1$. We deduce that $H_1$ is
  pseudo-degenerated with edges data
  $(q_2,m_2,P_2,N_2)=(2,3,Z_2-\alpha,1)$.  As $N_2=1$, we deduce that
  $F$ is balanced with $g=2$. By Theorem \ref{prop:pseudo}
  (assuming only $F$ monic), we get that $F$ has $f=\ell_1\ell_2=3$
  irreducible factors over $\algclos{\Ki}[[x]][y]$ of same degrees
  $e=q_1 q_2=2$.  Thanks to \eqref{eq:BkMk}, we compute $B_0=e=2$,
  $B_1=0$, $B_2=3$ and $M_1=0$, $M_2=6$.  By Theorem
  \ref{prop:pseudo}, we deduce that all factors of $F$ have same
  characteristic exponents $\C(F_i)=(B_0;B_2)=(2;3)$ and same
  intersection sets $\Gamma_i(F)=(M_1,M_1)=(0,0)$ as required.
\end{xmp}

\section{Further comments}
\label{sec:conc}
We conclude this paper with some ongoing work that will deserve
further publication, providing the main perspectives. They are
twofold. We start by discussing a way to factorize the input
polynomial once non irreducibility has been discovered by our main
algorithm. Then we discuss the more general context of polynomials
defined over discrete valuation rings (e.g. $F\in\Qi_p[y]$) and
conclude by an open question concerning the assumption on the base
field we are making in this paper.

\paragraph{Analytic factorization.} Let's assume that $N_g>1$. Then
$\NP_g(F)$ has two distinct edges, or its \edgepoly{} factorizes. In
both cases, if $\val[]$ denotes the extended valuation defined by the
lower edge of $\NP_g(F)$, we get from the \edgepoly{} two polynomials
$G$ and $H$ such that $\val[](F-G\,H)>\val[](F)$. Then, using the classical
Hensel Lemma \cite[Section 15.4]{GaGe13}, we get a quadratic lifting
of $G$ and $H$. As in \cite{CaRoVa16}, we start with euclidean
division, denoting $\psi=\psi_g$ (it is important that $\psi$ is
monic, so that $\val[](\psi)\geq 0$), and \quorem{} the euclidean algorithm.

\begin{lem}
  \label{lem:quorem}
  Let $A$, $B\in\Ki[[x]][y]$ such that $B$ is monic in $\psi$
  (i.e. $B=\psi^b+\cdots$) and $\val[](B)=b\,\val[](\psi)$. Then,
  $Q,R=\quorem(A,B)$ satisfies $\val[](R)\geq \val[](A)$ and
  $\val[](Q)\geq \val[](A)-\val[](B)$.
\end{lem}

\begin{proof}
  We focus on the computation of $R$. First note that it be computed
  as follows\footnote{in practice, we use the classical algorithm of
    $\Ai[y]$, this is only for this proof purpose}: write
  $A=\sum_{i=0}^m a_i\,\psi^i$ the $\psi$-adic expansion of $A$, then
  compute $\At=A-a_m\,B$, and apply recursively this strategy to
  $\At$. As $\deg(\At)<\deg(A)$, this procedure converges towards the
  unique remainder $R$.  We now prove the result by induction on the
  degree of $A$. Nothing has to be done whe $\deg(A)<d$. When
  $\deg(A)=d$, then $A=c\,\psi^b$ with $c\in\Ai$ and result is
  straightforward for $\psi^d$ (we have
  $\val[](B-\psi^d)\geq \val[](B)$ by assumption). Finally, when
  $\deg(A)>d$, apply the above step $\At=A-a_m\,B$. We have
  $\val[](a_m\,B)\geq \val[](A)$ so that $\val[](\At)\geq \val[](A)$,
  and $\deg(\At)<\deg(A)$, which proves the lemma for $R$
  recursively. Result for $Q$ is then a straightforward consequence,
  as $\val[](Q\,B)=\val[](A-R)$.
\end{proof}

From this Lemma, it is trivial to show that the Hensel
lemma, when starting with correct initial
polynomials, ``double the precision'' according to an extended
valuation $(\val[],\psi)$: given $F$, $G$, $H$, $S$, $T\in \Ai[Y]$ with
$H$ monic in $\psi$, and $n\in\Ni^*$ satisfying,
\begin{itemize}
\item $\val[](F-G\,H)\geq \val[](F)+n$
\item $\val[](S\,G+T\,H-1) \geq n$ with $\deg(S)<\deg(H)$,
  $\deg(T)<\deg(G)$, $\val[](S)=-\val[](G)$ and $\val[](T)=-\val[](H)$.
\end{itemize}
it outputs polynomials
$\Gt$, $\Ht$, $\St$, $\Tt\in\Ki[X,Y]$ with $\Ht$ monic in $\psi$ such that:
\begin{itemize}
\item $\val[](F-\Gt\,\Ht)>\val[](F)+2\,n$, with
  $\val[](\Gt-G)\geq n+\val[](G)$ and $\val[](\Ht-H)\geq n+\val[](H)$,
\item $\val[](\St\,\Gt+\Tt\,\Ht-1)>2\,n$ ; $\deg(\Tt)<\deg(\Gt)$,
  $\deg(\St)<\deg(\Ht)$, $\val[](\St)=-\val[](G)$ and $\val[](\Tt)=-\val[](H)$.%
\end{itemize}
We recall the algorithm (this is exactly \cite[Algorithm 15.10]{GaGe13}):
\begin{algorithm}[ht]
  \nonl\TitleOfAlgo{\onestep($F,G,H,S,T$)\label{algo:OneStep}}
  $\alpha\assign{}(F-G\, H)$\;%
  $Q,R\assign{}\quorem(S\,\alpha,H)$\;%
  $\Gt\assign{}G+\alpha\,{}T+Q\, G$\;%
  $\Ht\assign{}H+R$\;%
  $\beta\assign{}(S\,{}\Gt+T\,\Ht)-1$\;%
  $A,B\assign{}\quorem(S\,\beta,\Ht)$\;%
  $\St\assign{}S-B$\;%
  $\Tt\assign{}T-\beta\,{}T-A\,\Gt$\;%
  \Return{$\Ht$, $\Gt$, $\St$, $\Tt$}%
\end{algorithm}

\begin{prop}
  \label{prop:hensel}
  Algorithm \onestep{} is correct.
\end{prop}

\begin{proof}
  We have
  $F-\Gt\,\Ht=(1-S\,G+T\,H)\,\alpha-S\,T\,\alpha^2-(S\,G-T\,H)\,Q\,\alpha+G\,H\,Q^2$.
  By assumption, we have $\val[](\alpha)\geq\val[](F)+n$,
  $\val[](S\,T)=-\val[](F)$, and by Lemma \ref{lem:quorem},
  $\val[](Q)\geq n$. This shows
  $\val[](F-\Gt\,\Ht)>\val[](F)+2\,n$. Similarly,
  $\val[](\Gt-G)=\val[](T\,\alpha+Q\,G)\geq n+\val[](G)$ and
  $\val[](\Ht-H)=\val[](R)\geq n+\val[](H)$. As for monicity of $\Ht$, it is
  a obvious as $\deg(R)<\deg(H)$.

  To conclude, as
  $\St\,\Gt+\Tt\,\Ht-1=\beta\,((\Gt-G)\,S+(\Ht-H)\,T-\beta)$,
  $\val[](\beta)\geq n$ (assumption), $\val[]((\Gt-G)\,S)>n$ and
  $\val[]((\Ht-H)\,T)>n$ (see above), we get
  $\val[](\St\,\Gt+\Tt\,\Ht-1)>2\,n$. As $\val[](B)>\val[](S)$ and
  $\val[](\beta\,T-A\,\Gt)>\val[](T)$, we obsiously have
  $\val[](\St)=-\val[](G)$ and $\val[](\Tt)=-\val[](H)$. Condition
  $\deg(\St)<\deg(\Ht)$ is obvious as $\deg(B)<\deg(\Ht)$. 
\end{proof}

To conclude this paragraph, let us illustrate how to get the initial
$G$, $H$, $S$ and $T$ with $H$ monic on an example.

\begin{xmp}
  \label{xmp:fact-init}
  Consider $F=\psi^3+y^2\,x^3\,\psi+x^6\,y\in\Qi[[x]][y]$ with
  $\psi=y^3-x^2$ and the associated valuation
  $V_1=(3,2,6)$. Considering the lower edge $((1,1),(3,0))$, we get
  $V=(6,4,13)$. Then, we can initialise $G$ and $H$ as respectively
  $\psi^2+y^2\,x^3$ and $\psi$, so that
  $\val[](F-G\,H)=\val[](x^6\,y)=40>39=\val[](F)$. We then use the
  extended euclidean algorithm over $\Qi[Z]$, getting $s=1$ and $t=-Z$
  such that $s\,(Z^2+1)+t\,Z=1$. Then, we can multiply $s$ and $t$ by
  a monomial of valuation $26$, so that $\val[](S)=-\val[](G)$ and
  $\val[](T)=-=\val[](H)$, getting $S=x^{-5}\,y$ and
  $T=-x^{-5}\,y\,\psi$. They indeed satisfy
  $\val[](S\,G+T\,H-1)=\val[](x^{-2}\psi)=1>0$.
\end{xmp}

Finally, not that finding the monomial $x^{-5}\,y$ in the above
example is always possible (see e.g. \cite[Lemma 4.23, page
24]{Ru14}).

Such a factorization done, we can recursively apply our main algorithm
on each factor, factorising again if needed, until we get the full
factorization. This provides an algorithm with complexity
$\Ot(n\,\rho\,\dy)$ to get the $\rho$ factors separable from precision
$n$, improving the bound $\Ot(n\,\dy^2)$ of \cite[Proposition
7]{PoWe17}. Nevertheless, this will not improve the overall complexity
of \cite[Section 7]{PoWe17} (we might need to take $n=\vF$), and the
divide and conquer strategy therein will still have to be used.
Details concerning this algorithm will be presented in a forthcoming
paper.

\paragraph{Working over $\Ai[y]$ and small characteristic.} The
algorithm using approximate roots described in our paper can be
adapted to the case where for instance $F\in\Qi_p[y]$, using $\val[p]$
instead of $\val$ as initial valuation. In such a context, we would
not define the valuations $v_k$ as in Section \ref{ssec:phi-main}, but
as augmented valuations (see \cite{Ma36a,Ma36b,Ru14} or
\cite{Mo99,GuMoNa12}). They would however be computed exactly as
described in Section \ref{subsec:resume} or - equivalently - as in
Definition \ref{def:alternative} of Section
\ref{sec:absolute}. This strategy improves
the computation of \emph{optimal representatives of types}
\cite[Section 3]{GuMoNa11}: the computed approximate root is actually
always optimal (in the sense they described in \cite{GuMoNa11}), and
we do not need the \emph{refinment process} anymore.

However, several points remain to be investigated: is there any need
to use some ``correcting terms'' as we are doing here with the
morphisms $\lambda_k$ ? And how to deal with the case where the
characteristic of the field divides $\dy$ (or more generally is less
than $\dy$ when considering the factorization algorithm described
above), in which case Proposition \ref{prop:approx-root} does not make
sense anymore. These points are being studied by the authors at the
time of the writing, and will be the topic of a further paper.

\paragraph{Computing Puiseux series using approximate roots ?}
As mentioned in Remark \ref{rmk:Puiseux}, our strategy does not
compute all terms of the Puiseux series of the input
polynomial. However, there might be ways to compute them.  We here
comment a few special cases. First note that the coefficients
corresponding to integer exponents are given by the $\dy$-th
approximate root. Then, if $N_1=\dy/2$, we then have
$\psi_1=\psi_0^2+X^{m_1}\,S_1(X)^2$, so that $S_1(X)$ can be computed
via quadratic Newton iteration (this provides all coefficients
corresponding to exponents with denominator $2$). When $q_1>2$, one
can probably compute additional coefficients of the Puiseux series by
solving a linear system. For instance, the case $q_1=3$ can be dealt
with as follows: if
$S_1(x)=x^{\frac 1 3}\,P_1(x)+x^{\frac 2 3}\,P_2(x)$, then we have
$\psi_1=\psi_0^3-3\,x\,P_1\,P_2\,\psi_0-x\,P_1^3-x^2\,P_2^3$, defining
the linear system to solve.


\bibliographystyle{abbrv} {\bibliography{tout}}
\addcontentsline{toc}{section}{References.}

\end{document}